\documentclass[a4paper,reqno]{amsart}

\usepackage[french,english]{babel}
\usepackage{setspace}
\usepackage[totalwidth=14cm,totalheight=22cm]{geometry}
\usepackage[T1]{fontenc}
\usepackage[dvips]{graphicx}
\usepackage{epsfig}
\usepackage{amsmath,amssymb,amscd,amsthm}
\usepackage{subfigure}
\usepackage[latin1]{inputenc}
\usepackage{latexsym}
\usepackage{graphics}
\usepackage{colortbl}
\usepackage{color}
\usepackage{ae}
\usepackage{enumitem}
\usepackage[all]{xy}

\newtheorem{cor}{Corollary}[section]
\newtheorem{theorem}[cor]{Theorem}
\newtheorem{prop}[cor]{Proposition}
\newtheorem{lemma}[cor]{Lemma}

\theoremstyle{definition}
\newtheorem{defi}[cor]{Definition}
\theoremstyle{remark}
\newtheorem{remark}[cor]{Remark}
\newtheorem{example}[cor]{Example}

\newcommand{\C}{{\mathbb C}}

\newcommand{\R}{{\mathbb R}}
\newcommand{\Q}{{\mathbb Q}}

\newcommand{\Hyp}{\mathbb{H}}
\newcommand{\AdS}{\mathbb{A}\mathrm{d}\mathbb{S}}

\newcommand{\SL}{\mathrm{SL}}
\newcommand{\PSL}{\mathrm{PSL}}

\newcommand{\ph}{\varphi}

\newcommand{\im}{\mathrm{Im}}

\newcommand{\bb}{b}

\newcommand{\tr}{\mbox{\rm tr}\,}

\newcommand{\isom}{\mathrm{Isom}}

\newcommand{\II}{I\hspace{-0.1cm}I}

\newcommand{\rar}{\rightarrow}

\newcommand{\war}{\rightharpoonup}

\newcommand{\SO}{\mathrm{SO}}
\newcommand{\so}{\mathfrak{so}}
\newcommand{\ddt}{\left.\frac{d}{dt}\right|_{t=0}}

\newcommand{\D}{\mathbb{D}}
\newcommand{\RP}{\R \mathrm{P}}



\begin{document}

\setcounter{secnumdepth}{3}
\setcounter{tocdepth}{2}

\title[Area-preserving diffeomorphisms of $\Hyp^2$ and $K$-surfaces in $\AdS^3$]{Area-preserving diffeomorphisms of the hyperbolic plane and  $K$-surfaces in Anti-de Sitter space}

\author[Francesco Bonsante]{Francesco Bonsante}
\address{Francesco Bonsante: Dipartimento di Matematica ``Felice Casorati", Universit\`{a} degli Studi di Pavia, Via Ferrata 5, 27100, Pavia, Italy.} \email{bonfra07@unipv.it} 
\author[Andrea Seppi]{Andrea Seppi}
\address{Andrea Seppi: University of Luxembourg, Mathematics Research Unit, Maison du Nombre, 6 Avenue de la Fonte, Esch-sur-Alzette L-4364 Luxembourg. } \email{andrea.seppi@uni.lu}

\thanks{The authors were partially supported by FIRB 2010 project ``Low dimensional geometry and topology'' (RBFR10GHHH003). The first author was partially supported by
PRIN 2012 project ``Moduli strutture algebriche e loro applicazioni''.
The authors are members of the national research group GNSAGA}

\begin{abstract}




We prove that any weakly acausal curve $\Gamma$ in the boundary of Anti-de Sitter (2+1)-space is the asymptotic boundary of
two spacelike $K$-surfaces, one of which is past-convex and the other future-convex, for every $K\in(-\infty,-1)$.
The curve $\Gamma$ is the graph  of a quasisymmetric
homeomorphism of the circle if and only if the $K$-surfaces have bounded principal curvatures.
Moreover in this case a uniqueness result holds. 

The proofs rely on a well-known correspondence between spacelike surfaces in Anti-de Sitter space
and area-preserving diffeomorphisms of the hyperbolic plane. In fact, an important ingredient is a representation formula, which reconstructs a spacelike surface from the associated area-preserving diffeomorphism. 

Using this correspondence we then deduce
that, for any fixed $\theta\in(0,\pi)$, every quasisymmetric homeomorphism of the circle admits a unique extension which is a $\theta$-landslide of the hyperbolic plane. These extensions are quasiconformal.
\end{abstract}



\maketitle




\section{Introduction}

Since the groundbreaking work of Mess \cite{Mess}, the interest in the study of Anti-de Sitter manifolds has grown, often motivated by the similarities with hyperbolic three-dimensional geometry, and with special emphasis on its relations with Teichm\"uller theory of hyperbolic surfaces. See for instance \cite{ notes, bbzads, barbotzeghib, bsk_multiblack, bonschlfixed, Schlenker-Krasnov,bonseppitamb}. In fact, as outlined in \cite{aiyama,bon_schl,seppimaximal,seppiminimal}, several constructions can be generalized in the context of universal Teichm\"uller space. For instance, given a smooth spacelike convex surface $S$ in Anti-de Sitter space, or a maximal surface, one can define two \emph{projections} from $S$ to the hyperbolic plane $\Hyp^2$, and their composition provides a diffeomorphism $\Phi$ between domains of $\Hyp^2$. In the context of non-smooth surfaces, an example of this phenomenon was already introduced by Mess, who observed that a \emph{pleated surface} provides an \emph{earthquake map} of $\Hyp^2$. If $S$ is a smooth \emph{maximal surface} (namely, a surface of zero mean curvature), then the associated map is a \emph{minimal Lagrangian map}. A generalization of minimal Lagragian maps are the so-called $\theta$-\emph{landslides}, which are one of the main objects of this paper. 
A $\theta$-landslide can be defined as a composition $\Phi=f_2\circ (f_1)^{-1}$, where $f_1$ and $f_2$ are harmonic maps from a fixed Riemann surface, with Hopf differentials satisfying the relation
$$\mathrm{Hopf}(f_1)=e^{2i \theta }\mathrm{Hopf}(f_2)\,.$$
The $\theta$-landslides are precisely the maps associated to $K$-surfaces, i.e. surfaces of constant Gaussian curvature.

In this paper, we will frequently jump from one approach to the other: on the one hand, the study of convex surfaces in Anti-de Sitter space, with special interest in $K$-surfaces; on the other hand, the corresponding diffeomorphisms of (subsets of) $\Hyp^2$, in particular $\theta$-landslides.

\subsection*{A representation formula for convex surfaces}
Let us briefly review the definition of the diffeomorphism $\Phi$ associated to a convex spacelike surface in Anti-de Sitter space, and survey the previous literature. Anti-de Sitter space $\AdS^3$ can be identified to $\isom(\Hyp^2)$, endowed with the Lorentzian metric of constant curvature $-1$ which comes from the Killing form. The group of orientation-preserving, time-preserving isometries of $\AdS^3$ is identified to $\isom(\Hyp^2)\times\isom(\Hyp^2)$, acting by composition on the left and on the right. 

The essential point of the construction is the fact that the space of timelike geodesics of $\AdS^3$ is naturally identified to $\Hyp^2\times \Hyp^2$, where a point $(x,y)$ parameterizes the geodesic $L_{x,y}=\{\gamma\in\isom(\Hyp^2):\gamma(y)=x\}$, as proved in \cite{barbotbtz1}. This fact has been used in several directions, see for instance \cite{fannythesis, JDK, guerkassel, guersurvey, barbotbtz1, barbotbtz2, bbsads}. When $S$ is spacelike surface, for every point $\gamma\in S$ the orthogonal line is timelike and thus it determines two points $(x,y)$ of $\Hyp^2$. Thus one can define the \emph{left projection} $\pi_l(\gamma)=x$ and the \emph{right projection} $\pi_r(\gamma)=y$, and therefore construct the associated map $\Phi=\pi_r\circ\pi_l^{-1}$. This fact was first observed in \cite{Mess} when $S$ is a pleated surface, thus producing earthquake maps. For smooth convex surfaces, as observed in \cite{Schlenker-Krasnov}, it turns out that $\Phi=\pi_r\circ\pi_l^{-1}$ is a diffeomorphism of (subsets of) $\Hyp^2$ which preserves the area. In \cite{aiyama,bon_schl,seppimaximal} the case in which $S$ is a maximal surface, and correspondingly $\Phi$ is a minimal Lagrangian diffeomorphism, has been extensively studied. A similar construction has also been applied to surfaces with certain singularities in \cite{Schlenker-Krasnov,toulisse}. More recently, progresses have been made on the problem of characterising the area-preserving maps $\Phi$ obtained by means of this construction, satisfying certain equivariance properties, see \cite{barbotkleinian, bonsepequivariant, seppiflux}.

The first problem addressed in this paper is to what extent the surface $S$ can be reconstructed from the datum of the area-preserving diffeomorphism $\Phi$. By construction, the surface $S$ has to be orthogonal to the family of timelike geodesics $\{L_{x,\Phi(x)}\}$. Another basic observation is that, given a surface $S$, for every smooth surface obtained from $S$ by following the normal evolution of $S$, the associated map is still the same map $\Phi$. In fact, the orthogonal geodesics of the parallel surfaces of $S$ are the same as those of $S$. Up to this ambiguity, we are able to provide an explicit construction for the inverse of the left projection $\pi_l$, only in terms of $\Phi$. Actually we will give a 1-parameter family of maps into $\AdS^3$, thus reconstructing all the parallel surfaces which have associated map $\Phi$.

Of course there are conditions on the map $\Phi$. In fact, if $\Phi$ is a map which is associated to a surface $S$, then it is not difficult to prove that there exists a smooth $(1,1)$-tensor $b$ such that, if $g_{\Hyp^2}$ denotes the hyperbolic metric, 
\begin{equation} \label{condition1}
\Phi^*g_{\Hyp^2}=g_{\Hyp^2}(b\cdot,b\cdot)\,.
\end{equation}
and $b$ satisfies the conditions:
\begin{align}
&d^\nabla b=0\,, \label{condition2}\\
&\mathrm{det}\, b=1\,, \label{condition3}\\
&\tr b\in(-2,2) \label{condition4}\,,
\end{align}
where $\nabla$ is the Levi-Civita connection of $\Hyp^2$.  We prove the following converse statement:

\begin{theorem} \label{rep formula intro}
Let $\Phi:\Omega\to\Omega'$ be a diffeomorphism between two open domains of $\Hyp^2$. Suppose that there exists a smooth $(1,1)$-tensor $b$ satisfying   Equations \eqref{condition1},\eqref{condition2},\eqref{condition3},\eqref{condition4}. Consider the map
$\sigma_{\Phi,\bb}:\Omega\to\AdS^3$ defined by the condition that $\sigma_{\Phi,\bb}(x)$ is the unique isometry $\sigma$ satisfying
\begin{align}
&\sigma(\Phi(x))=x\,; \label{definition1}\\
&d\sigma_{\Phi(x)}\circ d\Phi_x=-\bb_x\,.  \label{definition2}
\end{align}
Then $\sigma_{\Phi,b}$ is an embedding of $\Omega\subseteq \Hyp^2$ onto a convex surface in $\AdS^3$. If $\pi_l$ is the left projection of the image of $\sigma_{\Phi,b}$, then $\pi_l\circ\sigma_{\Phi,b}=\mathrm{id}$, and
$\pi_r\circ\sigma_{\Phi, b}=\Phi$.
\end{theorem}
In fact, given a tensor $b$ satisfying Equations \eqref{condition1},\eqref{condition2} and \eqref{condition3}, one can still define $\sigma_{\Phi,b}$ by means of Equations \eqref{definition1} and \eqref{definition2}, and the image of the differential of $\sigma_{\Phi,b}$ is orthogonal to the family of timelike lines $\{L_{x,\Phi(x)}\}$, but in general $d\sigma_{\Phi,b}$ will not be injective. Actually, given $b$ which satisfies Equations \eqref{condition1},\eqref{condition2} and \eqref{condition3}, for every angle $\rho$, also the $(1,1)$-tensor $R_\rho\circ b$ satisfies Equations \eqref{condition1},\eqref{condition2}, where $R_\rho$ denotes the counterclockwise rotation of angle $\rho$. Changing $b$ by post-composition with $R_\rho$ corresponds to changing the surface $S$ to a parallel surface, moving along the timelike geodesics $\{L_{x,\Phi(x)}\}$. The condition $\tr b\in(-2,2)$ (which in general is only satisfied for some choices of $b$) ensures that $\sigma_{\Phi,b}$ is an embedding. 

We now restrict our attention to $K$-surfaces and $\theta$-landslides. A $\theta$-\emph{landslide} is a diffeomorphism $\Phi$ for which there exists $b$ satisfying the conditions of Equations \eqref{condition1},\eqref{condition2} and \eqref{condition3} and moreover its trace is constant. More precisely,
\begin{equation}
\tr b=2\cos\theta \qquad \mathrm{and}\qquad \tr Jb<0
\end{equation}
for $\theta\in(0,\pi)$. It turns out that $\theta$-landslides are precisely the maps associated to past-convex $K$-surfaces, for $K=-{1}/{\cos^2({\theta}/{2})}$. On the other hand, by means of the map defined in Equations \eqref{definition1} and \eqref{definition2}, one associates a $K$-surface to a $\theta$-landslide. Changing $b$ by $-b$ in the defining Equations \eqref{definition1} and \eqref{definition2} enables to pass from the $K$-surface to its dual surface, which is still a surface of constant curvature $K^*=-K/(K+1)$. Hence a $\theta$-landslide is also the map associated with a future-convex $K^*$-surface.

A special case of $\theta$-landslides are minimal Lagrangian maps, for $\theta=\pi/2$. In this case, we get two $(-2)$-surfaces, dual to one another. It is well-known (see \cite{bon_schl}) that a minimal Lagrangian map is associated to a maximal surface $S_0$ in $\AdS^3$, and that the two $(-2)$-surfaces are obtained as parallel surfaces at distance $\pi/4$ from $S_0$. Since in this case $\tr b=0$, changing $b$ by $Jb$, one has that $Jb$ is self-adjoint for the hyperbolic metric, and the map $\sigma_{\Phi,Jb}$ recovers the maximal surface with associated minimal Lagrangian map $\Phi$.

\subsection*{Foliations by $K$-surfaces of domains of dependence}

We then focus on the case of $K$-surfaces. The first result we prove in this setting concerns the existence of convex $K$-surfaces with prescribed boundary at infinity. The boundary of $\AdS^3$ is identified to $\partial\Hyp^2\times\partial\Hyp^2$, where $\partial \Hyp^2$ is the visual boundary of hyperbolic space, and thus $\partial\AdS^3$ is a torus. It is naturally endowed with a natural conformal Lorentzian structure, for which the null curves have either the first or the second component constant. A weakly acausal curve has the property that in a neighborhood of every point $\xi$, the curve is contained in the complement of the region  connected to $\xi$ by timelike segments. 

The basic example of a weakly acausal curve is the graph of an orientation-preserving homeomorphism of $\partial\Hyp^2$. 
In fact, in this paper we prove an existence theorem for $K$-surfaces with boundary at infinity the graph of any orientation-preserving homeomorphism of $\partial\Hyp^2$:

\begin{theorem}   \label{thm:qui}
Given any orientation-preserving homeomorphism $\phi:\partial\Hyp^2\to\partial\Hyp^2$, the two connected components of the complement of the convex hull of $\Gamma=graph(\phi)$ in the domain of dependence of $\Gamma$ are both foliated by $K$-surfaces $S_K$, as $K\in(-\infty,-1)$, in such a way that if $K_1<K_2$, then $S_{K_2}$ is in the convex side of $S_{K_1}$.
\end{theorem}

Analogously to the case of hyperbolic space $\Hyp^3$ (as proved in \cite{MR1293658}), there exists two $K$-surfaces with asymptotic boundary $\Gamma$, one of which is past-convex and the other future-convex.
Theorem \ref{thm:qui},
in the case of $\phi$ a quasisymmetric homeomorphism, gives positive answer to the existence part of Question 8.3 in \cite{questionsads}. As expressed in \cite{questionsads}, Question 8.3 also conjectured uniqueness  and boundedness of principal curvatures. These are proved in Theorem \ref{thm uniqueness and boundedness} below.

In general, a weakly acausal curve $\Gamma$ can contain null segments. In particular, if $\Gamma$ contains a \emph{sawtooth}, that is, the union of adjacent ``horizontal'' and ``vertical'' segments in $\partial\Hyp^2\times\partial\Hyp^2$, then the convex hull of the sawtooth is a lightlike totally geodesic triangle, which is contained both in the boundary of the convex hull of $\Gamma$ and in the boundary of the domain of dependence of $\Gamma$. Hence any (future or past) convex surface with boundary $\Gamma$ must necessarily contain such lightlike triangle.

An example is a 1-\emph{step curve}, which is the union of a ``horizontal'' and a ``vertical'' segment in $\partial\Hyp^2\times\partial\Hyp^2$. A 2-\emph{step curve} is the union of four segments, two horizontal and two vertical in an alternate way. It is not possible to have a convex surface in $\AdS^3$ with boundary a 1-step or a 2-step curve.

The proof of Theorem \ref{thm:qui} actually extends to the case of a general weakly acausal curve $\Gamma$, except the two degenerate cases above.

\begin{theorem} \label{thm Ksurfaces ads intro} 
Given any weakly acausal curve $\Gamma$ in $\partial\AdS^3$ which is not a 1-step or a 2-step curve,
for every $K\in(-\infty,-1)$ there exists a past-convex (resp. future-convex) surface $S_K^+$ (resp. $S_K^-$) with $\partial S_K^\pm=\Gamma$, such that:
\begin{itemize}
\item Its lightlike part  is  union of lightlike triangles  associated to sawteeth;
\item Its spacelike part is a smooth $K$-surface.
\end{itemize}

Moreover, the two connected components of the complement of the convex hull of $\Gamma$ in the domain of dependence of $\Gamma$ are both foliated by the spacelike part of surfaces $S_K^\pm$, as $K\in(-\infty,-1)$, in such a way that if $K_1<K_2$, then $S^\pm_{K_2}$ is in the convex side of $S^\pm_{K_1}$.
\end{theorem}

In \cite{barbotzeghib}, the existence (and uniqueness) of a foliation by $K$-surfaces was proved in the complement of the convex core of any maximal globally hyperbolic Anti-de Sitter spacetime containing a compact Cauchy surface. Using results of \cite{Mess}, this means that the statement of Theorem \ref{thm Ksurfaces ads intro}  holds for curves $\Gamma$ which are the graph of an orientation-preserving homeomorphism which conjugates two Fuchsian representations of the fundamental group of a closed surface in $\isom(\Hyp^2)$. Moreover, the $K$-surfaces are invariant for the representation in $\isom(\AdS^3)\cong \isom(\Hyp^2)\times \isom(\Hyp^2)$ given by the product of the two Fuchsian representations.

The proof of Theorem \ref{thm:qui}, and more generally Theorem \ref{thm Ksurfaces ads intro}, relies on an approximation from the case of \cite{barbotzeghib}. Some technical tools are needed. First, we need to show that it is possible to approximate any weakly spacelike curve $\Gamma$ by curves invariant by a pair of Fuchsian representations. For this purpose, we adapt a technical lemma proved in \cite{Bonsante:2015vi}.

Second, we use a theorem of Schlenker (\cite{schlenkersurfconv}) which, in this particular case, essentially ensures that a sequence $S_n$ of $K$-surfaces in $\AdS^3$ converges $C^\infty$ to a spacelike surface $S_\infty$ (up to subsequences) unless they converge to a totally geodesic lightlike plane (whose boundary at infinity is a 1-step curve) or to the union of two totally geodesic lightlike half-planes, meeting along a spacelike geodesic (in this case the boundary is a 2-step curve).  

To apply the theorem of Schlenker, and deduce that the limiting surface $S_\infty$ is a $K$-surface with $\partial S_\infty=\Gamma$ (thus proving Theorem \ref{thm:qui}), one has to prove that $S_\infty$ does not intersect the boundary of the domain of dependence of $\Gamma$. More in general, for the proof of Theorem \ref{thm Ksurfaces ads intro}, one must show that the \emph{spacelike part} of $S_\infty$ does not intersect the boundary of the domain of dependence of $\Gamma$. This is generally the most difficult step in this type of problems, and frequently requires the use of \emph{barriers}. Here this issue is indeed overcome by applying the representation formula of Theorem \ref{rep formula intro} in order to construct suitable barriers. 

In fact, it is possible to compute a family of $\theta$-landslides from a half-plane in $\Hyp^2$ to itself, which commutes with the hyperbolic 1-parameter family of isometries of $\Hyp^2$ preserving the half-plane. By using this invariance, the equation which rules the condition of a map $\Phi$ being a $\theta$-landslide is reduced to an ODE. By a qualitative study it is possible to show that there exists a 1-parameter family of smooth, spacelike $K$-surfaces whose boundary coincides with the boundary of a totally geodesic spacelike half-plane in $\AdS^3$. In other words, Theorem \ref{thm Ksurfaces ads intro}  is proved by a \emph{hands-on} approach when the curve $\Gamma$ is the union of two null segments and the boundary of a totally geodesic half-plane, in the boundary at infinity of $\AdS^3$. Such $K$-surfaces are then fruitfully used as \emph{barriers} to conclude the proof of Theorem \ref{thm:qui} - and the proof actually works under the more general hypothesis of Theorem \ref{thm Ksurfaces ads intro}.

In the case $\Gamma$ is the graph of a quasisymmetric homeomorphism, we then prove that the $K$-surfaces with boundary $\Gamma$ are unique. Moreover, it is not difficult to prove that if $S$ is a convex surface in $\AdS^3$ with $\partial S=\Gamma$ and with bounded principal curvatures, then $\Gamma$ is the graph of a quasisymmetric homeomorphism. We give a converse statement for $K$-surfaces, namely, a $K$-surface with boundary $\Gamma=gr(\phi)$, for $\phi$ quasisymmetric, necessarily has bounded principal curvatures.

\begin{theorem} \label{thm uniqueness and boundedness}
Given any quasisymmetric homeomorphism $\phi:\partial\Hyp^2\to\partial\Hyp^2$, for every $K\in(-\infty,-1)$ there exists a unique future-convex $K$-surface $S^+_K$ and a unique past-convex $K$-surface $S^-_K$ in $\AdS^3$ with $\partial S_K^\pm=gr(\phi)$. Moreover, the principal curvatures of $S_K^\pm$ are bounded.
\end{theorem}

To prove uniqueness, the standard arguments for these problems are applications of the \emph{maximum principle}, by using the existence of a foliation $\{S_K\}$ by $K$-surfaces and showing that any other $K$-surface $S_K'$ must coincide with a leaf of the given foliation. However, in this case, due to non-compactness of the surfaces, one would need a form of the maximum principle \emph{at infinity}. This is achieved more easily in this case by applying isometries of $\AdS^3$ so as to bring a maximizing (or minimizing) sequence on $S_K'$ to a compact region of $\AdS^3$. Then one applies two main tools: the first is again the convergence theorem of Schlenker, and the second is a compactness result for quasisymmetric homeomorphisms with uniformly bounded cross-ratio norm. Up to subsequences, both the isometric images of $S_K'$ and the isometric images of the leaves $\{S_K\}$ of the foliation converge to an analogous configuration in $\AdS^3$. But now it is possible to apply the classical maximum principle to conclude the argument.

\subsection*{Extensions of quasisymmetric homeomorphisms by $\theta$-landslides}

By interpreting Theorems \ref{thm Ksurfaces ads intro}  and \ref{thm uniqueness and boundedness} in the language of diffeomorphism of $\Hyp^2$ we can draw a direct consequence. 

\begin{cor} \label{cor intro 2}
Given any quasisymmetric homeomorphism $\phi:\partial\Hyp^2\to\partial\Hyp^2$ and any $\theta\in(0,\pi)$, there exist a unique $\theta$-landslide $\Phi_\theta:\Hyp^2\to\Hyp^2$ which extends $\phi$. Moreover, $\Phi_\theta$ is quasiconformal.
\end{cor}

Again, uniqueness follows from the uniqueness part Theorem \ref{thm uniqueness and boundedness} and from the construction of Theorem \ref{rep formula intro}, while quasiconformality is a consequence of boundedness of principal curvatures.

For $\theta=\pi/2$, we obtain a new proof of the following result of \cite{bon_schl}:

\begin{cor} \label{cor intro 1}
Given any quasisymmetric homeomorphism $\phi:\partial\Hyp^2\to\partial\Hyp^2$ there exists a unique minimal Lagrangian extension $\Phi:\Hyp^2\to\Hyp^2$ of $\phi$. Moreover, $\Phi$ is quasiconformal.
\end{cor}

In fact, in \cite{bon_schl} the extension by minimal Lagrangian maps was proved again by means of Anti-de Sitter geometry, by proving existence and uniqueness  (and boundedness of principal curvatures, which implies quasiconformality) of a maximal surface with boundary the graph of $\phi$. Here we instead proved the existence of two $(-2)$-surfaces, which coincide with the parallel surfaces at distance $\pi/4$ from the maximal surface - the minimal Lagrangian map $\Phi$ associated to these three surfaces in the same.






\subsection*{Organization of the paper}
In Section \ref{prel ads} we give an introduction of Anti-de Sitter space as the group $\isom(\Hyp^2)$, from a Lie-theoretical approach. Section \ref{sec causal} discusses some properties of the causal geometry of Anti-de Sitter space, including the definition of domain of dependence. In Section \ref{sec surfaces} we introduce the left and right projection from a convex surface, and thus the associated diffeomorphism of $\Hyp^2$, and their relation with the differential geometry of smooth surfaces. In Section \ref{sec representation formula} we construct the ``representation'' for the inverse of the left projection of a convex (or maximal) surface, and we prove Theorem \ref{rep formula intro}. Section \ref{sec barrier} studies $\theta$-landslides which commute with a 1-parameter hyperbolic group, constructs the ``barriers'' which are necessary to prove Theorem \ref{thm Ksurfaces ads intro} and their relevant properties. In Section \ref{sec existence} the existence part of Theorem \ref{thm Ksurfaces ads intro} is proved, while Section \ref{sec foliations}
proves that the $K$-surfaces give a foliation of the complement of the convex hull in the domain of dependence. Finally, Section \ref{sec uniqueness boundedness} proves Theorem \ref{thm uniqueness and boundedness} and discusses Corollaries \ref{cor intro 1} and \ref{cor intro 2}.

\subsection*{Acknowledgements}
We are grateful to an anonymous referee for carefully reading the paper and for several useful comments which improved the exposition. We would also like to thank Thierry Barbot for pointing out several relevant references on the topic.

\section{Anti-de Sitter space and isometries of $\Hyp^2$} \label{prel ads}

Let $\Hyp^2$ denote the hyperbolic plane, which is the unique complete, simply connected Riemannian surface without boundary of constant curvature -1. We denote by $\partial\Hyp^2$ its visual boundary, and by $\isom(\Hyp^2)$ the Lie group of orientation-preserving isometries of $\Hyp^2$. Recall that the Killing form $\kappa$ on the Lie algebra of $\mathfrak{isom}(\Hyp^2)$, namely
$$\kappa(v,w)=\mathrm{tr}(\mathrm{ad}(v)\circ\mathrm{ad}(w))\,,$$
 is $\mathrm{Ad}$-invariant. Thus it defines a bi-invariant pseudo-Riemannian metric on $\isom(\Hyp^2)$, still denoted by $\kappa$, which has signature $(2,1)$. 
We will normalize  $\kappa$ to impose that its sectional curvature is $-1$.
As this normalization will be relevant in this paper, we will briefly outline the computation of the sectional curvature of $\kappa$.
\begin{lemma}
The Killing form $\kappa$ has constant sectional curvature $-1/8$.
\end{lemma}
\begin{proof}
Let us fix $v,w\in\mathfrak{isom}(\Hyp^2)$.
By results in \cite{milnor_curvature} the sectional curvature of the plane $\Pi$ generated by $v,w$ is given by
\[
    K_{\Pi}=\frac{1}{4}\frac{\kappa([v,w], [v,w])}{\kappa(v,v)\kappa(w,w)-\kappa(v,w)\kappa(v,w)}~.
\]
Now suppose that $\kappa(v,v)=1, \kappa(w,w)=1, \kappa(v,w)=0$. Then $[v,w]\neq 0$ as $v$ and $w$ are linearly independent. Denote $u=[v,w]$.
Notice that by the properties of the Killing form $\kappa(u, v)=\kappa(u, w)=0$, so $u$ is timelike.
As $[v,u]$ is orthogonal to both $v$ and $u$, it turns out that $[v,u]=\lambda w$ and analogously
$[w,u]=\mu v$.
Imposing that $\tr(\mathrm{ad}(v)^2)=\tr(\mathrm{ad}(w)^2)=1$ we deduce that $\lambda=1/2$ and $\mu=-1/2$.
So a simple computation shows that $\kappa(u,u)=\tr(\mathrm{ad}(u)^2)=-1/2$, so that $K_{\Pi}=-1/8$. By an analogous computation, picking $v$ and $w$ with $\kappa(v,v)=1, \kappa(w,w)=-1, \kappa(v,w)=0$, the sectional curvature of a timelike plane is also $-1/8$.
\end{proof}

\begin{defi}
Anti-de Sitter space of dimension 3 is the Lie group $\isom(\Hyp^2)$ endowed with the bi-invariant metric $g_{\AdS^3}=\frac{1}{8}\kappa$, and will be denoted by $\AdS^3$.
\end{defi}

It turns out that $\AdS^3$ has the topology of a solid torus, is orientable and time-orientable, and the induced pseudo-Riemannian metric has constant sectional curvature $-1$. 

We will choose the time orientation of $\AdS^3$ in such a way that the timelike vectors which are tangent to the differentiable curve 
$$t\in[0,\epsilon)\longrightarrow R_t\circ \gamma\,,$$
where $R_t$ is a rotation of positive angle $t$ around  any point $x\in\Hyp^2$
with respect to the orientation of $\Hyp^2$, are future-directed. Moreover, we fix an orientation of $\AdS^3$ so that
 if $v, w$ are linearly independent spacelike elements of $\mathfrak{isom}(\Hyp^2)$, then $\{v,w,[v,w]\}$ is a positive basis of  $\mathfrak{isom}(\Hyp^2)$.

By construction, the group of orientation-preserving, time-preserving isometries of $\AdS^3$ is:
$$\isom(\AdS^3)\cong\isom(\Hyp^2)\times\isom(\Hyp^2)\,,$$
where the left action on $\isom(\Hyp^2)$ is given by:
$$(\alpha,\beta)\cdot \gamma=\alpha\circ\gamma\circ \beta^{-1}\,.$$

The boundary at infinity of Anti-de Sitter space is defined as 
$$\partial\AdS^3\cong \partial\Hyp^2\times \partial\Hyp^2$$
where a sequence $\gamma_n\in\isom(\Hyp^2)$ converges to a pair $(p,q)\in \partial\Hyp^2\times \partial\Hyp^2$ if there exists a point $x\in\Hyp^2$  such that 
\begin{equation} \label{defi convergence boundary}
\gamma_n(x)\to p\qquad \gamma_n^{-1}(x)\to q\,,
\end{equation}
and in this case the condition is true for any point $x$ of $\Hyp^2$. The isometric action of  $\isom(\Hyp^2)\times\isom(\Hyp^2)$ on $\AdS^3$ continuously extends to  the product action on $\partial\Hyp^2\times \partial\Hyp^2$: if $p,q\in\partial \Hyp^2$, then
$$(\alpha,\beta)\cdot(p,q)=(\alpha(p),\beta(q))\,.$$ The boundary at infinity $\partial\AdS^3$ is endowed with a conformal Lorentzian structure, in such a way that $\isom(\AdS^3)$ acts on the boundary by conformal transformations. The null lines of $\partial\AdS^3$ are precisely $\partial\Hyp^2\times\{\star\}$ and $\{\star\} \times\partial\Hyp^2$. 

Since the exponential map  at the identity 
for the Levi-Civita connection of the bi-invariant metric coincides with the Lie group exponential map, 
the geodesics through the identity are precisely the 1-parameter subgroups. 
In particular elliptic subgroups correspond to timelike geodesics through the identity.
Using the action of the isometry group, it follows that 
timelike geodesics (i.e. those of negative squared length) have the form:
$$L_{x,x'}=\{\gamma\in\isom(\Hyp^2):\gamma(x')=x\}\,.$$
They are closed and have length $\pi$. Observe that with this definition
the isometry group acts on timelike geodesics in such a way that 
\begin{equation} \label{transf rule timelike geodesic}
(\alpha,\beta)\cdot L_{x,x'}=L_{\alpha(x), \beta(x')}\,.
\end{equation}

For a similar argument as above, spacelike geodesics (i.e. those on which the metric is positive) are of the form 
$$L_{\ell,\ell'}=\{\gamma\in\isom(\Hyp^2):\gamma(\ell')=\ell\}\,,$$
where $\ell$ and $\ell'$ are oriented geodesics of $\Hyp^2$. The geodesic $L_{\ell,\ell'}$ has infinite length and its endpoints in $\partial\AdS^3$ are $(p_1,q_1)$ and $(p_2,q_2)$, where $p_1$ and $p_2$ are the final and initial endpoints of $\ell$ in $\partial \Hyp^2$, while $q_1$ and $q_2$ are the final and initial endpoints of $\ell'$.

The involutional rotations of angle $\pi$ around a point $x\in\Hyp^2$, which we denote by $\mathcal{I}_x$, are the antipodal points to the identity in the geodesics $L_{x,x}$, and form a totally geodesic plane
$$\mathcal R_\pi=\{\mathcal I_x\,:\,x\in\Hyp^2\}\,.$$ 
See also Figure \ref{fig:torus}.

\begin{figure}[htb]
\centering
\includegraphics[height=7cm]{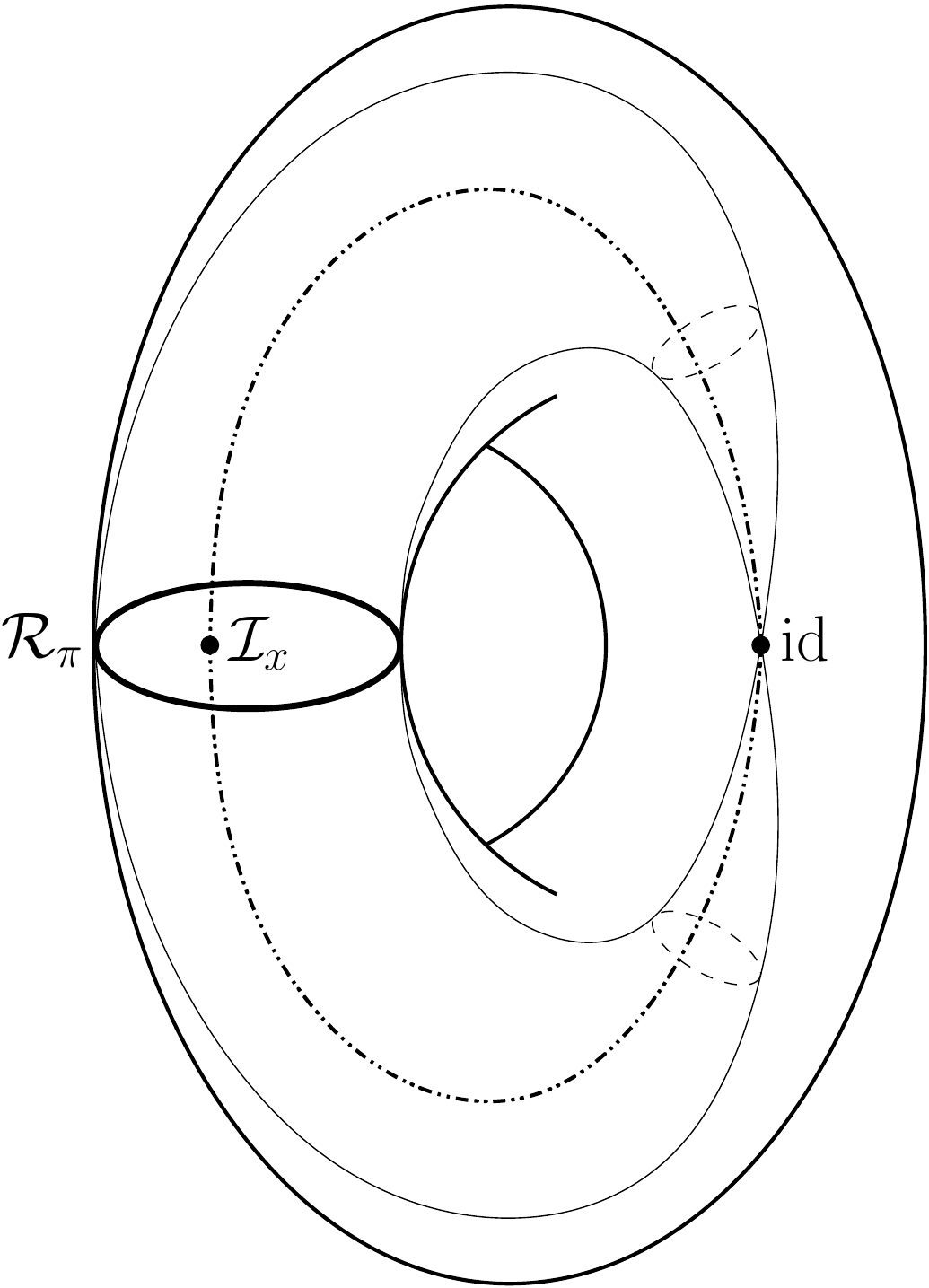}
\caption{A (topological) picture of $\AdS^3$. The totally geodesic plane $\mathcal R_\pi$ is given by the midpoints of timelike geodesics from the identity. Its boundary at infinity is the tangency locus of the lightcone from the identity with the boundary of the solid torus. \label{fig:torus}}
\end{figure}

By using the definition in Equation \eqref{defi convergence boundary}, it is easy to check that its boundary at infinity $\partial \mathcal R_\pi$ is the diagonal in $\partial\Hyp^2\times \partial\Hyp^2$:
$$\partial \mathcal R_\pi=\{(p,p)\,:\,p\in\partial\Hyp^2\}\subset\partial\Hyp^2\times \partial\Hyp^2\,.$$
More generally, given any point $\gamma$ of $\AdS^3$, the points which are connected to $\gamma$ by a timelike segment of length $\pi/2$ (actually, two timelike segments whose union form a closed timelike geodesic) form a totally geodesic plane, called the \emph{dual} plane $\gamma^*$. From this definition, $\mathcal R_\pi=(\mathrm{id})^*$. If $(\alpha,\beta)$ is an isometry of $\AdS^3$ which sends $\mathrm{id}$ to $\gamma$, then $\gamma^*=(\alpha,\beta)\cdot \mathcal R_\pi$.
In particular, $\gamma^*=(\gamma,1)\cdot \mathcal R_\pi=(1,\gamma^{-1})\cdot \mathcal R_\pi$.
The boundary at infinity of a totally geodesic spacelike plane is thus the graph of the trace on $\partial\Hyp^2$ of an isometry of $\Hyp^2$.

An isometry of $\AdS^3$ fixes a point if and only if it preserves its dual plane. In particular, 
the subgroup which fixes the identity (and therefore preserves $ \mathcal R_\pi=\mathrm{id}^*$) is given by the diagonal:
$$\mathrm{Stab}(\mathrm{id})=\mathrm{Stab}(\mathcal R_\pi)=\{(\gamma,\gamma)\,:\,\gamma\in\isom(\Hyp^2)\}\,.$$ 
We can give an explicit isometric identification of $\Hyp^2$ to the totally geodesic plane $ \mathcal R_\pi$ by means of the following map:
$$x\in\Hyp^2\mapsto\mathcal{I}_x\in  \mathcal R_\pi\,.$$
The identification is a natural isometry, in the sense that the action of $\isom(\Hyp^2)$ on $\Hyp^2$ corresponds to the action of $\mathrm{Stab}(\mathcal R_\pi)$ on $\mathcal R_\pi$, since
$$\gamma\mathcal{I}_x\gamma^{-1}=\mathcal{I}_{\gamma(x)}\,.$$

\subsection{Two useful models} \label{subsec models}
In this subsection we discuss two models of $\AdS^3$, arising from the choice of two different models of $\Hyp^2$, which will be both useful for different reasons in some computations necessary for this paper.

For the first model, let us consider $\Hyp^2$ as a sheet of the two-sheeted hyperboloid in Minkowski space, namely:
$$\Hyp^2=\{x\in \R^{2,1}\,:\,\langle x,x\rangle_{\R^{2,1}}=-1\,,\,x_3>0\}\,,$$
where the Minkowski product of $\R^{2,1}$ is:
$$\langle x,x\rangle_{\R^{2,1}}=x_1^2+x_2^2-x_3^2\,.$$
In this model, the orientation-preserving isometries of $\Hyp^2$ are identified to the connected component of the identity in $\SO(2,1)$, namely:
$$\isom(\Hyp^2)\cong \SO_0(2,1)\,.$$
Hence, the identification of $\Hyp^2$ with $\mathcal R_\pi\subset\isom(\Hyp^2)$ is given by the map which associates to $x\in\Hyp^2$ the linear map of $\R^{2,1}$ sending $x$ to $x$ and acting by multiplication by $-1$ on the orthogonal complement $x^\perp=T_x\Hyp^2$. 
Moreover, the tangent space at $\mathrm{id}$ is identified to the Lie algebra of $\SO(2,1)$:
$$T_{\mathrm{id}}\AdS^3\cong \so(2,1)\,.$$
There is a natural Minkowski cross product, defined by $x\boxtimes y=*(x\wedge y)$, where $*:\Lambda^2(\mathbb R^{2,1})\to\mathbb R^{2,1}$
is the Hodge operator associated to the Minkowski product. Observe that, by means of this cross product, one can write the almost-complex structure of $\Hyp^2$. For $v\in T_x\Hyp^2$,
\begin{equation} \label{almost-complex cross}
J_x(v)=x\boxtimes v\,.
\end{equation}
Like the classical case of Euclidean space, one can define an isomorphism
$$\Lambda:\R^{2,1}\to\so(2,1)\,,$$ which is defined by:
$$\Lambda(x)=x\boxtimes(\cdot)\,.$$
See also Figure \ref{fig:tangent}. This isomorphism has several remarkable properties:
\begin{itemize}[leftmargin=*]
\item It is equivariant for the action of $\SO(2,1)$ on $\R^{2,1}$ by isometries, and the adjoint action on $\so(2,1)$, which is the natural action of $\mathrm{Stab}(\mathrm{id})\cong \isom(\Hyp^2)$ on $T_{\mathrm{id}}\AdS^3$: if $\gamma\in\SO(2,1)$, then
\begin{equation} \label{prop lambda 1}
\Lambda(\gamma\cdot x)(v)=(\gamma\cdot x)\boxtimes v=\gamma(x\boxtimes (\gamma^{-1}v))=\gamma\circ \Lambda(x)\circ\gamma^{-1}(v)\,.
\end{equation}
\item It pulls back the Lie bracket of $\so(2,1)$ to the Minkowski cross product:
\begin{equation} \label{prop lambda 3}
[\Lambda(x),\Lambda(y)]=\Lambda(x\boxtimes y)\,.
\end{equation}
\item It is an isometry between the Killing form of $\so(2,1)$ and the Minkowski metric on $\R^{2,1}$, up to a factor:
\begin{equation} \label{prop lambda 2}
(g_{\AdS^3})_{\mathrm{id}}(\Lambda(x),\Lambda(y))=\frac{1}{4}\langle x,y\rangle_{\R^{2,1}}\,.
\end{equation}
\end{itemize}

Let us check the factor in Equation \eqref{prop lambda 2}. By \eqref{prop lambda 3} for any $x,y\in\R^{2,1}$ we have
\[
\mathrm{ad}(\Lambda(x))(\Lambda(y))=\Lambda(\Lambda(x)y)
\]
that is
\[
    \Lambda^{-1}\circ \mathrm{ad}(\Lambda(x))\circ\Lambda=\Lambda(x)~.
\]
It follows that for any $x,y\in\R^{2,1}$
\[
   (g_{\AdS^3})_{\mathrm{id}}(\Lambda(x),\Lambda(y))=\frac{1}{8}\tr(\Lambda(x)\Lambda(y))~.
\]
Using that
\[
x\boxtimes (y\boxtimes z)=(\langle x, y\rangle_{\R^{2,1}}) z-(\langle x,z\rangle_{\R^{2,1}}) y
\]
we have that $\tr(\Lambda(x)\Lambda(y))=2\langle x,y\rangle_{\R^{2,1}}$, hence obtaining Equation \eqref{prop lambda 2}.
\begin{figure}[htb]
\centering
\includegraphics[height=5cm]{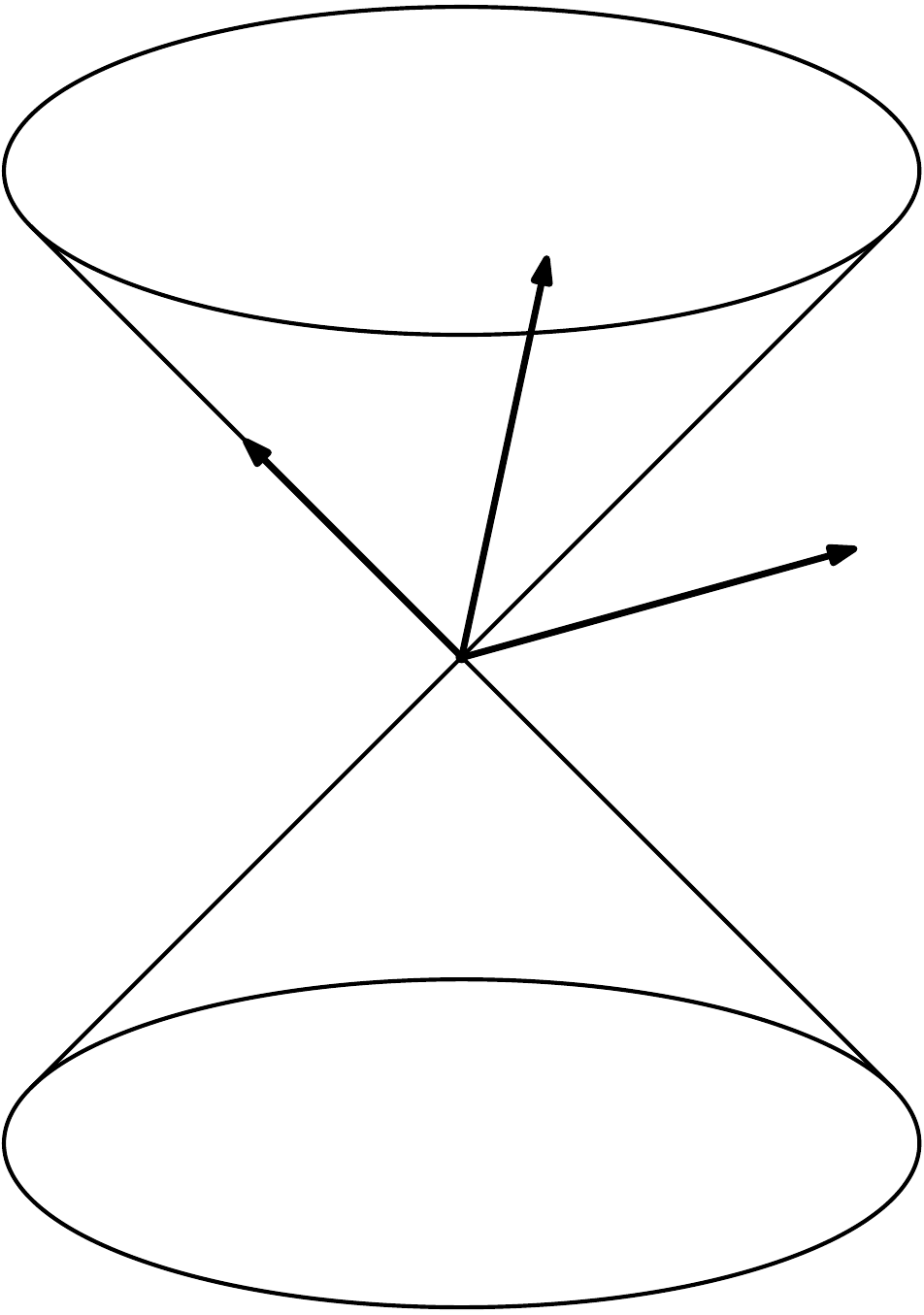}
\caption{In the identification of $T_{\mathrm{id}}\AdS^3$ with $\R^{2,1}$, timelike vectors of Minkowski space (in the interior of the cone) are the generators of elliptic 1-parameter subgroups; lightlike vectors generate parabolic subgroups and spacelike vectors generate hyperbolic subgroups. \label{fig:tangent}}
\end{figure}

The second model we consider comes from the choice of the upper half-plane model of $\Hyp^2$, namely
$$\Hyp^2=\{z\in\C:\im(z)>0\}\,,$$
endowed with the Riemannian metric
${|dz|^2}/{\im(z)^2}$
which makes every biholomorphism of the upper-half plane an isometry.
In this model $\isom(\Hyp^2)$ is naturally identified to $\PSL(2,\R)$ and the visual boundary $\partial\Hyp^2$ is identified with the extended line $\RP^1$. 
The  Lie algebra $\mathfrak{sl}(2,\R)$ is thus the vector space of traceless 2 by 2 matrices, and in this model the Anti de Sitter metric at the identity is given by:
\begin{equation}\label{eq:minnie}(g_{\AdS^3})_{\mathrm{id}}( m,m')=\frac{1}{2}\tr(mm')\,.\end{equation}

Consider  the quadratic form $q(M)=-\det M$ on the space of $2$-by-$2$ matrices.
Its polarization, say $b$,
has signature $(2,2)$. We notice that the restriction of
$b$
to 
$\SL(2,\mathbb R)$ is exactly the double cover of $g_{\AdS^3}$. Indeed left and right multiplication by elements in $\SL(2,\R)$ preserve 
$q$, so that the restriction of $b$
to $\SL(2,\R)$ is a bi-invariant metric. Moreover at the identity
Equation \eqref{eq:minnie} shows that it coincides with $g_{\AdS^3}$.

Hence this model of $\AdS^3$, namely
$$\PSL(2,\R)=\{A:q(A)=-1\}/\{\pm 1\}$$
endowed with the pseudo-Riemannian metric which descends from $b$,
is remarkably a \emph{projective model}. 
In fact, $\PSL(2,\R)$ is a subset of $\RP^3$, geodesics for the pseudo-Riemannian metric of $\PSL(2,\R)$ are the intersections of $\PSL(2,\R)$ 
with a projective line, and totally geodesic planes are the intersections with projective planes. 
Thus in an affine chart, geodesics are straight lines and totally geodesic planes are affine planes. 

\begin{figure}[htb]
\centering
\includegraphics[height=6cm]{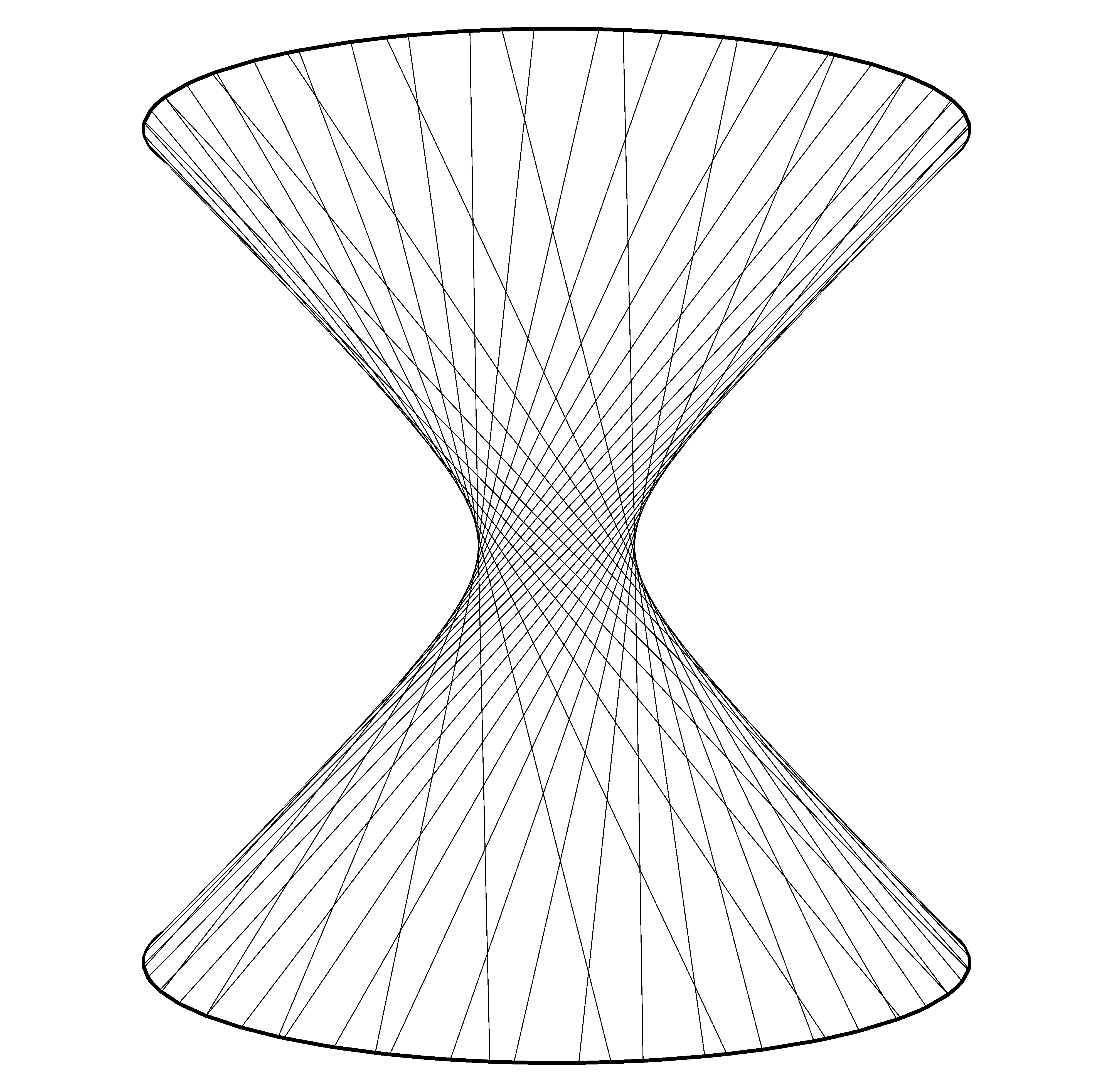}
\caption{In an affine chart, $\AdS^3$ is the interior of a one-sheeted hyperboloid. The null lines of $\partial\AdS^3$ coincide with the rulings of the hyperboloid. 
The intersection with the horizontal plane $z=0$ is a totally geodesic hyperbolic plane, in the Klein model. \label{fig:ruled}}
\end{figure}

\section{Causal geometry of $\AdS^3$} \label{sec causal}

Given a continuous curve $\Gamma$ in $\partial\AdS^3$, we say that a $\Gamma$ is  \emph{weakly acausal} if for every point $p$ of $\Gamma$, 
there exists a neighborhood $U$ of $p$ in $\partial\AdS^3$ such that $U\cap\Gamma$ is contained in the complement of the regions of $U$ which are connected to $p$ by timelike curves. 

There are two important objects we associate to a weakly spacelike curve.

\begin{defi}
Given a weakly acausal curve $\Gamma$, 
the convex hull $\mathcal{C}(\Gamma)$ is the smallest closed convex subset which contains $\Gamma$. 
\end{defi}
It turns out that $\mathcal{C}(\Gamma)$ is contained in $\AdS^3\cup\partial\AdS^3$ and that $\mathcal{C}(\Gamma)\cap\partial\AdS^3=\Gamma$ (see \cite[Lemma 4.8]{bon_schl}).

\begin{defi}
Given a weakly acausal curve $\Gamma$, 
the domain of dependence $\mathcal{D}(\Gamma)$ is the union of points $p$ of $\AdS^3$ such that $p^*$ is disjoint from $\Gamma$.
\end{defi}

It turns out that the domain of dependence $\mathcal{D}(\Gamma)$ is always an open subset of $\AdS^3$ containing the interior part of $\mathcal C(\Gamma)$.
Moreover, $\mathcal{D}(\Gamma)$ is contained in an affine chart for $\AdS^3$, and admits no timelike support planes. 
As for the convex hull, $\overline{\mathcal{D}(\Gamma)}\cap\partial\AdS^3=\Gamma$. 
Let us denote by $\mathcal{D}_\pm(\Gamma)$ the connected components of $\mathcal{D}(\Gamma)\setminus \mathcal{C}(\Gamma)$, so that
$$\mathcal{D}(\Gamma)\setminus \mathcal{C}(\Gamma)=\mathcal{D}_+(\Gamma)\sqcup\mathcal{D}_-(\Gamma)\,.$$
We choose the notation in such a way that there exists a  future-directed timelike arc in $\mathcal{D}(\Gamma)$
 going from  $\mathcal{D}_-(\Gamma)$ to $\mathcal{D}_+(\Gamma)$ and intersecting $\mathcal{C}(\Gamma)$.

\begin{defi} \label{defi locally convex}
Let $K$ be any convex subset of $\AdS^3$ contained in some affine chart.
A \emph{locally convex spacelike (resp. nowhere timelike)}  surface $S$ is a connected region of $\partial K$ such that the support planes of $K$ at points of $S$ are all spacelike (resp. non timelike).  
\end{defi}

\begin{lemma}\label{lemma:bandabassotti}
Let $S$ be a locally convex nowhere timelike surface contained in an affine chart $U$.
Then  $S$ is contained either in all future half-spaces bounded by support planes at points in $S$ or in 
all past half-spaces. 
\end{lemma}

In the former case we say that $S$ is \emph{future-convex}, in the latter \emph{past-convex}.

\begin{proof}
For any point $x\in S$ and any support plane $P$ at $x$,
$S$ is contained either in the future or in the past half-space bounded by $P$.
It is immediate to verify that  the set of $(x,P)$ such that $S$ is contained in
the future half-space of $P$ is open and closed in the set of pairs $(x,P)$ with $x\in S$ and $P$ support plane at $x$. 
As we are assuming that $S$ is connected, we conclude that this set is either the whole
set or the empty set.
\end{proof}

For instance, the boundary of $\mathcal{C}(\Gamma)$ is composed of two \emph{pleated} nowhere timelike surfaces, 
which we denote by $\partial_\pm\mathcal{C}(\Gamma)$, so that $\partial_+\mathcal{C}(\Gamma)$ is past-convex and 
$\partial_-\mathcal{C}(\Gamma)$ is future-convex. They are distinct unless $\Gamma$ is the boundary of a totally geodesic plane. 
More precisely, the subset of $\partial_\pm\mathcal{C}(\Gamma)$ which admits spacelike support planes is a pleated hyperbolic surface. 
That is, there exists an isometric map $f_\pm:H_\pm\to \partial_\pm\mathcal{C}(\Gamma)$, where $H_\pm$ are complete simply connected 
hyperbolic surfaces with geodesic boundary (also called \emph{straight convex sets} in \cite{bebo}), such that every point $x\in H_\pm$ 
is in the interior of a geodesic arc which is mapped isometrically to a spacelike geodesic of $\AdS^3$.

Given two spacelike planes $Q_1,Q_2$ in $\AdS^3$, we can define their \emph{hyperbolic angle} $\alpha\geq 0$ by:
$$\cosh\alpha=|\langle N_1,N_2\rangle|\,,$$
where $N_1$ and $N_2$ are the unit normal vectors. Using this notion, one can define (similarly to \cite{epsteinmarden} in the case of hyperbolic geometry) a \emph{bending lamination}, which is a measured geodesic lamination on the straight convex set $H_\pm$ in the sense of \cite[Section 3.4]{bebo}. In particular, the \emph{transverse measure} satisfies the usual requirements in the definition of measured geodesic laminations (see also Subsection \ref{subsec approximation} for more details), with the additional property that the measure on a transverse arc $I$ is infinite if and only if $I$ has nonempty intersection with $\partial H_\pm$. The data of a straight convex set $H$ and the measured geodesic lamination on $H$ determines the pleated surface up to global isometries of $\AdS^3$.

It is not difficult to show that $\partial_+\mathcal{C}(\Gamma)$ intersects the boundary of the domain of dependence $\mathcal{D}(\Gamma)$ in $\AdS^3$ if and only if the weakly acausal curve $\Gamma$ contains a \emph{past-directed sawtooth}, that is, if it contains two segments of $\partial\AdS^3$, one which is part of a line of the left ruling and the other in a line of the right ruling, having a common past endpoint. In this case, $\partial_+\mathcal{C}(\Gamma)$ contains a totally geodesic lightlike triangle, bounded by the two above segments in $\partial\AdS^3$ and by a spacelike complete geodesic of $\AdS^3$. The bending lamination of $\partial_+\mathcal{C}(\Gamma)$ has therefore infinite weight on such spacelike geodesic. Of course one can give the analogous definition and characterization for \emph{future-directed sawteeth}, which can possibly be contained in $\partial_-\mathcal{C}(\Gamma)$. See also Figure \ref{fig:sawtooth}.

\begin{figure}[htb]
\centering
\begin{minipage}[c]{.45\textwidth}
\centering
\includegraphics[height=5.5cm]{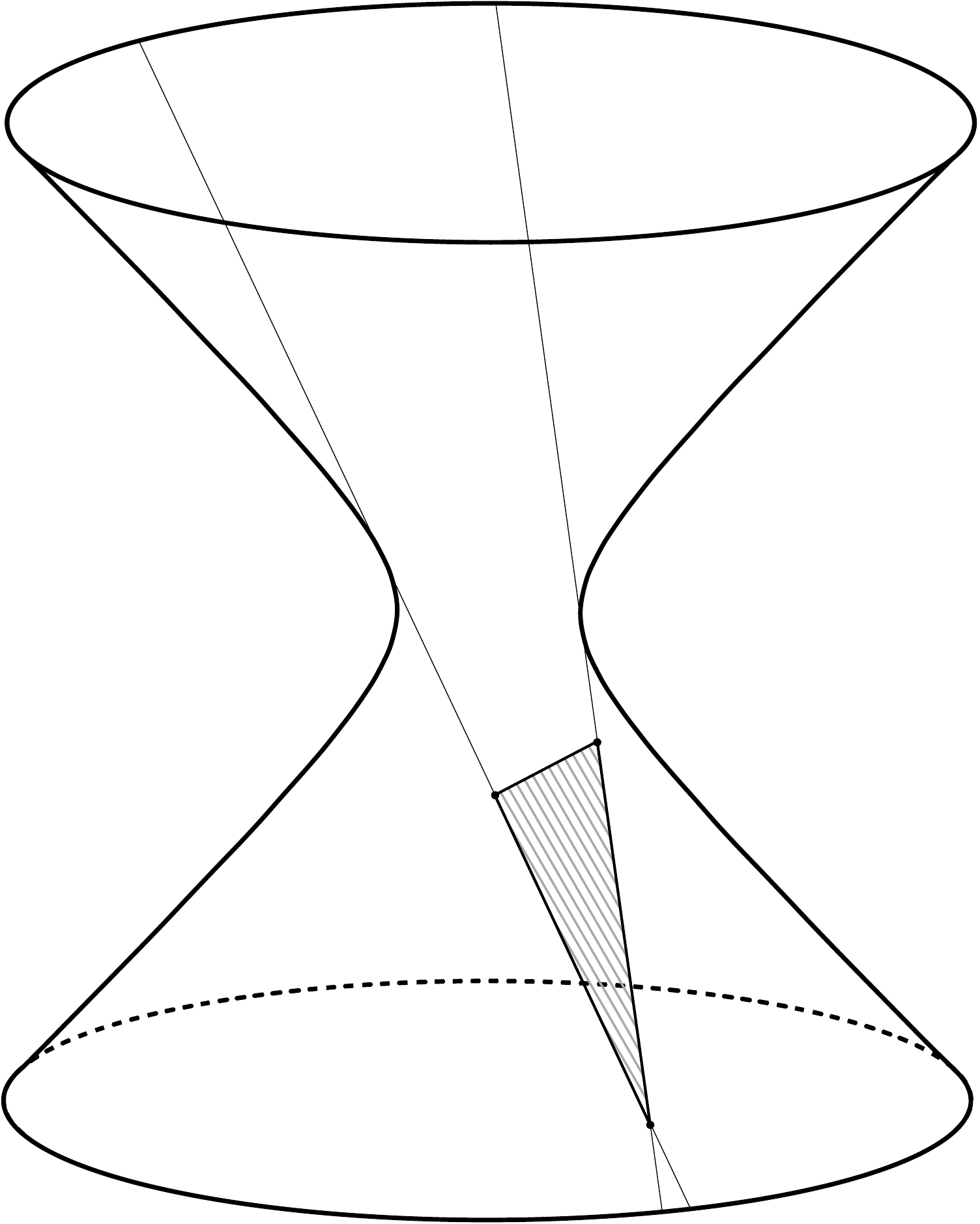}
\end{minipage}%
\hspace{1mm}
\begin{minipage}[c]{.45\textwidth}
\centering
\includegraphics[height=5.5cm]{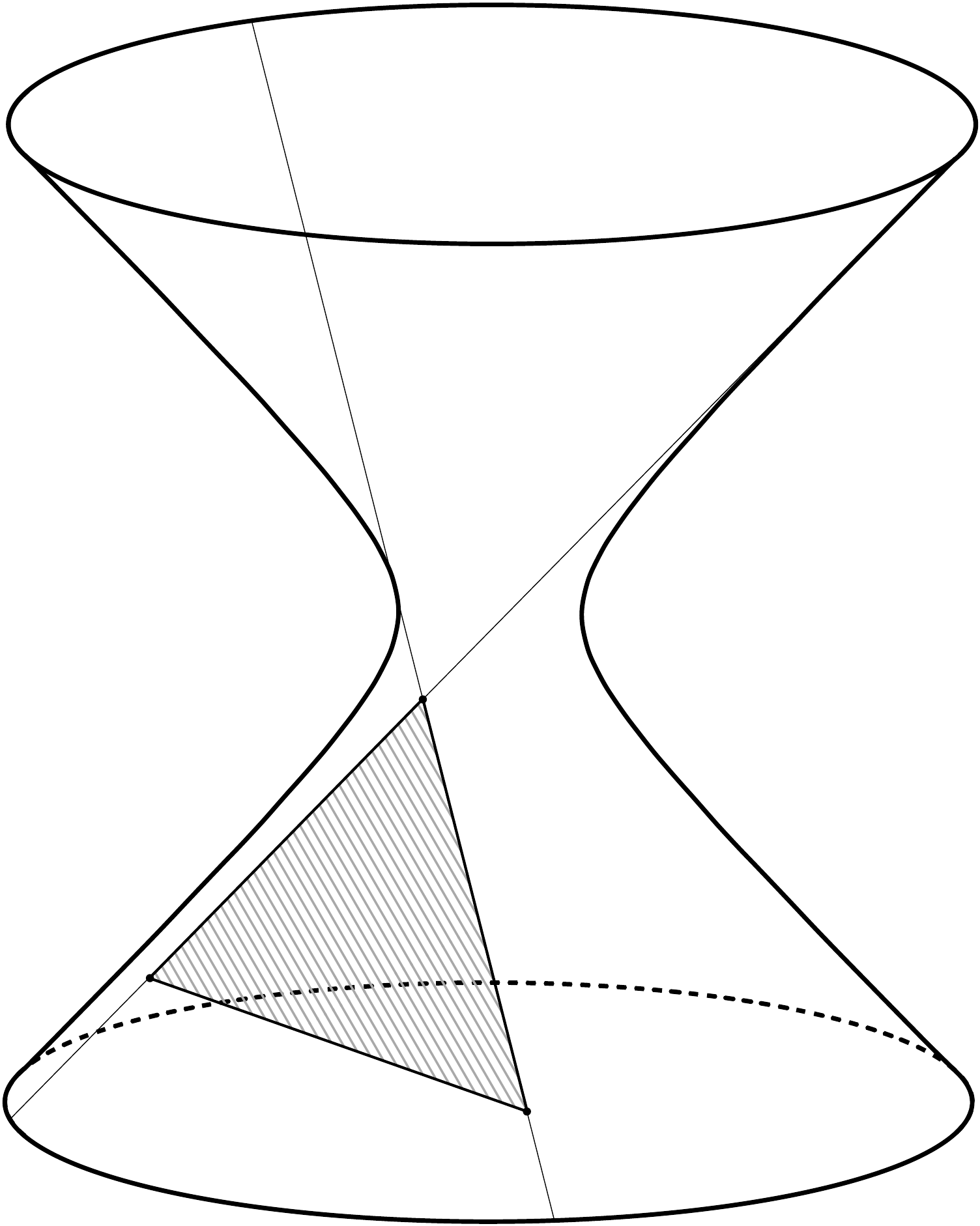}
\end{minipage}
\caption{A past-directed and a future-directed sawtooth bounding a lightlike triangle. \label{fig:sawtooth}}
\end{figure}

In this paper, we are interested in convex nowhere timelike surfaces having as a boundary at infinity a weakly acausal curve $\Gamma\subset\partial\AdS^3$. Given a future-convex (resp. past-convex) surface $S$ with $\partial S=\Gamma$, $S$ must necessarily be contained in $\overline{\mathcal{D}_-(\Gamma)}$ (resp. $\overline{\mathcal{D}_+(\Gamma)}$), and in the closure of the complement of $\mathcal{C}(\Gamma)$. 

If $\partial S=\Gamma$, the lightlike part of a future-convex surface $S$ (i.e. the subsurface of $S$ admitting lightlike support planes) necessarily contains $\partial_-\mathcal C(\Gamma)\cap \partial\mathcal D(\Gamma)$, that is, it contains all future-directed sawteeth. Analogously, the lightlike part of a past-convex surface $S$ contains $\partial_+\mathcal C(\Gamma)\cap \partial\mathcal D(\Gamma)$. 

The convex surface $\partial_-\mathcal C$ has the property that its lightlike part coincides precisely with the union of all future-directed sawteeth. The same of course holds for $\partial_+\mathcal C$. We define $\Gamma_-$ as the boundary of the spacelike part of $\partial_-\mathcal C$.

The curve $\Gamma_-$ coincides with $\Gamma$ in the complement of future-directed sawteeth. If  $\Gamma$ contains future-directed sawteeth, then $\Gamma_-$ contains the spacelike lines whose endpoints coincide with the endpoints of each sawtooth. The same description, switching future with past sawteeth, can be given for past-convex nowhere timelike surfaces. In particular we define $\Gamma_+$ as the boundary of the spacelike part of $\Gamma_+$.
See also Figure \ref{fig:curvegamma}. The $K$-surfaces which, in Theorem \ref{thm Ksurfaces ads intro}, form a foliation of the future (resp. past) connected component of $\mathcal D(\Gamma)\setminus \mathcal C(\Gamma)$ will have as boundary the curve $\Gamma_+$ (resp. $\Gamma_-$). 

\begin{figure}[htb]
\centering
\includegraphics[height=8cm]{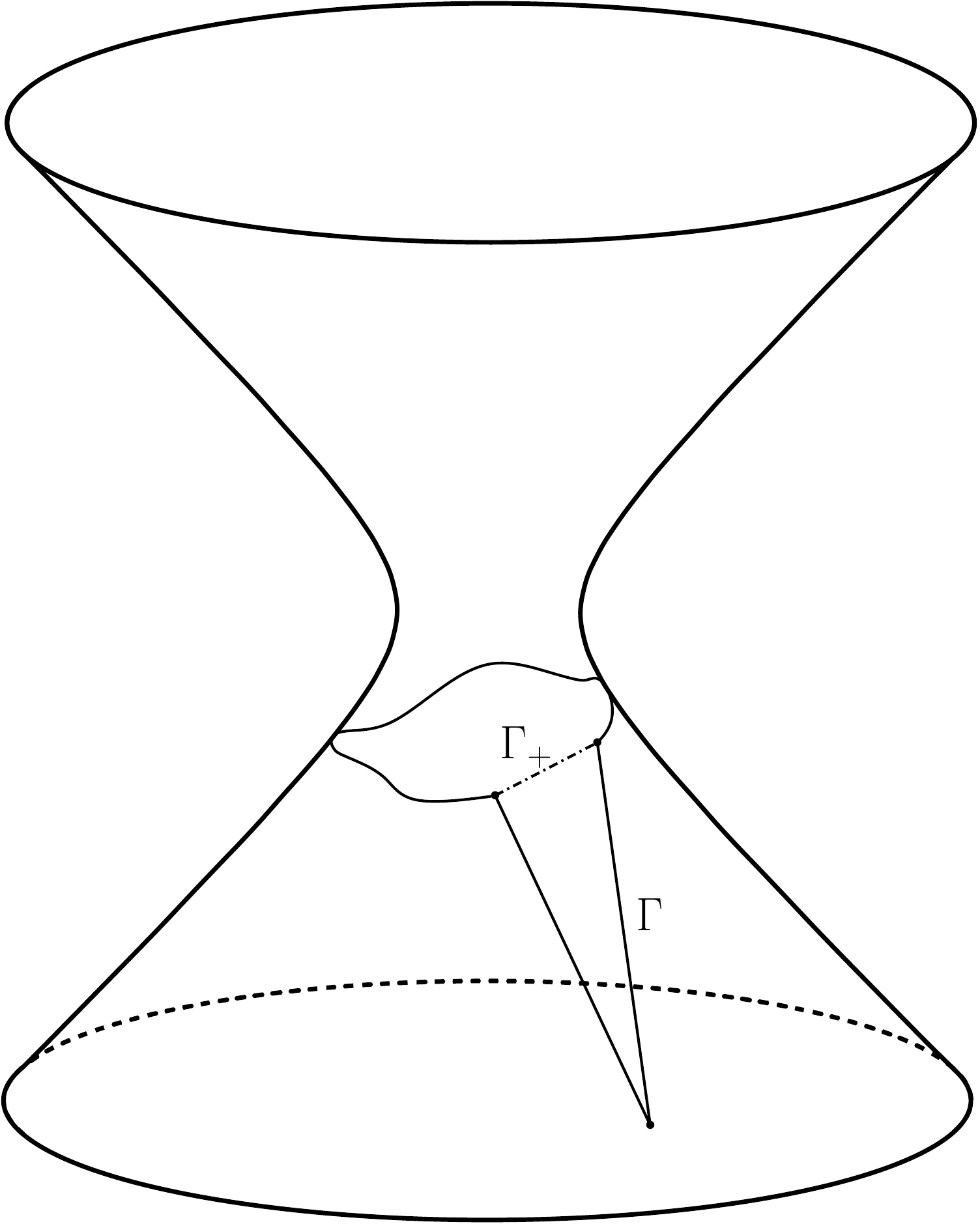}
\caption{If the curve $\Gamma$ contains a past-directed sawtooth, then $\Gamma_+$ coincides with $\Gamma$ in the complement of the sawtooth, whereas $\Gamma_+$ contains the spacelike line which is the third side of the lightlike triangle bounded by the sawtooth. \label{fig:curvegamma}}
\end{figure}

\begin{example} \label{ex crucial}
An example of the above constructions, which will be crucial in Section \ref{sec barrier} and \ref{sec existence}, is the following. Consider the curve $\Gamma'$ which is composed of the boundary of a totally geodesic half-plane and of a past-directed sawtooth. See Figure \ref{fig:example}. Then the upper boundary of the convex hull is composed of the totally geodesic half-plane and of the lightlike triangle bounded by the past sawtooth. On the other hand, the lower boundary of the convex hull $\partial_-\mathcal{C}(\Gamma')$ is pleated along a lamination which foliates the hyperbolic surface $\partial_-\mathcal{C}(\Gamma')$.

It is not difficult to understand the domain of dependence $\mathcal D(\Gamma')$. In fact, the lower boundary of $\mathcal D(\Gamma')$ consists of the surface obtained by two totally geodesic lightlike half-planes which intersect along a spacelike half-geodesic ($L_{\ell,\ell'}$ in Figure \ref{fig:example}, left), glued to the cone over the end-point of the half-geodesic. On the other hand, the component $\mathcal D_+(\Gamma')$, in the affine chart of the right side of Figure \ref{fig:example}, is a region of a vertical cylinder, bounded by the horizontal plane containing $\Gamma'_+$, and by the lightlike plane containing the past-directed sawtooth.
\end{example}

\begin{figure}[htb]
\centering
\begin{minipage}[c]{.45\textwidth}
\centering
\includegraphics[height=7cm]{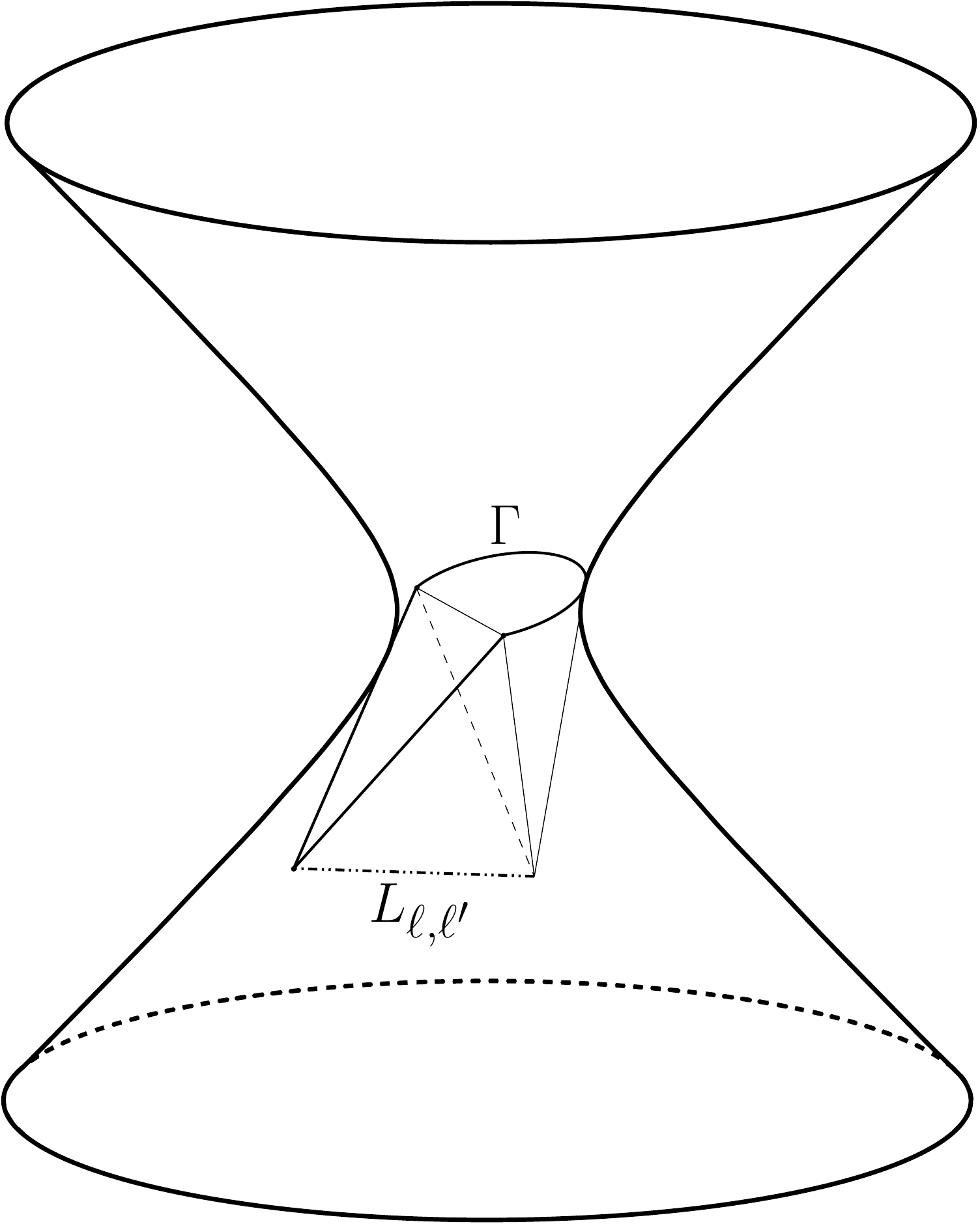}
\end{minipage}%
\hspace{3mm}
\begin{minipage}[c]{.45\textwidth}
\centering
\includegraphics[height=7cm]{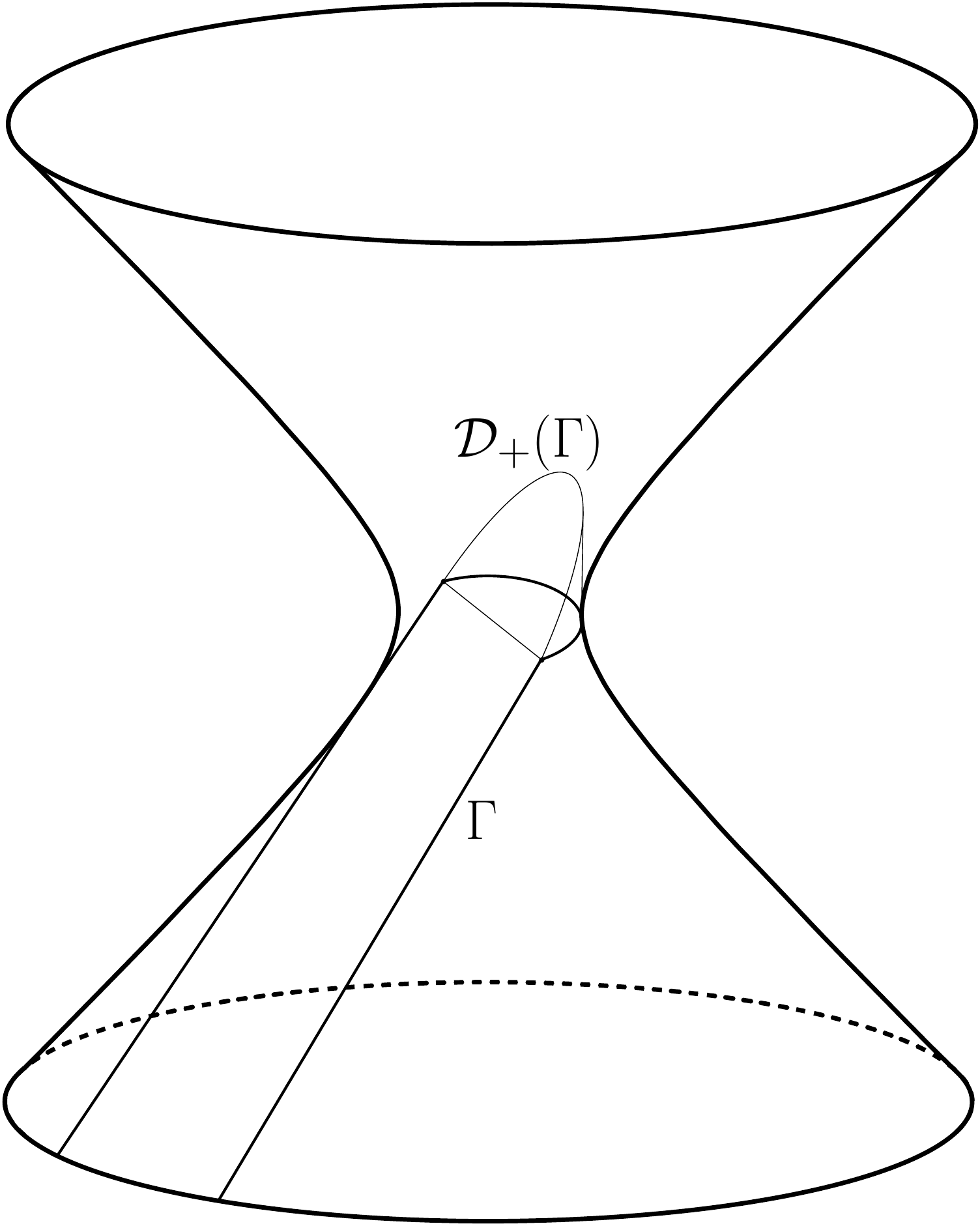}
\end{minipage}
\caption{The two boundaries of the domain of dependence $\mathcal D(\Gamma)$, where $\Gamma$ is the curve of Example \ref{ex crucial}. \label{fig:example}}

\end{figure}

\section{From differentiable surfaces in $\AdS^3$ to local diffeomorphisms of $\Hyp^2$} \label{sec surfaces}

The purpose of this section is to explain the construction which associates to a differentiable embedded spacelike surface in $\AdS^3$ a map between subsets of $\Hyp^2$. Most of the results of this section have been already known and used before, see also the references mentioned in the introduction. 

\subsection{Left and right projections} \label{subsec: projections}
Given a differentiable spacelike surface $ S$ in $\AdS^3$,  one defines two projections 
$\pi_l: S\to \Hyp^2$ and $\pi_r: S\to \Hyp^2$ in the following way.
If the timelike geodesic orthogonal to $S$ at $\gamma$ is $L_{x,x'}$, then 
$$\pi_l(\gamma)=x\qquad\qquad\pi_r(\gamma)=x'\,.$$ 

An equivalent construction of the projections is the following.
Given a point $\gamma\in  S$, let us denote by $P_\gamma S$ the totally geodesic plane tangent to $S$ at $\gamma$.
Then, there exists a unique isometry 
$$\Pi^\gamma_l\in\{\mathrm{id}\}\times\isom(\Hyp^2)$$
which maps $P_\gamma  S$ to $\mathcal R_{\pi}$. 
In fact, if we denote $G(\gamma)=(P_\gamma  S)^*$, then the isometry $(\mathrm{id}, G(\gamma))\in\isom(\Hyp^2)\times \isom(\Hyp^2)$ 
maps $G(\gamma)$ to $\mathrm{id}$, and therefore coincides with $\Pi^\gamma_l$. Analogously consider the unique
$$\Pi^\gamma_r\in\isom(\Hyp^2)\times\{\mathrm{id}\}$$
mapping $P_\gamma  S$ to $\mathcal R_\pi$, which has the form $(G(\gamma)^{-1},\mathrm{id})$.
 We then have
$$\Pi^\gamma_l(\gamma)=\gamma G(\gamma)^{-1}=\mathcal I_{\pi_l(\gamma)}\,,\qquad\qquad \Pi^\gamma_r(\gamma)=G(\gamma)^{-1}\gamma =\mathcal I_{\pi_r(\gamma)}\,.$$ 
Indeed, the previous formula is  true if $\gamma=\mathcal I_{x}$, and 
$P_\gamma S=\mathcal R_\pi$, in which case the isometries $\Pi^\gamma_l$ and $\Pi^\gamma_r$ are trivial and the normal geodesic is 
$L_{x,x}$. The general case then follows by observing that, from Equation \eqref{transf rule timelike geodesic}, 
the projection to the point $x$ of $L_{x,x'}$ is invariant by composition with a right isometry, and analogously $x'$ does not change when composing with an isometry on the left.



\begin{example} \label{ex pleated earthquake}
The most basic example of left and right projections can be seen when $ S$ is a totally geodesic plane. In this case, the $P_\gamma S=S$  and thus $\pi_l$ (resp. $\pi_r$) is the restriction to $ S$ of the unique left (resp. right) isometry of $\AdS^3$ sending $ S$ to $\mathcal R_\pi$ (up to the identification of $\mathcal R_\pi$ with $\Hyp^2$). More generally, if $ S$ is a pleated surface, the left and right projections are well-defined for every point $\gamma\in S$ which admits only one support plane. Moreover, the isometries $\Pi_l^\gamma$ and $\Pi_r^\gamma$ are constant on each stratum of the bending lamination. The induced metric on $ S$ is the hyperbolic metric of a straight convex set $H$, and it can be proved (see \cite{Mess}) that the projections are earthquake maps from $H$ to $\Hyp^2$. 
\end{example}

Using Example \ref{ex pleated earthquake}, one easily proves that if $ S$ is convex, then $\pi_l$ and $\pi_r$ are injective. 

\begin{lemma} \label{proj injective convex}
Let $ S$ be an embedded differentiable convex spacelike surface in $\AdS^3$. Then $\pi_l$ and $\pi_r$ are injective.
\end{lemma}
\begin{proof}
By Lemma \ref{lemma:bandabassotti} we may assume that $S$ is future convex.
Pick two points $\gamma,\gamma'\in S$. By the hypothesis on  $S$, the totally geodesic planes $P_\gamma S$ and $P_{\gamma'} S$ are spacelike and $S$ is contained in the intersection, say $K_1$, 
of the future half-spaces bounded by $P_\gamma S$ and $P_{\gamma'} S$, in some affine chart. Let $S_1$ be the boundary of $K_1$. It is either a totally geodesic plane or a pleated surface with pleating locus made by a single
geodesic.
Notice that $S_1$ is tangent to $S$ at $\gamma$ and $\gamma'$ so that the corresponding projections $\pi_l^1$ and $\pi_r^1$ coincide with $\pi_l$ and $\pi_r$ at $\gamma$ and $\gamma'$.
As $\pi_l^1$ and $\pi_r^1$ are both injective, we conclude that $\pi_l(\gamma)\neq\pi_l(\gamma')$ and $\pi_r(\gamma)\neq\pi_r(\gamma')$.
\end{proof}

It turns out that, if $S$ is oriented by means of the ambient orientation of $\AdS^3$ and the choice of the future-directed normal vector field, then $\pi_l$ and $\pi_r$ are orientation-preserving.

When $ S$ is an embedded differentiable convex spacelike surface, one can define a bijective map from $\Omega_l=\pi_l( S)$ to $\Omega_r=\pi_r( S)$, by
$$\Phi:\Omega_l\to\Omega_r\qquad \Phi=\pi_r\circ (\pi_l)^{-1}.$$ 

In \cite{bon_schl} it was proved that if $ S$ is a convex surface with $\partial S$ a curve $\Gamma\subset\partial\AdS^3=\partial\Hyp^2\times \partial\Hyp^2$ which is the graph of an orientation-preserving homeomorphism $\phi:\partial\Hyp^2\to\partial\Hyp^2$, then $\Omega_l=\Omega_r=\Hyp^2$, and $\Phi$ extends to $\phi$ on $\partial\Hyp^2$. This is essentially the content of the next lemma.

\begin{lemma}[{\cite[Lemma 3.18, Remark 3.19]{bon_schl}}]  \label{lemma extension}
Let $ S$ be  a convex spacelike surface in $\AdS^3$ with $\partial S=gr(\phi)$, for $\phi$ an orientation-preserving homeomorphism of $\partial\Hyp^2$. If $\gamma_n\in  S$ converges to $(p,q)\in\partial\AdS^3$, then $\pi_l(\gamma_n)$ converges to $p$ and $\pi_r(\gamma_n)$ converges to $q$.
\end{lemma}

In the case of pleated surfaces $ S$, as already observed by Mess \cite{Mess}, the associated map $\Phi$ is an earthquake map. Hence if $\phi:\partial\Hyp^2\to\partial\Hyp^2$ is an orientation-preserving homeomorphism, the boundaries of the convex hull of $gr(\phi)$ are pleated surfaces whose associated maps provide a left and a right earthquake map which extends $\phi$, hence recovering a theorem of Thurston \cite{thurstonearth}.

\begin{remark} \label{remark normal flow}
The map $\Phi$ is constant under the normal evolution of a surface. More precisely, suppose $ S$ is a differentiable surface such that the parallel surface $ S_\epsilon$ is well-defined for a short time $\epsilon$. Then the normal timelike geodesics of $ S_\epsilon$ are the same as those of $ S$. Hence by the above equivalent definition, using the identification of $ S$ and $ S_\epsilon$ given by the normal flow, the maps $\pi_l$ and $\pi_r$ for $ S$ and $ S_\epsilon$ are the same. In particular, the composition $\Phi=\pi_r\circ (\pi_l)^{-1}$ is invariant for the normal evolution.
\end{remark}

\begin{remark} \label{remark dual surface}
For the same reason, the left and right projections are the ``same" for a convex differentiable surface and its dual surface. 
More precisely, if $S$ is a spacelike convex differentiable surface, then one can define the dual surface $S^*$ as the image of 
the map from $S$ to $\AdS^3$ which associates to $\gamma\in S$ the dual point $(P_\gamma S)^*$. 
Recalling that $(P_\gamma S)^*$ is the midpoint of all timelike geodesics orthogonal to $P_\gamma S$, it turns out that $S^*$ is the 
$\pi/2$-parallel surface of $S$ and the identification between $S$ and $S^*$ is precisely the $\pi/2$-normal evolution. 
Therefore the area-preserving diffeomorphisms $\Phi$ associated with $S$ and $S^*$ coincide. 
\end{remark}

\subsection{Relation with the extrinsic geometry of surfaces} \label{subsec surfaces}

A smooth embedded surface $\sigma:\Sigma\rar\AdS^3$ is called spacelike if its tangent plane is a spacelike plane at every point, so that the \emph{first fundamental form} $I(v,w)=\langle d\sigma(v),d\sigma(w)\rangle$ is a Riemannian metric on $\Sigma$. 
Let $N$ be a unit normal vector field to the embedded surface $S=\sigma(\Sigma)$. We denote by $\nabla^{\AdS^3}$ and $\nabla^I$ the ambient connection of $\AdS^3$ and the Levi-Civita connection of the first fundamental form of the surface $\Sigma$, respectively. If $\sigma$ is $C^2$, the \emph{second fundamental form} of $\Sigma$ is defined by the equation
$$\nabla^{\AdS^3}_{d\sigma(v)}d\sigma(\tilde w)=d\sigma\nabla^I_{v}\tilde w+\II(v,w)N$$
if $\tilde w$ is a vector field extending $w$, and $N$ is the future-directed unit normal vector field. The \emph{shape operator} is the $(1,1)$-tensor defined as 
$$B(v)=(d\sigma)^{-1}\nabla^{\AdS^3}_{d\sigma(v)} N\,.$$ It is a self-adjoint operator for $I$, such that
$$\II(v,w)=I( B(v),w)\,.$$
It satisfies the Codazzi equation $d^{\nabla^I}\! B=0$, where (for vector fields $\tilde v$ and $\tilde w$ on $S$ extending $v$ and $w$):
$$d^{\nabla^I}\!B(v,w)=\nabla^I_{\tilde v} (B(\tilde w))-\nabla^I_{\tilde w} (B(\tilde v))-B[\tilde v,\tilde w]\,.$$
Moreover, $B$ satisfies the Gauss equation
\begin{equation} \label{Gauss}
K_I=-1-\det B\,,
\end{equation}
where $K_I$ is the curvature of the first fundamental form.

In \cite{Schlenker-Krasnov}, a formula for the pull-back on $S$ of the hyperbolic metric of $\Hyp^2$ by the left and right projections was given. We report this formula here:

\begin{prop} \label{formula KS}
Let $S$ be a smooth embedded spacelike surface in $\AdS^3$. Then:
\begin{equation} \label{eq pullback left}
\pi_l^*g_{\Hyp^2}(v,w)=I((E+J_I B)v,(E+J_I B)w)\,,
\end{equation}
and  
\begin{equation} \label{eq pullback right}
\pi_r^*g_{\Hyp^2}(v,w)=I((E-J_I B)v,(E-J_I B)w)\,.
\end{equation}
Here $E$ denotes the identity operator, and $J_I$ is the almost-complex structure induced by $I$.
\end{prop}

Observe that, since $B$ is self-adjoint for $I$, $\mathrm{tr}(J_I B)=0$ and thus $\det(E\pm J_I B)=1+\det B$. Hence if $\det B\neq -1$, then $\pi_l$ and $\pi_r$ are local diffeomorphisms. 
If moreover $S$ is convex ($\det B\geq 0$), then by  Lemma \ref{proj injective convex}, $\pi_l, \pi_r$ are diffeomorphisms onto their images. 
Since $\det(E+J_I B)=\det(E-J_I B)$, the composition $\Phi=\pi_r\circ(\pi_l)^{-1}$ is area-preserving. In conclusion, we have:

\begin{cor}
If $S$ is a smooth convex embedded spacelike surface in $\AdS^3$, then $\pi_l$ and $\pi_r$ are diffeomorphisms on their images. Therefore, $\Phi=\pi_r\circ(\pi_l)^{-1}:\Omega_l\to\Omega_r$, where $\Omega_l=\pi_l(S)$ and $\Omega_r=\pi_r(S)$, is an area-preserving diffeomorphism.
\end{cor}

Moreover, using Lemma \ref{lemma extension}, one gets:
\begin{cor}
If $S$ is a smooth convex embedded spacelike surface in $\AdS^3$ with $\partial S=gr(\phi)\subset\partial\AdS^3$, for $\phi$ an orientation-preserving diffeomorphism of $\partial\Hyp^2$, then $\pi_l$ and $\pi_r$ are diffeomorphisms onto $\Hyp^2$ and $\Phi=\pi_r\circ(\pi_l)^{-1}$ is an area-preserving diffeomorphism of $\Hyp^2$ which extends to $\phi$ on $\partial\Hyp^2$.
\end{cor}

Let us now discuss some properties of the associated map $\Phi$ to a smooth convex embedded surface $S$. In fact, we will prove the following proposition:

\begin{prop} \label{prop properties associated map}
Let $S$ be a smooth strictly convex embedded spacelike surface in $\AdS^3$, and let $\Phi=\pi_r\circ(\pi_l)^{-1}:\Omega_l\to\Omega_r$ be the associated area-preserving map. Then there exists a smooth tensor $b\in\Gamma(\mathrm{End}(T\Omega_l))$ such that
\begin{equation} \label{eq pullback b}
\Phi^*g_{\Hyp^2}=g_{\Hyp^2}(b\cdot,b\cdot)\,,
\end{equation}
which satisfies:
\begin{enumerate}
\item $d^\nabla b=0\,$;
\item $\det b=1\,$;
\item $\tr b\in(-2,2)\,$.
\end{enumerate}
Here $\nabla$ is the Levi-Civita connection of $\Hyp^2$.
\end{prop}
\begin{proof}
Let us define 
\begin{equation} \label{eq:emy}
b=(E+J_I B)^{-1}(E-J_I B)\,.
\end{equation}
Then from Equations \eqref{eq pullback left} and \eqref{eq pullback right}, the condition in Equation \eqref{eq pullback b} is satisfied. Let us check the three properties of $b$:
\begin{enumerate}[leftmargin=*]
\item Observe that, since $\nabla^I E$ and $\nabla^I J_I$ vanish and $B$ is Codazzi for $I$, then $E-J_I B$ satisfies 
$$d^{\nabla^I}\!(E-J_I B)=0\,.$$
Using a formula given in \cite{labourieCP} or \cite[Proposition 3.12]{Schlenker-Krasnov}, the Levi-Civita connection of the metric $\pi_l^*g_{\Hyp^2}$ is 
$$\nabla_v w=(E+J_I B)^{-1}\nabla^I_v (E+J_I B)w\,,$$
for every $v,w$. Hence:
\begin{align*}
 d^{\nabla}&\!b(v, w)=\nabla_v b(w)-\nabla_w b(v)-b[v,w] \\
 &=(E+J_I B)^{-1}(\nabla^I_v(E-J_I B)(w)-\nabla^I_w(E-J_I B)(v)-(E-J_I B)[v,w]) \\
 &=(E+J_I B)^{-1}(d^{\nabla^I}\!(E-J_I B)(v,w))=0\,.
\end{align*}
 This concludes the first point. 
\item Since $B$ is self-adjoint for $I$, then $\tr J_I B=0$, and thus $$\det(E+J_I B)=\det(E-J_I B)=1+\det B\,.$$ 
Therefore $\det b=1$.
\item Observe that 
\begin{equation} \label{eq:ely}
(E+J_I B)^{-1}=\frac{1}{1+\det B}(E-J_I B)\,.
\end{equation}
In fact, $(E+J_I B)(E-J_I B)=E-(J_I B)^2$ and, since $\tr (J_I B)=0$, by the Cayley-Hamilton Theorem $(J_I B)^2=-\det(J_I B)E=-(\det B)E$. Moreover by a direct computation, $\mathrm{tr}(E-J_I B)^2=\tr E+\tr(JB)^2=2(1-\det B)$ and thus:

\begin{equation} \label{eq: gambadilegno}
\tr b=\tr(E+J_I B)^{-1}(E-J_I B)=\frac{1}{1+\det B}\tr(E-J_I B)^2=2\left(\frac{1-\det B}{1+\det B}\right)\,.
\end{equation}

Since $\det B>0$ by hypothesis, $\tr b$ is always contained in $(-2,2)$ for a strictly convex surface. \qedhere
\end{enumerate}
\end{proof}

\begin{remark} \label{rmk:picodepaperis}
Using Equation \eqref{eq:emy} and $J=(E+J_I B)^{-1} J_I (E+J_I B)$, we get
$$Jb=J (E+J_I B)^{-1}(E-J_I B)=(E+J_I B)^{-1}J_I(E-J_I B)\,.$$
By applying again Equation \eqref{eq:ely}, one obtains that
$$\mathrm{tr}(Jb)=\frac{\mathrm{tr}(J_I(E-J_I B)^2)}{1+\det B}=\frac{2\mathrm{tr}B}{1+\det B}\,.$$
Since we have chosen the future-directed normal vector field to the surface $S$, this shows that $S$ is future-convex if $\mathrm{tr}(Jb)>0$, and past-convex if $\mathrm{tr}(Jb)<0$.
\end{remark}

\subsection{From $K$-surfaces to $\theta$-landslides}

Recall that a smooth spacelike surface in $\AdS^3$ is a \emph{K}-surface if its Gaussian curvature is constantly equal to $K$. In this paper we are interested in strictly convex surfaces, hence we will consider $K$-surfaces with $K\in(-\infty,-1)$, so that by the Gauss equation
$$K_I=-1-\det B\,,$$
we will have $\det B>0$.

In \cite{bms} it is proved that, given a $K$-surface in $\AdS^3$, for $K\in(-\infty,-1)$, the associated map $\Phi$ is a $\theta$-\emph{landslide}. In \cite{bms} the definition of $\theta$-landslides is the following:

\begin{defi} \label{defi landslide}
Let $\Phi:\Omega\subseteq \Hyp^2\to\Hyp^2$ be an orientation-preserving diffeomorphism onto its image and let $\theta\in[0,\pi]$. We say that $\Phi$ is a $\theta$-landslide if 
there exists a bundle morphism $m\in\Gamma(\mathrm{End}(T\Omega))$ such that
$$\Phi^*g_{\Hyp^2}=g_{\Hyp^2}((\cos\theta E+\sin\theta Jm)\cdot,(\cos\theta E+\sin\theta Jm)\cdot)\,,$$
and
\begin{itemize}
\item $d^{\nabla}m=0\,$;
\item $\det m=1\,;$
\item $m$ is positive self-adjoint for $g_{\Hyp^2}\,$.
\end{itemize}
\end{defi}

Let us observe that a $\theta$-landslide with $\theta=0,\pi$ is an isometry. In \cite{bms}
it was proved that a $\theta$-landslide can be decomposed as $\Phi=f_2\circ (f_1)^{-1}$, where $f_1$ and $f_2$ are harmonic maps from a fixed Riemann surface, with Hopf differentials satisfying the relation
$$\mathrm{Hopf}(f_1)=e^{2i \theta }\mathrm{Hopf}(f_2)\,.$$
We will restrict to the case of $\theta\in (0,\pi)$. 
In this case, we will mostly use an equivalent characterization of landslides:

\begin{lemma} \label{char landslide}
Let $\Phi:\Omega\subseteq \Hyp^2\to\Hyp^2$ be an orientation-preserving diffeomorphism onto its image. Then $\Phi$ is a $\theta$-landslide if and only if
there exists a bundle morphism $b\in\Gamma(\mathrm{End}(T\Omega))$ such that
$$\Phi^*g_{\Hyp^2}=g_{\Hyp^2}(b\cdot,b\cdot)\,,$$
and
\begin{itemize}
\item $d^{\nabla}b=0\,$;
\item $\det b=1\,;$
\item $\tr b=2\cos\theta\,$;
\item $\tr Jb< 0$.
\end{itemize}
\end{lemma}
\begin{proof}
Let us first suppose that $\Phi$ is a $\theta$-landslide and prove the existence of $b$. 
Let $b=\cos\theta E+\sin\theta Jm$, for $m$ as in Definition \ref{defi landslide}. Then $\tr b=2\cos\theta$ since $m$ is self-adjoint for $g_{\Hyp^2}$, and thus $Jm$ is traceless. On the other hand, $\mathrm{tr}(Jb)=-\sin\theta\,\mathrm{tr}\, m<0$ since $m$ is positive definite.  Moreover, $\det b=\cos^2\theta+\sin^2\theta=1$ since $\det E=\det J=\det m=1$ and again $\mathrm{tr}(Jm)=0$. Finally, $d^{\nabla}b=0$ since $\nabla E=\nabla J=0$.

Conversely, let $b$ satisfy the conditions in the statement. Then $b-\cos\theta E$ is traceless and thus is of the form $a Jm$, with $m$ positive self-adjoint for $g_{\Hyp^2}$, $\det m=1$, and $a\in\R$. Since $d^{\nabla}b=0$, one has $d^{\nabla}m=0$. Finally, $1=\det b=\cos^2\theta+a^2(\det m)=\cos^2\theta+a^2$ implies that $|a|=|\sin\theta|$. Imposing that $\tr Jb<0$, we conclude that $a=\sin\theta$.
\end{proof}

Hence we are now able to reprove the following fact, which is known from \cite{bms}:

\begin{prop} \label{prop past-convex}
Let $S$ be a past-convex $K$-surface in $\AdS^3$ with $K\in(-\infty,-1)$. Then the associated map $\Phi:\Omega_l\to\Omega_r$ is a $\theta$-landslide, with
$$K=-\frac{1}{\cos^2(\frac{\theta}{2})}\,.$$
\end{prop}
\begin{proof}
The proof follows, using the characterization of Lemma \ref{char landslide}, from the choice 
$$b=(E+J_I B)^{-1}(E-J_I B)$$
in Proposition \ref{prop properties associated map}. In fact, the first two points are satisfied, and from Equation \eqref{eq: gambadilegno} one obtains:
\begin{equation} \label{eq:tipetap}
\tr b=2\left(\frac{1-\det B}{1+\det B}\right)=2\cos\theta\,,
\end{equation}
for some $\theta\in(0,\pi)$. As observed in Remark \ref{rmk:picodepaperis}, $\tr(Jb)<0$ since $S$ is past-convex by hypothesis.
Finally, from Equation \eqref{eq:tipetap} one obtains:
$$\det B=\frac{1-\cos\theta}{1+\cos\theta}={\tan^2({\theta}/{2})}$$
and thus
$$K=-1-\det B=-\frac{1}{\cos^2({\theta}/{2})}\,,$$
thus concluding the claim.
\end{proof}

\begin{remark} \label{remark dual surface2}
There clearly is an analogous statement of Proposition \ref{prop past-convex} for future-convex surfaces. In fact, if $S$ is future-convex $K$-surface, consider first the tensor $b'=(E+J_I B)^{-1}(E-J_I B)$ as in the proof of Proposition \ref{prop past-convex}. In this case $\tr(Jb')>0$, as observed in Remark \ref{rmk:picodepaperis}. 

Hence it is the choice $b=-(E+J_I B)^{-1}(E-J_I B)$ which makes the conditions in Lemma \ref{char landslide} satisfied. This shows that the associated map $\Phi$ is a $\theta$-landslide where $\theta\in(0,\pi)$ is chosen so that $2\cos\theta=-\tr((E+J_I B)^{-1}(E-J_I B))$. The same computation as above shows that the curvature of $S$ is:
$$K=-\frac{1}{\sin^2(\theta/2)}\,.$$

Let us remark that the result of this computation is consistent with the fact that the dual surface of the $K$-surface $S$ is a past-convex surface $S^*$, gives the same associated map $\Phi$, and has constant curvature
$$K^*=-\frac{1}{\cos^2(\theta/2)}~.$$
In light of this remark, we will always restrict ourselves to consider past-convex $K$-surfaces.
\end{remark}

A special case is obtained for $\theta=\pi/2$. Applying Lemma \ref{char landslide}, a $\pi/2$-landslide is a diffeomorphism
 $\Phi:\Omega\subseteq \Hyp^2\to\Phi(\Omega)\subseteq\Hyp^2$ such that $\Phi^*g_{\Hyp^2}=g_{\Hyp^2}(b\cdot,b\cdot)$, where:
\begin{itemize}
\item $d^{\nabla}b=0\,;$
\item $\det b=1\,;$
\item $\tr b=0\,;$
\item $\tr Jb<0\,.$
\end{itemize}
Hence, by taking $b_0=-Jb$, we have $\tr(Jb_0)=0$, and thus $b_0$ has the property that $\Phi^*g_{\Hyp^2}=g_{\Hyp^2}(b_0\cdot,b_0\cdot)$ with:
\begin{itemize}
\item $d^{\nabla}b_0=0\,;$
\item $\det b_0=1\,;$
\item $b_0$ is positive self-adjoint for $g_{\Hyp^2}\,$.
\end{itemize}
This is known to be equivalent to $\Phi$ being a \emph{minimal Lagrangian map}, namely $\Phi$ is area-preserving and its graph is a minimal surface in $\Hyp^2\times\Hyp^2$. 

\begin{remark}
In \cite{bon_schl}, the existence of maximal surface in $\AdS^3$ (namely, such that $\tr B=0$ where $B$ is the shape operator) was used to prove that every quasisymmetric homeomorphism of $\partial \Hyp^2$ admits a unique minimal Lagrangian extension to $\Hyp^2$. We remark that, given a smooth maximal surface in $\AdS^3$ with principal curvatures in $(-1,1)$, the two $\pi/4$-parallel surfaces are $(-2)$-surfaces (see \cite{bon_schl} or \cite{Schlenker-Krasnov}). As already discussed in Remarks \ref{remark normal flow} and \ref{remark dual surface}, the map associated to the two parallel surfaces (which are dual to one another) is the same as the map associated to the original maximal surface.
\end{remark}

\section{A representation formula for convex surfaces} \label{sec representation formula}

The purpose of this section is to provide a \emph{representation formula} for convex surfaces in $\AdS^3$ in terms of the associated diffeomorphism 
$\Phi:\Omega\subseteq\Hyp^2\to\Phi(\Omega)\subseteq\Hyp^2$, for $\Omega$ a domain in $\Hyp^2$. 
A basic observation to start with is the fact that, if $S$ is a smooth strictly convex embedded surface, with associated map
 $\Phi=\pi_r\circ(\pi_l)^{-1}$, and if $\gamma\in S$ is such that $\pi_l(\gamma)=x$, then $\gamma$ belongs to the timelike geodesic
$$L_{x,\Phi(x)}=\{\eta\in\isom(\Hyp^2):\eta(\Phi(x))=x \}\,.$$ 
In fact, from Subsection \ref{subsec: projections}, we know $L_{x,\Phi(x)}$ is precisely the timelike geodesic orthogonal to $S$ at $\gamma$.

Motivated by this observation, we will now  define 
 a 1-parameter family of parallel surfaces $S_{\Phi,b}$, all orthogonal to the lines $L_{x,\Phi(x)}$, which depend on the diffeomorphism $\Phi$ and on the choice of $b\in\Gamma(\mathrm{End}(T\Omega))$ satisfying the necessary conditions provided by Proposition \ref{prop properties associated map}.
  The surface $S_{\Phi,b}$ will be parameterized by the inverse map of the left projection $\pi_l:S\to\Hyp^2$. 
 

\begin{defi} \label{defi sigma}
Let $\Omega$ be a domain in $\Hyp^2$ and $\Phi:\Omega\subseteq\Hyp^2\to\Phi(\Omega)\subseteq\Hyp^2$ be an orientation-preserving diffeomorphism such that
$$\Phi^*g_{\Hyp^2}=g_{\Hyp^2}(\bb\cdot,\bb\cdot)\,,$$
where $\bb\in\Gamma(\mathrm{End}(T\Omega))$ is a smooth bundle morphism such that $d^\nabla\! \bb=0$ and $\det \bb=1$. Then define the map
$\sigma_{\Phi,\bb}:\Hyp^2\to\isom(\Hyp^2)$, where $\sigma_{\Phi,\bb}(x)$ is the unique isometry $\sigma$ such that
\begin{itemize}
\item $\sigma(\Phi(x))=x\,$;
\item $d\sigma_{\Phi(x)}\circ d\Phi_x=-\bb_x\,$.
\end{itemize}
\end{defi}

\begin{example} 
If $\Phi$ is the identity of $\Hyp^2$ and $\cos\theta=1$, then clearly the identity operator $\bb=E$ satisfies the conditions $d^\nabla\!\bb=0$ and $\det\bb=1$. Hence by applying the definition, $\sigma_{\mathrm{id},E}(x)$ is the unique isometry which fixes $x$ and has differential $-E$. In conclusion, $\sigma_{\mathrm{id},E}:\Hyp^2\to\isom(\Hyp^2)$ is our usual identification which maps $x\in\Hyp^2$ to $\mathcal{I}_x$.
\end{example}

\begin{remark} \label{remark parallel immersion}
Clearly the smooth tensor $b\in\Gamma(\mathrm{End}(T\Omega))$ is not uniquely determined. In fact, if $\bb'$ is another tensor satisfying 
$$\Phi^*g_{\Hyp^2}=g_{\Hyp^2}(\bb'\cdot,\bb'\cdot)\,,$$
with $\det \bb'=1$, then $\bb'_x=R^x_{\rho(x)}\bb_x$, where $\rho$ is a function and $R^x_{\rho}$ is the rotation of $\rho$ around the basepoint $x\in\Hyp^2$. Then we claim that $d^\nabla \bb'=R_\rho (d\rho \wedge Jb+d^\nabla\bb)$. Indeed, using that $\nabla R_\rho=d\rho\otimes JR_\rho$ we get

\begin{equation} \label{eq:etabeta}
\begin{split}
d^\nabla (R_\rho b)(v,w)&=\nabla_v (R_\rho b)(w)-\nabla_w (R_\rho b)(v) \\
&=(\nabla_v R_\rho)\circ b(w)+R_\rho\circ(\nabla_v b)(w)-(\nabla_w R_\rho)\circ b(v)-R_\rho\circ(\nabla_w b)(v) \\
&= d\rho(v)JR_\rho b(w) -d\rho(w)JR_\rho b(v)+R_\rho ((\nabla_v b)(w)-(\nabla_w b)(v)) \\
&=R_\rho (d\rho(v)J b(w) -d\rho(w)J b(v)+d^\nabla b(v,w))\,,
\end{split}
\end{equation}

Therefore, from the condition $d^\nabla\! \bb=0$, if $\bb'$ is another tensor as in the hypothesis of Definition \ref{defi sigma}, then (up to replacing $\rho$ with $2\rho$), $\bb'=R^x_{2\rho}\bb$, where $\rho$ is a constant. 

In this case, one gets that $\sigma:=\sigma_{\Phi,\bb}(x)$ satisfies the condition
$d\sigma_{\Phi(x)}=-\bb_x\circ (d\Phi_{x})^{-1}$, while $\sigma':=\sigma_{\Phi,\bb'}(x)$ satisfies
$$d\sigma'_x=- \bb'\circ(d\Phi_{x})^{-1}=-R^x_{2\rho}\circ\bb\circ (d\Phi_x)^{-1}\,.$$
In other words, the map $\sigma_{\Phi,\bb'}$ is obtained by displacing the map $\sigma_{\Phi,\bb}$ along the timelike lines $L_{x,\Phi(x)}$ of a constant length $\rho$.
\end{remark}

The next lemma discuss the naturality of the construction of the map $\sigma_{\Phi,\bb}$. 

\begin{lemma} \label{lemma composition rule}
Let $b\in\Gamma(\mathrm{End}(T\Omega))$ be a smooth bundle morphism such that 
 $$\Phi^*g_{\Hyp^2}=g_{\Hyp^2}(\bb\cdot,\bb\cdot)\,,$$
with $d^\nabla\! \bb=0$ and $\det \bb=1$. Let $\alpha,\beta$ be isometries of $\Hyp^2$. Then:
\begin{itemize}
\item The smooth tensor $\bb'$  on $\beta(\Omega)$ defined by $\bb'_{\beta x}=d\beta_{x}\circ\bb_{x}\circ (d\beta_x)^{-1}$ is such that $(\alpha\Phi\beta^{-1})^*g_{\Hyp^2}=g_{\Hyp^2}(b'\cdot,b'\cdot)$, $d^\nabla\! \bb'=0$ and $\det b'=1$;
\item
$\sigma_{\alpha\circ\Phi\circ\beta^{-1},\bb'}(\beta(x))=\beta\circ\sigma_{\Phi,\bb}(x)\circ\alpha^{-1}=(\beta,\alpha)\cdot \sigma_{\Phi,\bb}(x)\,.$
\end{itemize}
\end{lemma}
\begin{proof}
For the first point, since $\beta$ is an isometry, one has:
\begin{align*}
(\alpha\Phi\beta^{-1})^*g_{\Hyp^2}(v,w)&=g_{\Hyp^2}(d(\alpha\Phi\beta^{-1})(v),d(\alpha\Phi\beta^{-1})(w))\\
&=g_{\Hyp^2}(d(\Phi\beta^{-1})(v),d(\Phi\beta^{-1})(w))=g_{\Hyp^2}(\bb'(v),\bb'(w))\,.
\end{align*}
Moreover,
\begin{align*}
d^\nabla\!\bb'&=\nabla_v (d\beta\bb \,d\beta^{-1}(w))-\nabla_w (d\beta\bb\, d\beta^{-1}(v))-d\beta\bb\, d\beta^{-1}[v,w] \\
&=d\beta \left(\nabla_{d\beta^{-1}(v)}(\bb\, d\beta^{-1}(w))-d\beta\nabla_{d\beta^{-1}(w)}(\bb\, d\beta^{-1}(v))-\bb[d\beta^{-1}(v),d\beta^{-1}(w)]\right) \\
&=d\beta d^\nabla\!\bb(d\beta^{-1}(v),d\beta^{-1}(w))=0\,.
\end{align*}
For the second point, we must check that $\hat\sigma=\beta\circ\sigma_{\Phi,\bb}(x)\circ\alpha^{-1}$ satisfies the two properties defining $\sigma_{\alpha\circ\Phi\circ\beta^{-1},b'}(\beta(x))$, namely  $\hat\sigma(\alpha\Phi\beta^{-1}(\beta(x)))=\beta(x)$ and 
$d\hat\sigma_{\alpha\Phi(x)}\circ d(\alpha\Phi\beta^{-1})_{\beta(x)}=-\bb'_{\beta(x)}$.

 For the first defining property, we have that for every $x\in\Omega$:
$$(\beta\circ\sigma_{\Phi,\bb}(x)\circ\alpha^{-1})(\alpha\circ\Phi\circ\beta^{-1})(\beta(x))=\beta(\sigma_{\Phi,\bb}(x) \Phi(x))=\beta(x)\,.$$
For the second property, 
\begin{align*}
d(\beta\sigma_{\Phi}(x)\alpha^{-1})_{\alpha\Phi(x)}\circ d(\alpha\Phi\beta^{-1})_{\beta x}&=d\beta_x \circ d(\sigma_{\Phi}(x))_{\Phi(x)}\circ d\Phi_{x}\circ d\beta^{-1}_{\beta (x)} \\
&=-d\beta_x \circ\bb_{x}\circ (d\beta_x)^{-1}=
-\bb'_{\beta(x)}\,.
\end{align*}
This concludes the proof.
\end{proof}
Let us now compute the pull-back of the induced metric on the surface $S_{\Phi,\bb}=\sigma_{\Phi,\bb}(\Omega)$ by means of $\sigma_{\Phi,\bb}$.

\begin{prop} \label{prop first ff}
Let $\Omega$ be a domain in $\Hyp^2$ and $\Phi:\Omega\subseteq\Hyp^2\to\Phi(\Omega)\subseteq\Hyp^2$ be an orientation-preserving diffeomorphism such that
$$\Phi^*g_{\Hyp^2}=g_{\Hyp^2}(\bb\cdot,\bb\cdot)\,,$$
where $\bb\in\Gamma(\mathrm{End}(T\Omega))$ is a smooth bundle morphism such that $d^\nabla\! \bb=0$ and $\det \bb=1$. Then
\begin{equation} \label{firstff}
(\sigma_{\Phi,\bb})^* g_{\AdS^3}=\frac{1}{4} g_{\Hyp^2}((E+\bb)\cdot,(E+\bb)\cdot)
\,.
\end{equation}
\end{prop}
\begin{proof}
In this proof, we will consider $\Hyp^2$ in the hyperboloid model (recall Subsection \ref{subsec models}), so that $\isom(\Hyp^2)$ is identified to the connected component of the identity in $\SO(2,1)$.

Let us fix a point $x\in\Hyp^2$. By post-composing $\Phi$ with an isometry $\alpha$ (where actually $\alpha=\sigma_{\Phi,b}(x)$), we can assume that $\Phi(x)=x$ and that $d\Phi_x=-\bb$. Indeed, by Lemma \ref{lemma composition rule}, $\sigma_{\Phi,\bb}$ is modified by right multiplication with $\alpha^{-1}$ and thus the embedding data are unchanged. Therefore by Definition \ref{defi sigma}, $\sigma_{\Phi,\bb}(x)=\mathrm{id}$ since $\sigma_{\Phi,\bb}(x)$ fixes $x$ and its differential at $x$ is the identity. 

Consider a smooth path $x(t)$ with $x(0)=x$ and $\dot x(0)=v$. Hence, regarding the isometries of $\Hyp^2$ as $\SO(2,1)$-matrices and points of $\Hyp^2$ as vectors in $\R^{2,1}$,
by differentiating at $t=0$ the relation
$$\sigma_{\Phi,\bb}(x(t))\cdot \Phi(x(t))=x(t)\,,$$
we get
$$d\sigma_{\Phi,\bb}(v)\cdot x- \bb(v)=v\,,$$
and therefore
\begin{equation} \label{relation dsigma1}
d\sigma_{\Phi,\rho}(v)\cdot x=(E+\bb)v\,.
\end{equation}
On the other hand, choose a vector $w\in T_x\Hyp^2$ and extend $w$ to a parallel vector field $w(t)$ along $x(t)$. We start from the following relation (recall that we consider $\isom(\Hyp^2)$ in the $\SO(2,1)$ model, hence the differential of an isometry $\sigma$ is $\sigma$ itself):
$$\sigma_{\Phi,\bb}(x(t))\cdot d\Phi_{x(t)}(w(t))=- \bb_{x(t)}(w(t))\,.$$
By differentiating this identity in $\R^{2,1}$, we obtain
\begin{equation} \label{equality differentiate}
-d\sigma_{\Phi,\bb}(v)\cdot (\bb_x w)+\ddt( d\Phi_{x(t)}(w(t)))=-\ddt (\bb_{x(t)}w(t))\,.
\end{equation}
Now observe that, if $\nabla$ is the Levi-Civita connection of $\Hyp^2$ and $\langle\cdot,\cdot\rangle_{\R^{2,1}}$ is the Minkowski product of $\R^{2,1}$, then
$$\ddt (\bb_{x(t)}w(t))=\nabla_v(\bb_{x(t)}w(t))+(\langle v,\bb (w)\rangle_{\R^{2,1}}) x\,.$$
For the other term, we obtain in a similar fashion that
$$\ddt( d\Phi_{x(t)}(w(t)))=\nabla_{d\Phi_x(v)}(d\Phi_{x(t)}w(t))
+(\langle \bb(v),\bb(w)\rangle_{\R^{2,1}}) x\,,$$
In the second term, we have used that 
$d\Phi_x=-\bb$. Since $\Phi$ is an isometry between $g_{\Hyp^2}$ and $\Phi^*g_{\Hyp^2}=g_{\Hyp^2}(\bb\cdot,\bb\cdot)$, we have
$$\nabla_{d\Phi_x(v)}(d\Phi_{x(t)}w(t))=d\Phi_x\widetilde \nabla_{v}(w(t))=-\bb_x \widetilde \nabla_{v}(w(t))\,,$$
where $\widetilde \nabla$ is the Levi-Civita connection of $\Phi^*g_{\Hyp^2}$. Using again the  formula given in \cite{labourieCP} or \cite[Proposition 3.12]{Schlenker-Krasnov}, it turns out that 
for any two vector fields $X,Y$
\begin{equation}
\widetilde \nabla_X Y=\bb^{-1}\nabla_X(\bb(Y))\,.
\end{equation}
Therefore, from Equation \eqref{equality differentiate} we obtain 
\begin{equation}
-d\sigma_{\Phi,\bb}(v)\cdot (\bb_x w)-\nabla_v(\bb_{x(t)}w(t))+(\langle \bb_x v,\bb_x w\rangle_{\R^{2,1}}) x=-\nabla_v(\bb_{x(t)}w(t))-(\langle v,\bb_x w\rangle_{\R^{2,1}}) x
\end{equation}
Therefore, replacing $b_x w$ with an arbitrary vector $w$, we obtain the relation
\begin{equation} \label{relation dsigma2}
d\sigma_{\Phi,\bb}(v)\cdot w=(\langle w,(E+\bb) v\rangle _{\R^{2,1}})x
\,.
\end{equation}
It is straightforward to check, using Equations \eqref{relation dsigma1} and \eqref{relation dsigma2}, that $d\sigma_{\Phi,\rho}(v)$ corresponds to the linear map 
\begin{equation} \label{eq differential sigma}
d\sigma_{\Phi,b}(v)(w)=J(E+\bb)v
\boxtimes w =\Lambda(J(E+\bb)(v))(w)\,,
\end{equation}
where $J$ is the almost-complex structure of $\Hyp^2$. Recalling from Equation \eqref{prop lambda 2} that the identification of $\so(2,1)$ with $\R^{2,1}$ sends the AdS metric of $\isom(\Hyp^2)$ 
to $1/4$ the Minkowski product, the first fundamental form is:
\begin{align*}
g_{\AdS^3}(d\sigma_{\Phi,\bb}(v_1),d\sigma_{\Phi,\bb}(v_2))&=\frac{1}{4}\langle J(E+\bb)v_1
,J(E+\bb)v_2
\rangle_{\R^{2,1}} \\
&=\frac{1}{4}\langle (E+\bb)v_1,(E+\bb)v_2\rangle_{\R^{2,1}}
\,,
\end{align*}
as in our claim.
\end{proof}

\begin{remark} \label{remark immersion trace condition} 
From the proof of Proposition \ref{prop first ff}, it turns out that, under the usual identification, the differential of $\sigma_{\Phi,\bb}$ is given by $J(E+\bb)$. Hence the map $\sigma_{\Phi,\bb}$ is an immersion provided
$$\det d \sigma_{\Phi,\bb}=2+\tr\bb\neq 0\,,$$
or equivalently $\tr\bb\neq -2$.
\end{remark}

\begin{cor} \label{cor orthogonal lines}
In the hypothesis of Proposition \ref{prop first ff}, the image of $d\sigma_{\Phi,b}$ at a point $x\in\Omega$ is orthogonal to the line $L_{x,\Phi(x)}$ at the point $\sigma_{\Phi,b}(x)$. In particular, if $x$ is not a critical point, then the family of lines $L_{\cdot,\Phi(\cdot)}$ foliates a neighborhood of $\sigma_{\Phi,b}(x)$ and the map $\sigma_{\Phi,b}$ gives a local parameterization of an integral surface of the orthogonal distribution. 

Moreover, if $\sigma_{\Phi,\bb}$ is an immersion, then $\pi_l\circ\sigma_{\Phi,\bb}$ is the identity, and $\pi_r\circ\sigma_{\Phi,\bb}=\Phi$.
\end{cor}
\begin{proof}
As in the proof of Proposition \ref{prop first ff}, we can assume that $\Phi(x)=x$ and that $d\Phi_x=-\bb$, which implies that $\sigma_{\Phi,b}(x)=\mathrm{id}$. From Equation \eqref{eq differential sigma}, it follows that the image of $d\sigma_{\Phi,b}$ at $x$ is orthogonal to 
the direction given by $\Lambda(x)$. 

On the other hand, in this assumption $L_{x,\Phi(x)}=L_{x,x}$ is the 1-parameter subgroup generated by $\Lambda(x)$ itself.
\end{proof}

\begin{remark} \label{rmk unit vector}
From the proof Proposition \ref{prop first ff} we have actually shown that, under the assumption $\Phi(x)=x$ and that $d\Phi_x=- \bb$, if $d\sigma_{\Phi,\bb}$ is non-singular at $x$, then under the identification $T_{\mathrm{id}}\isom(\Hyp^2)\cong\so(2,1)\cong\R^{2,1}$, the future unit normal vector at $\sigma_{\Phi,\bb}(x)=\mathrm{id}$ is $N(x)=2\Lambda(x)$.
In fact, $N(x)=2\Lambda(x)$ is orthogonal to the surface at $\sigma_{\Phi,\bb}(x)$ by the above observation, and is unit, since the $\AdS$ metric 
at $\mathrm{id}$ is identified to $1/4$ the Minkowski product, see Equation \eqref{prop lambda 2}. 
\end{remark}

\begin{prop} \label{prop sec ff}
Let $\Omega$ be a domain in $\Hyp^2$ and $\Phi:\Omega\subseteq\Hyp^2\to\Phi(\Omega)\subseteq\Hyp^2$ be an orientation-preserving diffeomorphism such that
$$\Phi^*g_{\Hyp^2}=g_{\Hyp^2}(\bb\cdot,\bb\cdot)\,,$$
where $\bb\in\Gamma(\mathrm{End}(T\Omega))$ is a smooth bundle morphism such that $d^\nabla\! \bb=0$ and $\det \bb=1$. Then the shape operator of the immersion $\sigma_{\Phi,\bb}$ is:
\begin{equation} \label{formula shape operator sigma}
B=-J_I(E+\bb)^{-1}(E-\bb)\,,
\end{equation}
where $J_I$ denotes the almost-complex structure associated to the induced metric $I$.
\end{prop}
\begin{proof}
To compute the covariant derivative of the normal vector field $N$ along a tangent vector $d\sigma_\Phi(v)$, we use the following formula (see \cite{milnor_curvature}):
$$\nabla^{\AdS^3}_{d\sigma_{\Phi,\bb}(v)} N=\nabla^l_{d\sigma_{\Phi,\bb}(v)} N+\frac{1}{2}[d\sigma_{\Phi,\bb}(v),N]=\nabla^r_{d\sigma_{\Phi,\bb}(v)} N-\frac{1}{2}[d\sigma_{\Phi,\bb}(v),N]\,,$$
where $\nabla^l$ (resp. $\nabla^r$) is the left-invariant (resp. right-invariant) flat connection which makes the left-invariant (resp. right-invariant) vector fields parallel, and $[\cdot,\cdot]$ denotes the bi-invariant extension of Lie bracket on $\mathfrak{isom}(\Hyp^2)$. 
Assuming that $\sigma_{\Phi,\bb}(x)=\mathrm{id}$ as before, we can compute
$\nabla^r_{d\sigma_{\Phi,\bb}(v)} N$ as the differential at the identity of the vector field $N\sigma_{\Phi,\bb}(x)^{-1}$. 
Observe that in the argument of Proposition \ref{prop first ff} and Remark \ref{rmk unit vector}, while doing the assumption that $\Phi(x)=x$ and that $d\Phi_x=-\bb$ (so that $\sigma_{\Phi,\bb}(x)=\mathrm{id}$), one is actually composing with an isometry of $\AdS^3$ of the form $(\mathrm{id},\alpha)$, where $\alpha=\sigma_{\Phi,\bb}(x)$. Thus the formula
$N(x)\sigma_{\Phi,\bb}(x)^{-1}=2\Lambda(x)$ is still true for every point $x$.
Recall that, under the identification of the Lie algebra $\so(2,1)$ with $\R^{2,1}$, the Lie bracket $[v,w]$ is identified to the Minkowski product $v\boxtimes w$ (see Equation \eqref{prop lambda 3}).  Therefore we get, using Equation \eqref{eq differential sigma},
$$\nabla^{\AdS^3}_{d\sigma_{\Phi,\bb}(v)} N=\Lambda\left(2v-\frac{1}{2}J(E+\bb)v\boxtimes 2x\right)=\Lambda\left(2v-v-\bb (v)\right)=\Lambda\left((E-\bb)v\right)\,,$$
where we have used Equation \eqref{almost-complex cross}. Therefore, using Equation \eqref{eq differential sigma}, that is, $d\sigma_{\Phi,\bb}(v)=\Lambda(J(E+\bb)v)$, one concludes that
$$B=-(E+\bb)^{-1}J(E-\bb)\,.$$
Using that, from Equation \eqref{firstff} for the metric, $J_I=(E+\bb)^{-1}J(E+\bb)$, one concludes the proof.
\end{proof}

Observe that, if $\Phi$ is a $\theta$-landslide, then by Lemma \ref{char landslide} we can choose $\bb$ as in the hypothesis, with moreover $\tr \bb=2\cos\theta$. Hence by Equation \eqref{formula shape operator sigma}, the curvature of $\sigma_{\Phi,\bb}(\Omega)$ is
$$\det B=\frac{2-\tr\bb}{2+\tr\bb}=\frac{1-\cos\theta}{1+\cos\theta}=\tan^2\left(\frac{\theta}{2}\right)\,,$$
and thus $\sigma_{\Phi,\bb}$ is a $K$-surface for $K=-1/\cos^2(\theta/2)$.

More generally, we can now prove:

\begin{theorem}
Let $S$ be a smooth strictly convex embedded spacelike surface in $\AdS^3$, and let $\Phi=\pi_r\circ(\pi_l)^{-1}:\Omega_l\to\Omega_r$ be the associated map. By identifying $g_{\Hyp^2}$ with the metric $I((E+J_I B)\cdot,(E+J_I B)\cdot)$ by means of the left projection $\pi_l$, and choosing $$\bb=(E+J_I B)^{-1}(E-J_I B)\,,$$
then the map $\sigma_{\Phi,\bb}:\Omega_l\to\AdS^3$ defined by $\sigma_{\Phi,\bb}(x)=\sigma$, with:
\begin{itemize}
\item $\sigma(\Phi(x))=x\,$;
\item $d\sigma_{\Phi(x)}\circ d\Phi_{x}=-\bb_x\,$.
\end{itemize}
is an embedding whose image is the original surface $S$.
\end{theorem}
\begin{proof}
Recall from Proposition \ref{prop properties associated map} that $\bb$ satisfies the hypothesis $d^{\nabla}\bb=0$ and $\det\bb=1$. Moreover, since $S$ is strictly convex, then $\tr\bb\in(-2,2)$ and thus $\sigma_{\Phi,\bb}$ is an immersion by Remark \ref{remark immersion trace condition}. We will check that the embedding data of $\sigma_{\Phi,\bb}\circ \pi_l$ coincide with those of $S$, and thus in principle $\sigma_{\Phi,\bb}\circ \pi_l$ extends to a global isometry $(\gamma_l,\gamma_r)$ of $\AdS^3$. By Corollary \ref{cor orthogonal lines}, if $\pi_l(\gamma)=x$, then the point $\sigma_{\Phi,\bb}(x)$ belongs to the geodesic $L_{x,\Phi(x)}$ which is orthogonal to $S$ at $\gamma$, and is also orthogonal to the image of $\sigma_{\Phi,\bb}$. So we deduce that $(\gamma_l,\gamma_r)$ preserves every geodesic of the form $L_{x,\Phi(x)}$. Hence $\gamma_l$ fixes every point $x\in\Omega_l$, and analogously $\gamma_r$ fixes every point $\Phi(x)$. Therefore $\gamma_l=\gamma_r=\mathrm{id}$ and this will conclude the proof.

To check the claim of the embedding data, using Equation \eqref{firstff},
$$
\pi_l^* \sigma_{\Phi,\bb}^*(g_{\AdS^3})=\frac{1}{4}\pi_l^*g_{\Hyp^2}((E+\bb)\cdot,(E+\bb)\cdot)=\frac{1}{4}I((E+J_I B)(E+\bb)\cdot,(E+J_I B)(E+\bb)\cdot)\,.
$$
On the other hand,
\begin{align*}
(E+J_I B)(E+\bb)&=(E+J_I B)(E+(E+J_I B)^{-1}(E-J_I B)) \\
&=(E+J_I B)(E+J_I B)^{-1}(E+J_I B+E-J_I B)=2E\,.
\end{align*}
In a similar way, according to Equation \eqref{formula shape operator sigma}, the shape operator of $\sigma_{\Phi,\bb}\circ\pi_l$ is:
\begin{align*}
-J_I(E+\bb)^{-1}(E-\bb)&=J_I(E+(E+J_I B)^{-1}(E-J_I B))^{-1}(E-(E+J_I B)^{-1}(E-J_I B)) \\
&=J_I(2(E+J_I B)^{-1})^{-1}(-2(E+J_I B)^{-1}J_I B)=-J_I^{2}B\,,
\end{align*}
and thus it coincides with $B$, the shape operator of the original surface $S$. This concludes the proof.
\end{proof}

As already observed, the condition that $\tr\bb$ is a constant in $(-2,2)$ coincides with the condition that $S$ has constant curvature in $(-\infty,-1)$. Hence we obtain the following corollary:

\begin{cor} \label{cor reconstruct ksurf}
Let $\Omega$ be a  domain in $\Hyp^2$ and $\Phi:\Omega\subseteq\Hyp^2\to\Hyp^2$ be a $\theta$-landslide with $\theta\in(0,\pi)$. If $b\in\Gamma(\mathrm{End}(T\Omega))$ satisfies
$$\Phi^*g_{\Hyp^2}=g_{\Hyp^2}(b\cdot,b\cdot)\,,$$
with $d^\nabla b=0$, $\det b=1$, $\tr b=2\cos\theta$ and $\tr Jb<0$, then the map $\sigma_{\Phi,b}$ defined by:
\begin{itemize}
\item $\sigma(\Phi(x))=x\,$;
\item $d\sigma_{\Phi(x)}\circ d\Phi_{x}=-\bb_x\,$.
\end{itemize}
is the embedding of the past-convex $K$-surface in $\AdS^3$ (for $K=-1/\cos^2(\theta/2)$) such that $\pi_l\circ \sigma_{\Phi,b}=\mathrm{id}$ and $\pi_r\circ \sigma_{\Phi,b}=\Phi$.
\end{cor}

\begin{remark}
Applying Remark \ref{remark parallel immersion}, the map $\sigma_{\Phi,-b}$, which is defined by:
\begin{itemize}
\item $\sigma(\Phi(x))=x\,$;
\item $d\sigma_{\Phi(x)}\circ d\Phi_{x}=\bb_x\,$.
\end{itemize}
corresponds to the dual embedding, which is a $K^*$-surface, with $K^*=-K/(K+1)$, whose associated landslide map $\Phi$ is the same.
\end{remark}

Although this is not the main point of this paper, we observe that this representation formula holds also for maximal surfaces. In fact, if $\Phi$ is a minimal Lagrangian map (which is the same as a $(\pi/2)$-landslide, as already observed), then there exists a smooth $b_0\in\Gamma(\mathrm{End}(T\Omega))$ such that $\Phi^*g_{\Hyp^2}=g_{\Hyp^2}(b_0\cdot,b_0\cdot)$ with
\begin{itemize}
\item $d^{\nabla}b_0=0\,;$
\item $\det b_0=1\,;$
\item $b_0$ is self-adjoint for $g_{\Hyp^2}\,$.
\end{itemize}
Moreover, one can assume $b_0$ is positive definite (up to replacing $b_0$ with $-b_0$). The condition that $b_0$ is positive definite ensures that $\tr b_0\neq -2$, and thus $\sigma_{\Phi,b_0}$ is an immersion.  Moreover, by a direct computation (see \cite[§3]{bon_schl}), $B$ is traceless, that is, $\sigma_{\Phi,b_0}(\Omega)$ is a maximal surface. In fact, in \cite{bon_schl} the existence of maximal surfaces was proved in order to obtain results on the existence of minimal Lagrangian extensions. Our next corollary goes in the opposite direction.

\begin{cor}
Let $\Omega$ be a open domain in $\Hyp^2$ and $\Phi:\Omega\subseteq\Hyp^2\to\Hyp^2$ be a minimal Lagrangian map such that
$$\Phi^*g_{\Hyp^2}=g_{\Hyp^2}(b_0\cdot,b_0\cdot)\,,$$
so that $b_0$ is positive definite, $g_{\Hyp^2}$-self-adjoint, $d^\nabla b_0=0$ and $\det b_0=1$. Then the map $\sigma_{\Phi,b_0}$ defined by:
\begin{itemize}
\item $\sigma(\Phi(x))=x\,$;
\item $d\sigma_{\Phi(x)}\circ d\Phi_{x}=-(\bb_0)_x\,$.\end{itemize}
is the embedding of the maximal surface in $\AdS^3$  such that $\pi_l\circ \sigma_{\Phi,b_0}=\mathrm{id}$ and $\pi_r\circ \sigma_{\Phi,b_0}=\Phi$. 
\end{cor}

In fact, as it is to be expected, considering $b=\pm Jb_0$ one recovers the $K$-surfaces which are $(\pi/4)$-parallel surfaces to the maximal surface.


\section{The explicit construction of a barrier} \label{sec barrier}
In this section we will construct an explicit example of $\theta$-landslides between hyperbolic surfaces, which commute with a 1-parameter hyperbolic group of isometries.

We will use the upper half-plane model of $\Hyp^2$, and will denote by $z=x+iy$ the standard coordinates of the upper half-plane. Let us introduce a new coordinate $w=s+it$, with $w\in \R\times (-\pi/2,\pi/2)$, defined by $$z=i\exp(w)\,.$$ (See also \cite[Chapter 2]{hubbardbook}.) 
Clearly $w$ is a conformal coordinate, and the hyperbolic metric takes the form
\begin{equation} \label{eq:edera}
h_0=\frac{|dz|^2}{y^2}=\frac{|dw|^2}{\cos^2(t)}\,.
\end{equation}
In this coordinates, the line $l=\{t=0\}$ is a geodesic with endpoints $0,\infty\in \RP^1$. The isometries preserving the geodesic $l$ have the form $\gamma_a(s,t)=(s+a,t)$. Moreover, the lines $\{s=s_0\}$ are geodesics orthogonal to $l$. See Figure \ref{fig:coordinateshp}.

\begin{figure}[htbp]
\centering
\includegraphics[height=4cm]{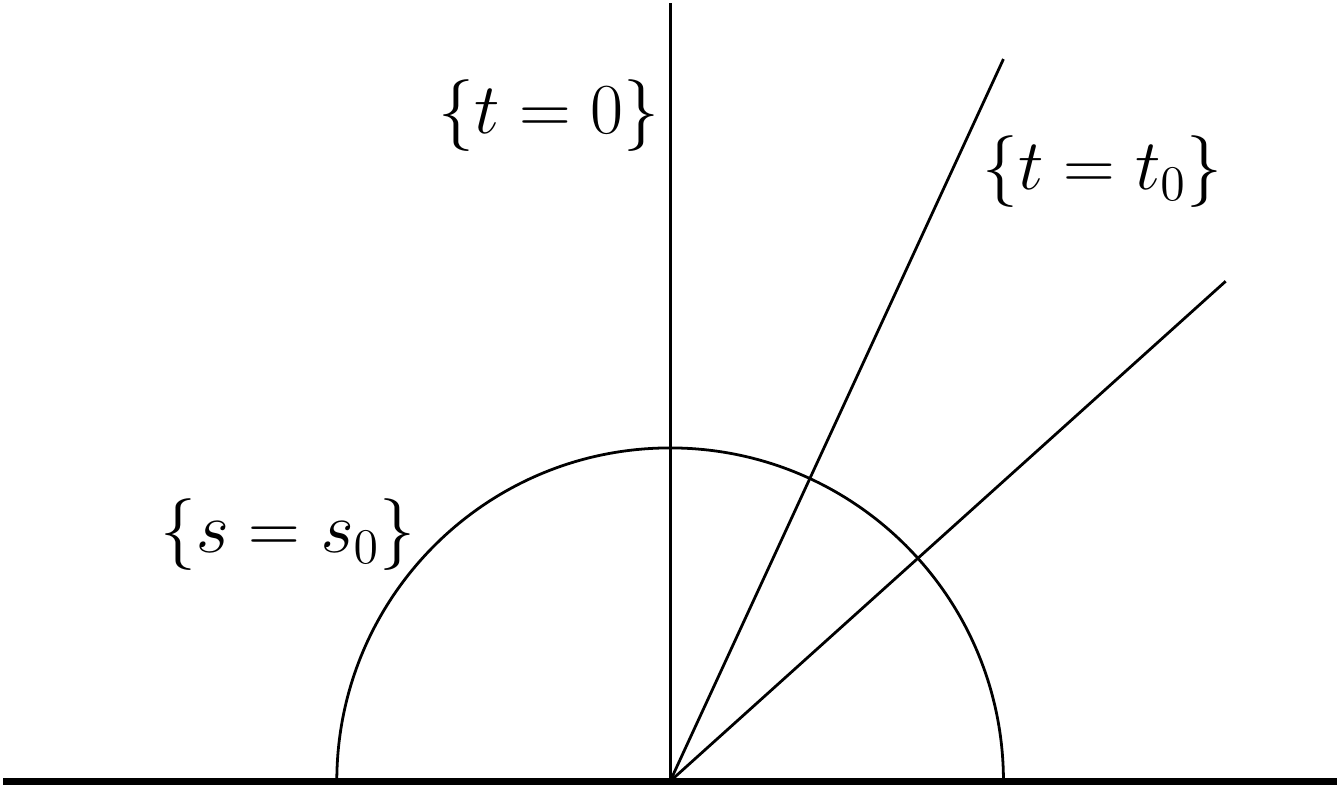}
\caption{In the $w=s+it$ on the upper half-plane, the curves $\{t=t_0\}$ are equidistant curves from the geodesic $\{t=0\}$. The curves $\{s=s_0\}$ are geodesics orthogonal to $\{t=0\}$. \label{fig:coordinateshp}}
\end{figure}

We will look for a $\theta$-landslide $\Phi_\theta$ invariant for the 1-parameter group $\gamma_a$. More precisely, we require that $\Phi_\theta\circ\gamma_a=\gamma_a\circ\Phi_\theta$. Hence $\Phi_\theta$ will necessarily have the form
$$\Phi_\theta(s,t)=(s-\ph(t),\psi(t))\,.$$
We want to impose that  $\Phi_\theta$ is an orientation-preserving diffeomorphism, hence we will assume $\psi$ increasing. Thus we will always consider $\psi'>0$. 

We will make use of the following remark:
\begin{remark} \label{rmk varpi}
Given an orientation-preserving local diffeomorphism $\Phi:\Omega\subseteq \Hyp^2\to\Hyp^2$, there exists a unique bundle morphism $b_0\in\Gamma(\mathrm{End}(T\Hyp^2))$ with
$$\Phi^*g_{\Hyp^2}=g_{\Hyp^2}(b_0\cdot,b_0\cdot)\,,$$
such that 
 $b_0$ is positive definite and self-adjoint for $g_{\Hyp^2}$ (which is equivalent to $\mathrm{tr}Jb_0=0$).
Indeed, $b_0$ is the square root of $(g_{\Hyp^2})^{-1}\Phi^*g_{\Hyp^2}$.  In general, $\Phi^*g_{\Hyp^2}$ is equal to $g_{\Hyp^2}(b\cdot,b\cdot)$, for a (smooth) bundle morphism $b$, if and only if $b$ is of the form $R_{\rho}b_0$ for some smooth function $\rho$.

 \end{remark}

Hence, let us begin by computing
\begin{align*}
\Phi_\theta^*h_0=&\frac{1}{\cos^2(\psi(t))}\left((ds-\ph'(t)dt)^2+\psi'(t)^2dt^2\right) \\
=&\frac{ds^2-2\ph'(t)dsdt+(\ph'(t)^2+\psi'(t)^2)dt^2}{\cos^2(\psi(t))}\,.
\end{align*}
Therefore we have $\Phi_\theta^*h_0=h_0(b_0\cdot,b_0\cdot)=h_0(\cdot,b_0^2)$, where $b_0$ is self-adjoint for $h_0$, positive definite, and in the $(s,t)$-coordinates
$$b_0^2=h_0^{-1}\Phi_\theta^*h_0=\frac{\cos^2(t)}{\cos^2(\psi(t))}\begin{pmatrix} 1 & -\ph'(t) \\ -\ph'(t) & \ph'(t)^2+\psi'(t)^2 \end{pmatrix}\,.$$
Let us observe that $$\det b_0^2=\frac{\cos^4(t)}{\cos^4(\psi(t))}\psi'(t)^2\,.$$
Since we are imposing that $\Phi_\theta$ is area-preserving, we must require $\det b=1$, which is equivalent to $\det b_0=\det b_0^2=1$. We obtain (assuming $\psi'(t)>0$):
$$\frac{\psi'(t)}{\cos^2(\psi(t))}=\frac{1}{\cos^2(t)}\,,$$
which leads to the condition $\tan(\psi(t))=\tan t+C$ (after a choice of the sign).

In the following, we choose $C=0$, thus $\psi(t)=t$ and we look for a map of the form $\Phi(s,t)=(s-\ph(t),t)$.

Observe that, for the Cayley-Hamilton theorem, $b_0^2-(\tr b_0)b_0+(\det b_0)E=0$. Hence we have $b_0^2+E=(\tr b_0)b_0$ and, taking the trace, $(\tr b_0)^2=2+\tr(b_0^2)$. In our case, $\tr(b_0^2)=2+\ph'(t)^2$ and therefore we obtain the following positive definite root of $b_0^2$:
$$b_0=\frac{1}{\sqrt{4+\ph'(t)^2}}(b_0^2+E)=\frac{1}{\sqrt{4+\ph'(t)^2}}\begin{pmatrix} 2 & -\ph'(t) \\ -\ph'(t) & 2+\ph'(t)^2 \end{pmatrix}\,.$$

From Remark \ref{rmk varpi}, it remains to impose that $b=R_{\rho}b_0$ satisfies $\tr b=2\cos\theta$, $\tr Jb<0$ and the Codazzi condition for $g_{\Hyp^2}$. 

\begin{remark}
The condition $\tr b=2\cos\theta$ is equivalent to $\tr(R_{\rho}b_0)=2\cos\theta$. Observe that
$$R_{\rho}b_0=(\cos\rho E+\sin\rho J)b_0=(\cos\rho)b_0+(\sin\rho)Jb_0~.$$
Since $b_0$ is self-adjoint, we have $\tr(Jb_0)=0$, hence the following formula holds:
$$\mathrm{tr}(R_{\rho}b_0)=\cos\rho\,\tr b_0~.$$ 
Using this formula, the condition $\tr(R_{\rho}b_0)=2\cos\theta$ can be rewritten as
\begin{equation}\label{equation varpi trace}
\cos\rho=\frac{2\cos\theta}{\tr b_0}\,.
\end{equation}
Note that $\tr b_0\geq 2$, since $\det b_0=1$ and $b_0$ is positive definite. Thus there are precisely two possible continuous choices of the angle $\rho$, namely
\begin{equation} \label{choice varpi}
\rho=\pm\arccos\left(\frac{2\cos\theta}{\mathrm{tr}(b_0)}\right)\,.
\end{equation}
 Since $\cos\theta\neq 1$, neither of these two functions can vanish at any point. 
 \end{remark}

In the case being considered, $\tr(b_0)=\sqrt{4+\varphi'(t)^2}$. Hence, by imposing that $\tr(Jb)<0$, we get $\sin\rho>0$ and thus
$$\rho=\arccos\left(\frac{2\cos\theta}{\sqrt{4+\varphi'(t)^2}}\right)\,.$$
It thus remains to impose the Codazzi condition to $b=R_\rho b_0$ to determine the function $\ph(t)$.
Using Equation \eqref{eq:etabeta}, it suffices to impose
$$d\rho(\partial_s)J b_0(\partial_t) -d\rho(\partial_t)J b_0(\partial_s)+d^\nabla b_0(\partial_s,\partial_t)=0\,,$$
and since $d\rho(\partial_s)=0$ and $[\partial_s,\partial_t]=0$,
we must impose
\begin{equation} \label{eq:pino}
 \nabla_{\partial_s}b_0(\partial_t)-\nabla_{\partial_t}b_0(\partial_s)-d\rho(\partial_t)J b_0(\partial_s)=0\,.
\end{equation} 
 For this purpose, let us compute
\begin{align*}
\nabla_{\partial_s}b_0(\partial_t)&=\nabla_{\partial_s}\left(\frac{1}{\sqrt{4+\ph'(t)^2}}(-\ph'(t)\partial_s+(2+\ph'(t)^2)\partial_t)\right) \\
&=\frac{1}{\sqrt{4+\ph'(t)^2}}\left(-\ph'(t)\nabla_{\partial_s}\partial_s+(2+\ph'(t)^2)\nabla_{\partial_s}\partial_t\right)
\end{align*}
Observing that $\nabla_{\partial_s}\partial_s=-\tan(t)\partial_t$ and $\nabla_{\partial_s}\partial_t=\tan(t)\partial_s$, we obtain
$$\nabla_{\partial_s}b_0(\partial_t)=\frac{\tan(t)}{\sqrt{4+\ph'(t)^2}}\left((2+\ph'(t)^2){\partial_s}+\ph'(t)\partial_t\right)\,.$$
On the other hand, 
\begin{align*}
\nabla&_{\partial_t}b_0(\partial_s)=\nabla_{\partial_t}\left(\frac{1}{\sqrt{4+\ph'(t)^2}}(2\partial_s-\ph'(t)\partial_t)\right) \\
&=-\frac{\ph'(t)\ph''(t)}{(4+\ph'(t)^2)^{3/2}}(2\partial_s-\ph'(t)\partial_t) 
+ \frac{1}{\sqrt{4+\ph'(t)^2}}\left(2\nabla_{\partial_t}\partial_s-\ph''(t)\partial_t-\ph'(t)\nabla_{\partial_t}\partial_t\right) \,,
\end{align*}
Using moreover $\nabla_{\partial_t}\partial_t=\tan(t)\partial_t$,
\begin{align*}
\nabla_{\partial_t}b_0(\partial_s)=&\left(-\frac{2\ph'(t)\ph''(t)}{(4+\ph'(t)^2)^{3/2}}+\frac{2\tan(t)}{\sqrt{4+\ph'(t)^2}} \right)\partial_s \\
&+ \left(\frac{\ph'(t)^2\ph''(t)}{(4+\ph'(t)^2)^{3/2}}-\frac{\ph''(t)+\ph'(t)\tan(t)}{\sqrt{4+\ph'(t)^2}} \right)\partial_t\,.
\end{align*}
We also compute:
\begin{align*}
\partial_t\rho&= \frac{1}{\sqrt{1-\frac{4\cos^2\theta}{4+\varphi'(t)^2}}} \frac{2\cos\theta}{(4+\varphi'(t)^2)^{3/2}}\varphi'(t)\varphi''(t) \\
&=\frac{2\cos\theta\,\varphi'(t)\varphi''(t)}{(4+\varphi'(t)^2)\sqrt{4\sin^2\theta+\varphi'(t)^2}}\,.
\end{align*}

Finally $Jb_0(\partial_s)=(\varphi'(t)\partial_s+2\partial_t)/\sqrt{4+\varphi'(t)^2}$. Therefore the Codazzi condition in Equation \eqref{eq:pino} gives, for the $\partial_s$ components,
$$\frac{\tan(t)\ph'(t)^2}{\sqrt{4+\ph'(t)^2}}=-\frac{2\ph'(t)\ph''(t)}{(4+\ph'(t)^2)^{3/2}}+\frac{2\cos\theta\,\varphi'(t)^2\varphi''(t)}{(4+\varphi'(t)^2)^{3/2}\sqrt{4\sin^2\theta+\varphi'(t)^2}}\,,$$
and thus
\begin{equation} \label{eq min lag}
-\tan(t)=\frac{2}{4+\ph'(t)^2}\left(\frac{1}{\ph'(t)}-\frac{\cos\theta}{\sqrt{4\sin^2\theta+\varphi'(t)^2}}\right)\ph''(t)\,.
\end{equation}
On the other hand, equating the $\partial_t$-components in \eqref{eq:pino} gives
$$\frac{\ph'(t)^2\ph''(t)}{(4+\ph'(t)^2)^{3/2}}=\frac{\ph''(t)+2\ph'(t)\tan(t)}{\sqrt{4+\ph'(t)^2}}
-\frac{4\cos\theta\,\varphi'(t)\varphi''(t)}{(4+\varphi'(t)^2)^{3/2}\sqrt{4\sin^2\theta+\varphi'(t)^2}}$$
and thus
$$2\ph'(t)\tan(t)=\frac{\ph''(t)(\ph'(t)^2-4-\ph'(t)^2)}{(4+\ph'(t)^2)}+
\frac{2\cos\theta\,\varphi'(t)\varphi''(t)}{(4+\varphi'(t)^2)\sqrt{4\sin^2\theta+\varphi'(t)^2}}\,,$$
hence leading again to Equation \eqref{eq min lag}.

\begin{remark}
Observe that, for the first term in the RHS of Equation \eqref{eq min lag},
$$\int \frac{2\ph''(t)}{\ph'(t)(4+\ph'(t)^2)}dt=\frac{1}{2}\log(\ph'(t))-\frac{1}{4}\log(\ph'(t)^2+4)+\text{cost.}$$
Hence if $\cos\theta=0$, by straightforward algebra one obtains
$$\frac{\ph'(t)^2}{\ph'(t)^2+4}=e^{-4C}\cos^4(t)\,,$$
for some constant $C$. In conclusion,
$$\ph'(t)^2=4\left(\frac{1}{1-e^{-4C}\cos^4(t)}-1\right)\,,$$
or equivalently
\begin{equation} \label{derivative phi}
\ph'(t)=2\frac{e^{-2C}\cos^2(t)}{\sqrt{1-e^{-4C}\cos^4(t)}}\,.
\end{equation}

Observe that, if $C>0$, then $\ph'(t)$ is defined for all $t\in[-\pi/2,\pi/2]$, and therefore we obtain minimal Lagrangian maps from $\Hyp^2$ to $\Hyp^2$. Up to composing with a hyperbolic translation still of the form $\gamma_a(s,t)=(s+a,t)$, we can assume $\ph$ is a primitive of the RHS of Equation \eqref{derivative phi} such that $\ph(-\pi/2)=0$, so that the  points in $\partial\Hyp^2$ with coordinates $w=(s,-\pi/2)$ (namely, those points in the boundary on one side of the geodesic connecting $0$ and $\infty$) are fixed. It is also easy to check that the map obtained by choosing the opposite sign in Equation \eqref{derivative phi} provides the inverse map. See Figure \ref{fig:dynamics1}. However, in this paper we are interested in the solutions of Equation \eqref{eq min lag} which are \emph{not} defined on the whole interval $[-\pi/2,\pi/2]$.
\end{remark}

\begin{figure}[htbp]
\centering
\includegraphics[width=0.8\textwidth]{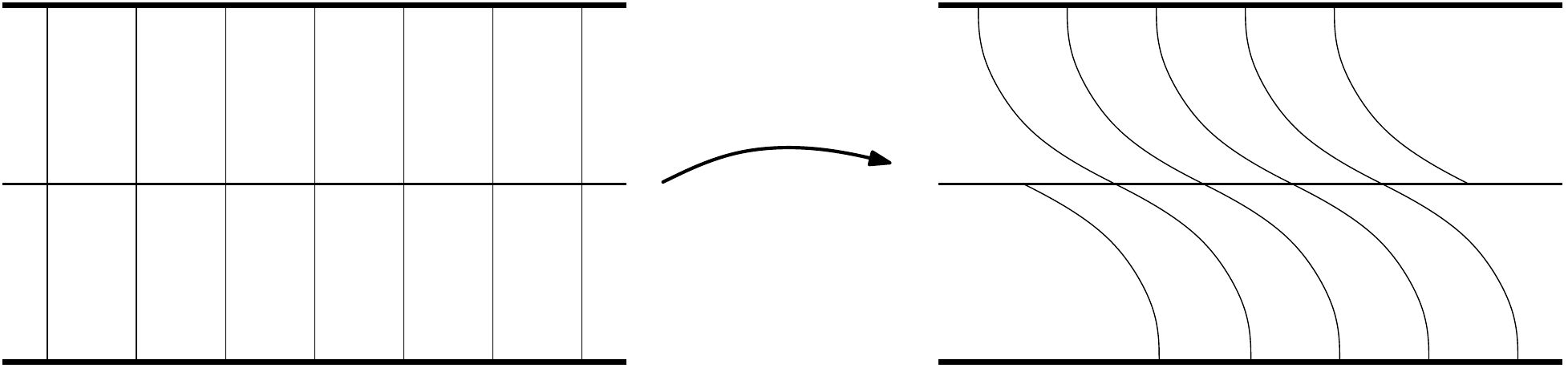}
\caption{The dynamics of the map $\Phi$ when the constant $C$ is positive, in the $w=s+it$ coordinate. This is a conformal model of $\Hyp^2$. \label{fig:dynamics1}}
\end{figure}

\subsection{Landslides of the hyperbolic half-plane} \label{subsec half-plane}

Let us now study the ODE in Equation \eqref{eq min lag}. Let us define 
$$F:[0,+\infty)\to\R\qquad F(r)=-2\cos\theta\int_{0}^r \frac{du}{(4+u^2)\sqrt{4\sin^2\theta+u^2}}\,.$$ 
Clearly $F(0)=0$ and 
$$|F(+\infty)|:=\lim_{r\to+\infty} |F(r)|<+\infty\,.$$
 By integrating both sides of Equation \eqref{eq min lag}, one obtains

\begin{equation} \label{integration 1}
\frac{1}{2}\log(\ph'(t))-\frac{1}{4}\log(\ph'(t)^2+4)+F(\varphi'(t))=\log(\cos(t))+C\,.
\end{equation}
Observe that the the function
$$r\mapsto \frac{1}{2}\log r-\frac{1}{4}\log(r^2+4)+F(r)\,,$$ 
which corresponds to the LHS of Equation \eqref{integration 1} with $\ph'(t)=r$, is an increasing function of $r$, since by construction its derivative is the expression 
$$r\mapsto \frac{2}{4+r^2}\left(\frac{1}{r}-\frac{\cos\theta}{\sqrt{4\sin^2\theta+r^2}}\right)$$
 which is positive (even if $\cos\theta>0$). (Compare Equation \eqref{eq min lag}.)
An equivalent form of Equation \eqref{integration 1} is:
\begin{equation} \label{eq ph F cos}
\frac{\ph'(t)^{1/2}}{(\ph'(t)^2+4)^{1/4}}e^{F(\varphi'(t))}=e^C \cos(t)\,.
\end{equation}

Let us put $C=F(+\infty)$. 
With this choice, the real function 
$$G(r)=-\arccos\left(\frac{r^{1/2}}{(r^2+4)^{1/4}}e^{F(r)-C}\right)$$
is strictly increasing and sends $[0,+\infty)$ to $[-\pi/2,0)$. 
Hence 
Equation \eqref{eq ph F cos} defines $\ph'(t)=G^{-1}(t)$, and in particular
 \begin{equation}\label{eq:ciccio}
 \lim_{t\to 0^-}\ph'(t)=+\infty \qquad \mathrm{and}\qquad   \lim_{t\to -\pi/2}\ph'(t)=0\,.
 \end{equation}

We will need a more precise analysis of the behavior of $\ph(t)$ close to $0$.
\begin{lemma} \label{lemma:rockerduck}
Any solution $\ph$ of Equation \eqref{eq ph F cos}, with $C=F(+\infty)$, satisfies
\begin{equation}\label{eq:zorro}
\lim_{t\to 0^-} t\ph'(t)=-\sqrt{2(1-\cos\theta)}\,. 
\end{equation}
Therefore, $\lim_{t\to 0}\ph(t)=+\infty$ and
\begin{equation}\label{eq:zorro2}
   \lim_{t\to 0^-} t e^{\ph(t)/2}=0\,.
\end{equation}
\end{lemma}
\begin{proof}
Recalling that $\ph'$ is defined by $\ph'(t)=G^{-1}(t)$, to prove \eqref{eq:zorro} we need to show that $$\lim_{t\to 0^-}tG^{-1}(t)=-\sqrt{2(1-\cos\theta)}\,.$$
Being $G^{-1}$ a diffeomorphism between $[-\pi/2,0)$ and $[0,+\infty)$, observe that
$$\lim_{t\to 0^-}tG^{-1}(t)=\lim_{v\to+\infty} vG(v)=\lim_{u\to 0}\frac{G(1/u)}{u}\,.$$
A simple computation shows that
\[
   F(1/u)-C=2\cos\theta\int_{1/u}^{+\infty}\frac{dr}{(4+r^2)\sqrt{4\sin^2\theta+r^2}}=2\cos\theta\int_0^u\frac{sds}{(4s^2+1)\sqrt{4s^2\sin^2\theta+1}}
\]
so $F(1/u)-C=\cos\theta u^2+o(u^2)$.
Thus
\[
\cos G(1/u)= \frac{e^{F(1/u)-C}}{(1+4u^2)^{1/4}} =(1-u^2+o(u^2))(1+\cos\theta u^2+o(u^2))=1-(1-\cos\theta)u^2+o(u^2)
\]
so
\[
   G^2(1/u)/2=(1-\cos\theta)u^2+o(u^2)
\]
and we conclude that $G(1/u)=-\sqrt{2(1-\cos\theta)}u +o(u)$, which implies \eqref{eq:zorro}.

For \eqref{eq:zorro2}, notice that $\sqrt{2(1-\cos\theta)}<1$, and choose a number $\alpha\in(\sqrt{2(1-\cos\theta)},1)$.
Now fixing $t_0<0$ so that $\ph'(t)<\alpha/|t|$ for $t\in[t_0, 0)$ (this exists by \eqref{eq:zorro}).
So for $t>t_0$ we have
\[
  \ph(t)-\ph(t_0)=\int_{t_0}^t\ph'(r)dr<\alpha\ln(t_0/t)
\]
so 
\[
  e^{\ph(t)}\leq e^{\ph(t_0)}\left(\frac{t_0}{t}\right)^\alpha=c|t|^{-\alpha}~,
\]
and this proves \eqref{eq:zorro2}.
\end{proof}

The associated map 
$$\Phi_\theta:\Hyp^2_+\to\Hyp^2_+\,,$$
where $\Hyp^2_+=\{z\in\Hyp^2\,|\,\Re(z)>0\}=\{t\leq 0\}$ is a half-plane in $\Hyp^2$, of the form
$$\Phi_\theta(s,t)=(s-\ph(t),t)\,,$$ 
is a $\theta$-landslide of $\Hyp^2_+$. 
Here $\ph$ is a primitive of $\ph'$ chosen in such a way that
 $\ph(-\pi/2)=0$, so that $\Phi(s,-\pi/2)=(s,-\pi/2)$ for every $s\in\R$. Since from Lemma \ref{lemma:rockerduck} we know $\ph(t)$ diverges as $t\to 0$, the map does not extend to the geodesic boundary of $\Hyp^2_+$, which is the geodesic invariant for the 1-parameter hyperbolic group.

We additionally observe that $\Phi_\theta$ maps the curve $q_{s_0}:(0,\pi/2)\to \Hyp^2_+$, parameterized by
$$q_{s_0}(t)=\left(s_0+\frac{1}{2}\ph(t),t\right)\,,$$
to the curve $p_{s_0}$ with parameterization
$$p_{s_0}(t)=\left(s_0-\frac{1}{2}\ph(t),t\right)\,,$$
which is the image of $q_{s_0}$ by means of the reflection $(s,t)\mapsto(2s_0-s,t)$. See Figure \ref{fig:dynamics2}.

\begin{figure}[htbp]
\centering
\includegraphics[width=0.8\textwidth]{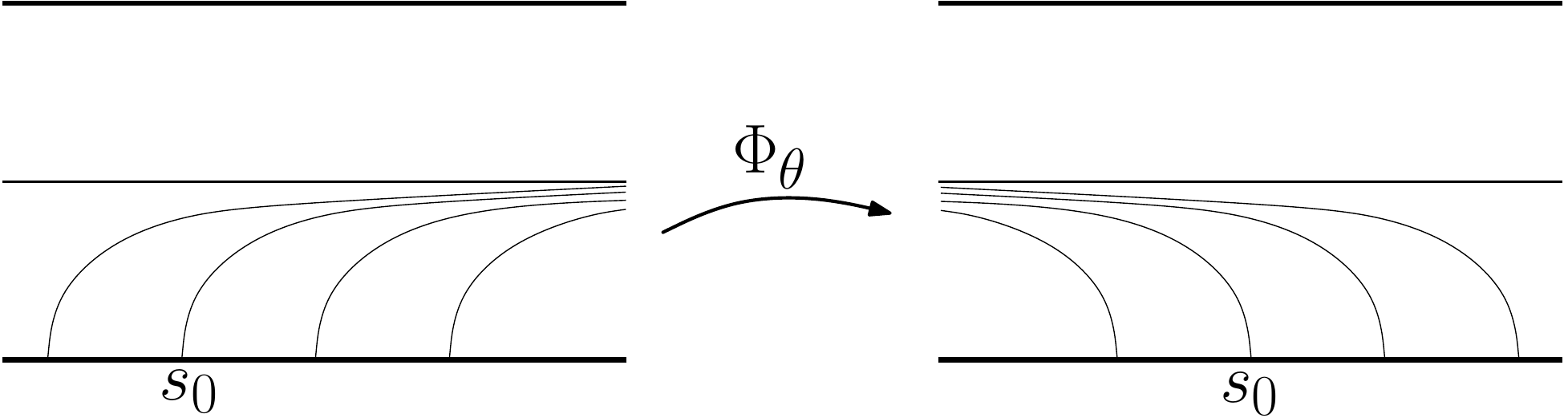}
\caption{A sketch of the dynamics of the map $\Phi_\theta$, in the $w=s+it$ coordinate. \label{fig:dynamics2}}
\end{figure}

Let us remark that, since $\Phi_\theta$ commutes with the 1-parameter hyperbolic group $\gamma_a$, then the images of the constant curvature embeddings $\sigma_{\Phi_\theta,R_\rho b_0}$ are invariant for the group of isometries of the form $(\gamma_a,\gamma_a)$, which preserve the plane $\mathcal R_\pi$. 
We want to prove that the image of $\sigma_{\Phi_\theta,R_\rho b_0}$ is the spacelike part of a past-convex $K$-surface $S_K$ (with as usual $K=-1/\cos^2(\theta/2)$) whose boundary in $\partial\AdS^3$ is the union of the boundary at infinity of an invariant half-plane in $\mathcal R_\pi$ and a sawtooth, i.e. two null segments in $\partial\AdS^3$. Moreover, the lightlike part of $S_\theta$ is precisely the lightlike triangle bounded by the two null segments. This last condition will be the key property we will need to use $S_K$ as a \emph{barrier} in the next sections. 

Defining $\sigma(t)=\sigma_{\Phi_\theta,R_\rho b_0}(q_{s_0}(t))$, by applying the definition, the isometry $\sigma(t)$ maps the point $p_{s_0}(t)$ to $q_{s_0}(t)$. From Equation \eqref{eq:edera}, the distance of a point from the geodesic $\{t=0\}$ only depends on the $t$-coordinate, and thus $p_{s_0}$ and $q_{s_0}$ have the same distance from the geodesic $\{t=0\}$. Hence $\sigma(t)$ can be an elliptic isometry with fixed point on the geodesic $\{s=s_0\}$,
a parabolic isometry fixing an endpoint of $\{s=s_0\}$, or
a hyperbolic isometry whose axis is orthogonal to the geodesic $\{s=s_0\}$. 

The goal of this part is to prove that for $t\to 0$ the family of isometries
$\sigma_{\Phi_\theta,R_\rho b_0}(q_{s_0}(t))$ actually converges to $\mathcal{I}_{(s_0,0)}$, that is, to the elliptic isometry of order two fixing the point of intersection of $\{s=s_0\}$ and $\{t=0\}$.
In this way we will get that the surface associated to $\Phi_\theta$ contains the spacelike line in the totally geodesic plane $\mathcal R_\pi$ invariant by $(\gamma_a, \gamma_a)$.

\begin{lemma} \label{lemma:dinamitebla}
If in the $w=s+it$ coordinate we denote $p(t)=p_{s_0}(t)=(s_0-({1}/{2})\ph(t),t)$, $q(t)=q_{s_0}(t)=(s_0+({1}/{2})\ph(t),t)$ and $\sigma(t)=\sigma_{\Phi_\theta,R_\rho b_0}(q_{s_0}(t))$, then
\begin{equation}\label{eq:archimede}
\begin{split}
d(\sigma(t))\left(\left.\frac{\partial\,}{\partial s}\right |_{p(t)}\right)&=\cos\xi(t) \left(\left.\frac{\partial\,}{\partial s}\right |_{q(t)}\right)+\sin\xi(t)\left(\left.\frac{\partial\,}{\partial t}\right |_{q(t)}\right)\\
 d(\sigma(t))\left(\left.\frac{\partial\,}{\partial t}\right |_{p(t)}\right)&=-\sin\xi(t) \left(\left.\frac{\partial\,}{\partial s}\right |_{q(t)}\right)+\cos\xi(t)\left(\left.\frac{\partial\,}{\partial t}\right |_{q(t)}\right)\,.
\end{split}
\end{equation}
where $\xi$ is a function of $t\in(-\pi/2,0)$ with
\begin{align}
&\lim_{t\to 0}\ph'(t)\sin\xi(t)=2(1-\cos\theta)\,, \label{eq:orazio} \\
&\lim_{t\to 0} \cos\xi(t)=-1\,.
\end{align}
In particular, $\xi(t)\to \pi \mod 2\pi$ as $t\to 0$
\end{lemma}

\begin{proof}
First observe that  $\{\cos t\frac{\partial\,}{\partial s}|_{p(t)}, \cos t\frac{\partial\,}{\partial t}|_{p(t)}\}$ and $\{\cos t\frac{\partial\,}{\partial s}|_{q(t)}, \cos t\frac{\partial\,}{\partial t}|_{q(t)}\}$
are  orthonormal basis  respectively at $p(t)$ and $q(t)$, for the form \eqref{eq:edera} of the metric. Thus $d(\sigma(t))$ at $p(t)$ must have the form of Equation \eqref{eq:archimede} for some angle $\xi(t)$.

Let us now compute $\xi(t)$. Observe that, with respect to the $(\partial_s,\partial_t)$-frame,
$$(d\Phi_0)_{(s,t)}=\begin{pmatrix} 1 & -\ph'(t) \\ 0 & 1 \end{pmatrix}\,.$$
Recall that, 
$$b_{(s,t)}=R_\rho\circ (b_0)_{(s,t)}=\frac{1}{\sqrt{4+\ph'(t)^2}}\begin{pmatrix} \cos\rho(t) & -\sin\rho(t) \\ \sin\rho(t) & \cos\rho(t) \end{pmatrix}\begin{pmatrix} 2 & -\ph'(t) \\ -\ph'(t) & 2+\ph'(t)^2 \end{pmatrix}\,,$$
and thus
$$
d\sigma=-b\circ (d\Phi_0)^{-1}=-\frac{1}{\sqrt{4+\ph'(t)^2}}\begin{pmatrix} \cos\rho(t) & -\sin\rho(t) \\ \sin\rho(t) & \cos\rho(t) \end{pmatrix}\begin{pmatrix} 2 & \ph'(t) \\ -\ph'(t)  & 2 \end{pmatrix} \,.
$$
It follows that
\[
  \sin\xi(t)=-\frac{1}{\sqrt{4+\ph'(t)^2}}\left(2\sin\rho(t)-\ph'(t)\cos\rho(t)\right)
\]

By Equation \eqref{eq:ciccio}, $\ph'(t)\to+\infty$ as $t\to 0$, so  $\mathrm{tr}(b_0)=\sqrt{4+\ph'(t)^2}$ tends to infinity as well. 
By Equation \eqref{equation varpi trace},
$$\ph'(t)\cos\rho(t) =\ph'(t)\frac{2\cos\theta}{\mathrm{tr}(b_0)}\to 2\cos\theta\,.$$
On the other hand,   as $\rho=\arccos(2\cos\theta/\mathrm{tr}(b_0))$, 
$\sin\rho(t)\to 1$ as $t\to 0$. 
Therefore, as $t\to 0$:
\[
\ph'(t)\sin\xi(t)=-\frac{\ph'(t)}{\sqrt{4+\ph'(t)^2}}\left(2\sin\rho(t)-\ph'(t)\cos\rho(t)\right)\to -2(1-\cos\theta)~.
\]
Finally an explicit computation shows that $$\cos\xi(t)=-\frac{1}{\sqrt{4+\ph'(t)^2}}\left(2\cos\rho(t)+\ph'(t)\sin\rho(t)\right)\,.$$
Hence $\lim \cos\xi(t)=-\lim\sin\rho(t)=-1$.
\end{proof}

\begin{prop}\label{prop:nonnapapera}
Denoting  $p(t)=(s_0-({1}/{2})\ph(t),t)$, and $\sigma(t)=\sigma_{\Phi_\theta,R_\rho b_0}(q(t))$ (in the $w=s+it$ coordinate, as in Lemma \ref{lemma:dinamitebla}), then as $t\to 0$ the family of isometries $\{\sigma(t)\}$ converges to $\mathcal{I}_{m}$,
the elliptic rotation of angle $\pi$ around the point $m=(s_0,0)$.
\end{prop}
\begin{proof}
The isometry $\sigma(t)$ is defined by the  properties 
 that it sends $p(t)$ to $q(t)$ and its differential 
 sends the vector $\frac{\partial\,}{\partial s}|_{p(t)}$ to
 a vector forming the positive angle $\xi(t)$ with $\frac{\partial\,}{\partial s}|_{q(t)}$, by Equation~\eqref{eq:archimede}.
 
  Recall that in the $z=x+iy$ coordinate of the upper half-plane model, $z=i\exp w=\exp(s+i(\pi/2+t))$,
 so $$p(t)=\exp(s_0-\ph(t)/2+i(\pi/2+t))\qquad q(t)=\exp(s_0+\ph(t)/2+i(\pi/2+t))\,.$$
 
 The transformation $\delta(t): z\mapsto e^{\ph(t)}z$ sends $p(t)$ to $q(t)$ and  $\frac{\partial\,}{\partial s}|_{p(t)}$ to
 $\frac{\partial\,}{\partial s}|_{q(t)}$. So $\sigma(t)$ is obtained as the rotation of angle $\xi(t)$ around $p(t)$ post-composed with $\delta(t)$.
 
 An explicit computation shows that the rotation around a point $p=|p|e^{i\eta}$ of angle $\xi$ is represented by the $\SL(2,\R)$ matrix
 $$
 \begin{pmatrix}
  \frac{\sin(\eta+\xi/2)}{\sin\eta} & -\frac{\sin(\xi/2)|p|}{\sin\eta}\\
  \frac{\sin(\xi/2)}{|p|\sin\eta} & \frac{\sin(\eta-\xi/2)}{\sin\eta}
  \end{pmatrix}~.
 $$
 (Notice that changing $\xi$ by $\xi+2\pi$ the matrix changes by sign.)
 
 Applying this formula to $p(t)$, $\eta=\eta(t)=\pi/2+t$ and $|p|=e^{s_0-\ph(t)/2}$ and multiplying the rotation matrix by 
 $\delta(t)=\mathrm{diag}(e^{\ph(t)/2}, e^{-\ph(t)/2})$ we get  that
 
 \begin{align*}
( \sigma(t))_{21}&=-e^{s_0}\frac{\sin(\xi(t)/2)}{\cos t}\\
(\sigma(t))_{12}&= e^{-s_0}\frac{\sin(\xi(t)/2)}{\cos t}\\
(\sigma(t))_{22}&=e^{-\ph(t)/2}\frac{\sin(\pi/2+t-\xi(t)/2)}{\cos t} \,. \\
 \end{align*}

As $\xi(t)\to\pi\ (\mathrm{mod} 2\pi)$, we obtain $\lim(\sigma(t))_{21}=\pm e^{-s_0}$ and $\lim(\sigma(t))_{12}=\mp e^{s_0}$, whereas
 $(\sigma(t))_{22}\to 0$.
 
On the other hand,  the first entry of $\sigma(t)$ is 
 \[
 (\sigma(t))_{11}=e^{\ph(t)/2}\frac{\sin(\pi/2+t+\xi(t)/2)}{\cos t}= e^{\ph(t)/2}\frac{\cos(t+\xi(t)/2)}{\cos t}\,.
 \]
 Now,
 \begin{align*}
 \cos(t+\xi(t)/2)&=\cos t\cos(\xi(t)/2)-\sin t\sin(\xi(t)/2)\\
 &=\cos t\sqrt{\frac{\cos\xi(t)+1}{2}}-\sin t\sin(\xi(t)/2)\\
&= \frac{\cos t\sin\xi(t)}{\sqrt{2(1-\cos\xi(t))}} -\sin t\sin(\xi(t)/2)~.
\end{align*}

 By Equations \eqref{eq:orazio} and \eqref{eq:zorro} it follows that the ratio
 \[
 \frac{\cos(t+\xi(t)/2)}{t}
 \]
 is bounded.
 Thus using Equation \eqref{eq:zorro2} we deduce that $\lim_{t\to 0}\sigma_{11}(t)=0$.

In conclusion
\[
\sigma(t)\to \begin{bmatrix} 0 & e^{-s_0}\\ -e^{s_0} &0\end{bmatrix}=\mathcal{I}_{m}~,
\]
which proves the claim.
\end{proof}


\begin{cor} \label{lemma barrier}
The constant curvature embedding $\sigma_{\Phi_\theta,R_\rho b_0}:\Hyp^2_+\to\isom(\Hyp^2)$ defined by
\begin{itemize}
\item ($\sigma_{\Phi_\theta,R_\rho b_0}(z))(\Phi_\theta(z))=z\,$;
\item $(d\sigma_{\Phi_\theta,R_\rho b_0}(z))_z\circ (d\Phi_\theta)_{z}=-R_\rho b_0\,$;
\end{itemize}
with respect to the map $\Phi_\theta$ and the bundle morphism $b=R_\rho b_0$ constructed above, has image a constant curvature surface in $\AdS^3$ whose boundary coincides with the boundary of the half-plane $\Hyp^2_+$.
\end{cor}
\begin{proof}
From Proposition \ref{prop:nonnapapera}, for every $s_0$ we found a family of points $p_{s_0}(t)$ whose images for the embedding  $\sigma_{\Phi_\theta,R_\rho b_0}$ converges to $(s_0,0)$. That is, all points on the geodesic boundary of  $\Hyp^2_+$ (identified to a half-plane of the totally geodesic plane $\mathcal R_\pi$) are in the frontier of $\sigma_{\Phi_\theta,R_\rho b_0}(\Hyp^2_+)$.

To conclude the proof, it remains to show that, when $t\to-\pi/2$, $\sigma_{\Phi_\theta,R_\rho b_0}$ tends to the boundary at infinity of the half-plane $\Hyp^2_+$ of $\mathcal R_\pi$. 
Denoting again by $\sigma(t)$ the image of the point $p_{s_0}(t)$ by the embedding $\sigma_{\Phi_\theta,R_\rho b_0}$, we claim that both sequences $\sigma(t)(m)$ and $\sigma(t)^{-1}(m)$, where $m$ is the point $ie^{s_0}$, converge to $e^{s_0}$, which is in the boundary of the upper half-plane model. Since $s_0$ is chosen arbitrarily, and recalling the definition of Equation \eqref{defi convergence boundary}, this will conclude the proof.  

Here the computation is similar, though much simpler. Indeed, recalling that $\ph'(t)\to 0$ as $t\to -\pi/2$, we have that $\mathrm{tr}(b_0)=\sqrt{4+\ph'(t)^2}\to 2$ as $t\to-\pi/2$. Therefore, from the choice of Equation \eqref{choice varpi}, we get $\rho\to\theta$. Moreover, recall that we have chosen $\ph$ as the primitive of $\ph'$ such that $\ph(-\pi/2)=0$. Using the above, a straigthforward computation shows that $\sigma(t)(ie^{s_0})\to e^{s_0}$. Analogously, $\sigma(t)^{-1}(ie^{s_0})\to e^{s_0}$.
\end{proof}

\section{Existence of constant-curvature surfaces} \label{sec existence}

The idea of the proof of the existence of $K$-surfaces with prescribed boundary is to obtain a $K$-surface  by approximation from surfaces which are invariant by a cocompact action. In that case, the existence of constant curvature surfaces is guaranteed by the following theorem:

\begin{theorem}[\cite{barbotzeghib}] \label{bbz}
Let 
$\phi:\partial\Hyp^2\to \partial\Hyp^2$ be an orientation-preserving homeomorphism which is equivariant for a pair of Fuchsian surface group representation $$\rho_l,\rho_r:\pi_1(\Sigma_g)\to\isom(\Hyp^2)\,,$$ where $\Sigma_g$ is a closed surface. Let $\Gamma$ the curve in $\partial\AdS^3$ which is the graph of $\phi$. Then there exists a foliation of $\mathcal{D}_+(\Gamma)$ and $\mathcal{D}_-(\Gamma)$ by surfaces of constant curvature $K\in(-\infty,-1)$, which are invariant for the representations 
$$(\rho_l,\rho_r):\pi_1(\Sigma_g)\to\isom(\AdS^3)\cong \isom(\Hyp^2)\times\isom(\Hyp^2)\,.$$
The foliations are such that, if $|K_1|<|K_2|$, then the $K_1$-surface is in the convex side of the $K_2$-surface.
\end{theorem}

For this purpose, we will use three main technical tools. The first is the following lemma, which we prove in Subsection \ref{subsec approximation} below.

\begin{lemma} \label{lemma approx acausal}
Given any weakly acausal curve $\Gamma$ in $\partial\AdS^3$, there exists a sequence of curves $\Gamma_n$ invariant for pairs of Fuchsian representations of the fundamental group of closed surfaces
$$((\rho_l)_n,(\rho_r)_n):\pi_1(\Sigma_{g_n})\to\isom(\AdS^3)\cong\isom(\Hyp^2)\times\isom(\Hyp^2)$$
such that $\Gamma_n$ converges to $\Gamma$ in the Hausdorff convergence.
\end{lemma}

The second is the following theorem of Schlenker on the smooth convergence of surfaces, specialized to the case of the ambient manifold $(M,g)=\AdS^3$:

\begin{theorem}[\cite{schlenkersurfconv}] \label{schl_degeneration}
Let $\sigma_n:\D\to M$ be a sequence of uniformly elliptic spacelike immersions (i.e. with uniformly positive determinant of the shape operator) in a Lorentzian manifold $(M,g)$. Assume $\sigma_n^*g$ converges $C^\infty$ to a metric $g_\infty$. If $x_n\in\D$ is a sequence converging to a point $x_\infty$, such that the 1-jet $j^1\sigma_n(x_n)$ is converging, but $\sigma_n$ does not converge $C^\infty$ in a neighborhood of $x_\infty$, then there exists a maximal geodesic $\gamma$ of $(\D,g_\infty)$ containing the point $x_\infty$ such that $\sigma_n$, restricted to $\gamma$, converges to an isometry onto a geodesic $\gamma'$ of $M$. 
\end{theorem}

Moreover, in the conditions of Theorem \ref{schl_degeneration}, the degeneration of the immersions $\sigma_n$ is well-understood in \cite{schlenkersurfconv}: if $\sigma_n$ does not converge $C^\infty$ in a neighborhood of $x_\infty$, then the surfaces $\sigma_n(\D)$ converge to a surface which contains all the (future-directed or past-directed) light rays starting from points on the geodesic $\gamma'$ of $M$.

The third tool is the use of the barrier we constructed in Section \ref{sec barrier}, in particular Corollary \ref{lemma barrier}.

\subsection{Approximation of measured geodesic laminations} \label{subsec approximation}

We recall here the definition of measured geodesic lamination on $\Hyp^2$. Let $\mathcal{G}$ be the set of (unoriented) geodesics of $\Hyp^2$. The space $\mathcal{G}$ is identified to $((\partial\Hyp^2\times \partial\Hyp^2)\setminus \Delta)/\sim$ where $\Delta$ is the diagonal and the equivalence relation is defined by $(p,q)\sim(q,p)$. Note that $\mathcal{G}$ has the topology of an open M\"{o}bius strip. 

\begin{defi} \label{defi mgl}
A geodesic lamination on $\Hyp^2$ is a closed subset of $\mathcal{G}$ such that its elements are pairwise disjoint geodesics of $\Hyp^2$. A measured geodesic lamination is a locally finite Borel measure on $\mathcal{G}$ such that its support is a geodesic lamination.
\end{defi}

For the approximation procedure, we will use the following Lemma \ref{fuchsian weak approximation} on the approximation of measured geodesic laminations, which is proved in \cite[Lemma 3.4]{Bonsante:2015vi}. We first need to recall some definitions.

\begin{defi}
A sequence $\{\mu_n\}_n$ of measured geodesics laminations converges in the weak* topology to a measured geodesic lamination, $\mu_n\war\mu$, if
$$\lim_{n\rar\infty}\int_{\mathcal{G}}f d\mu_n=\int_{\mathcal{G}}f d\mu$$
for every $f\in C_0^0(\mathcal{G})$.
\end{defi}

\begin{lemma} \label{fuchsian weak approximation}
Given a measured geodesic lamination $\mu$ on $\Hyp^2$, there exists a sequence of measured geodesic laminations $\mu_n$ such that $\mu_n$ is invariant under a torsion-free cocompact Fuchsian group $G_n<\isom(\Hyp^2)$ and $\mu_n\war \mu$.
\end{lemma}

The second technical lemma ensures that the weak*-convergence of the bending laminations implies the convergence of the curves in the boundary of $\AdS^3$, which is basically the same as the uniform convergence of the corresponding left earthquake maps.

\begin{lemma} \label{convergence earthquakes}
Let $\mu_n,\mu$ be measured geodesic laminations on $\Hyp^2$ such that $\mu_n\war \mu$ in the weak* topology. Assume $\mu_n$ and $\mu$ induce earthquakes of $\Hyp^2$ and $p_0$ is a point of $\Hyp^2$ which is not on any weighted leaf of $\mu$. Then the homeomorphism $\phi_n$ of $\partial\Hyp^2$ obtained by earthquake along $\mu_n$, normalized in such a way that the stratum containing $p_0$ is fixed, converges uniformly to the homeomorphism $\phi$ obtained by earthquake along $\mu$, normalized in the analogous way.
\end{lemma}
\begin{proof}
Let $C_n=gr(\phi_n)$ be the 1-dimensional submanifold of $\partial\Hyp^2\times \partial\Hyp^2$ which is the graph of $\phi_n:\partial\Hyp^2\to \partial\Hyp^2$. Up to a subsequence, $C_n$ converges to a subset $C$ of $\partial\Hyp^2\times \partial\Hyp^2$ in the Hausdorff convergence. As $C_n$ are acausal curves, it turns out that $C$ is a weakly acausal curve. Hence, in order to prove that $C=gr(\phi)$, it suffices to show that $gr(\phi)\subseteq C$. In fact we will prove that for every point $x\in \partial\Hyp^2$ there exists a sequence $x_k$ converging to $x$ such that $(x_k,\phi(x_k))\in C$. 

Let $p\in\Hyp^2$ which is not on any weighted leaf of $\mu$. Denote by $F_n(p)$ the stratum of $\mu_n$ which contains $p$, and analogously $F(p)$ is the stratum of $\mu$ contaning $p$. 
We claim that for any sequence $y_n\in \partial\Hyp^2$ in the boundary of $F_n(p)$ and any $y_0$ in the boundary of $F(p)$, if $y_n\to y_0$, then $\phi_n(y_n)\to \phi(y_0)$. This follows by \cite[Theorem 3.11.5]{epsteinmarden}, where in the language of cocycles, $E_{\mu_n}(p_0,p)\to E_{\mu}(p_0,p)$ in $\isom(\Hyp^2)$, and therefore $\phi_n(y_n)=E_{\mu_n}(p_0,p)(y_n)$ converges to $\phi(y_0)=E_\mu (p_0,p)(y_0)$. In this case, one then has $(y_0,\phi(y_0))\in C$.

Next, we claim that for every $x\in \partial\Hyp^2$ there exists a sequence of points $x_k$ such that $x_k\to x$ and $(x_k,\phi(x_k))\in C$. This will conclude the proof, by the above observations. Let $p_k$ be a sequence of points in $\Hyp^2$ converging to $x$, such that every $p_k$ is not on any weighted leaf of $\mu$. For every fixed $k$, consider the stratum $F_n(p_k)$ and let $y_n(p_k)$ be the point of $\partial\Hyp^2$, in the boundary of $F_n(p_k)$, closest to $x$ (in the round metric of $\partial\Hyp^2\sim S^1$, in the Poincar\'e disc model, for instance). If $x$ itself is in the boundary of $F_n(p_k)$, then obviously $y_n(p_k)=x$. 

Observe that the distance of $y_n(p_k)$ from $x$ is bounded in terms of the distance of $p_k$ from $x$ (in the Euclidean metric on the disc model, for instance). Let $x_k=\lim_{n} y_n(p_k)$, up to a subsequence. By our first claim, $(x_k,\phi(x_k))\in C$. Moreover, $d(y_n(p_k),x)\to 0$ as $k\to+\infty$ uniformly in $n$ (as it is estimated by $d(p_k,x)\to 0$). Therefore also $d(x_k,x)$ tends to zero as $k\to +\infty$. This implies that every $x\in \partial\Hyp^2$ is approximated by points $x_k$ such that $(x_k,\phi(x_k))\in C$.
\end{proof}

\begin{figure}[htb]
\centering
\begin{minipage}[c]{.45\textwidth}
\centering
\includegraphics[height=4.5cm]{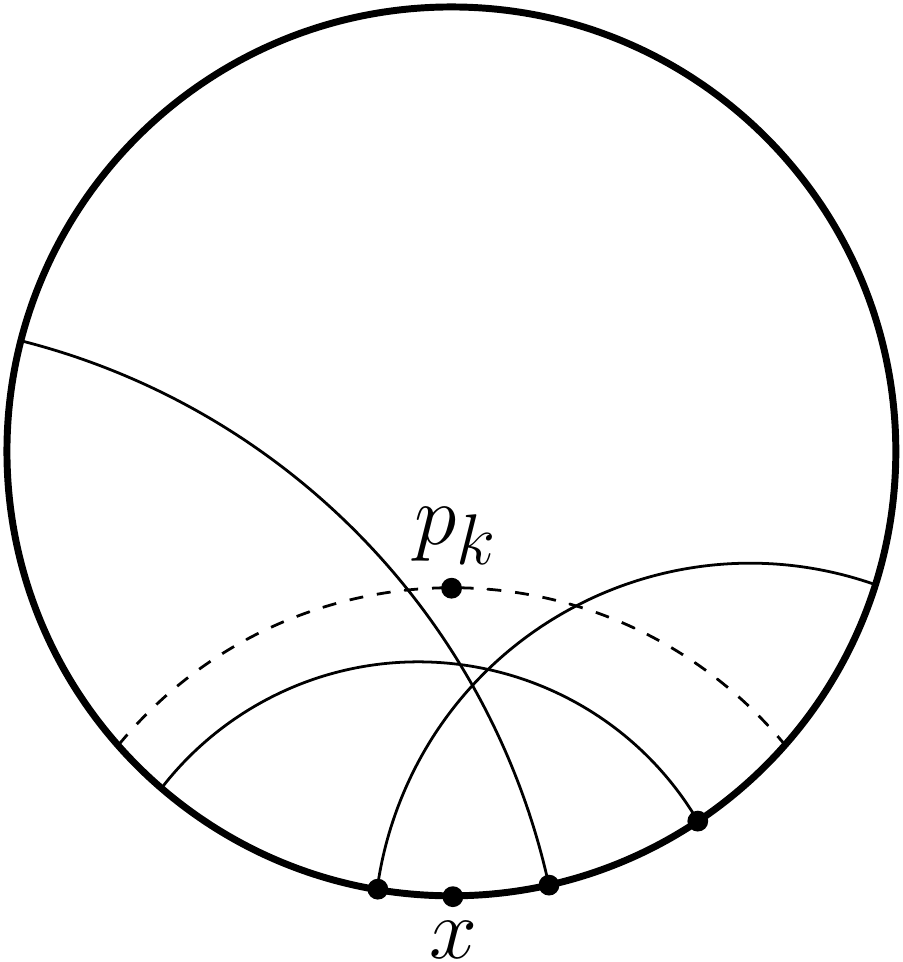}
\end{minipage}%
\hspace{3mm}
\begin{minipage}[c]{.45\textwidth}
\centering
\includegraphics[height=4.5cm]{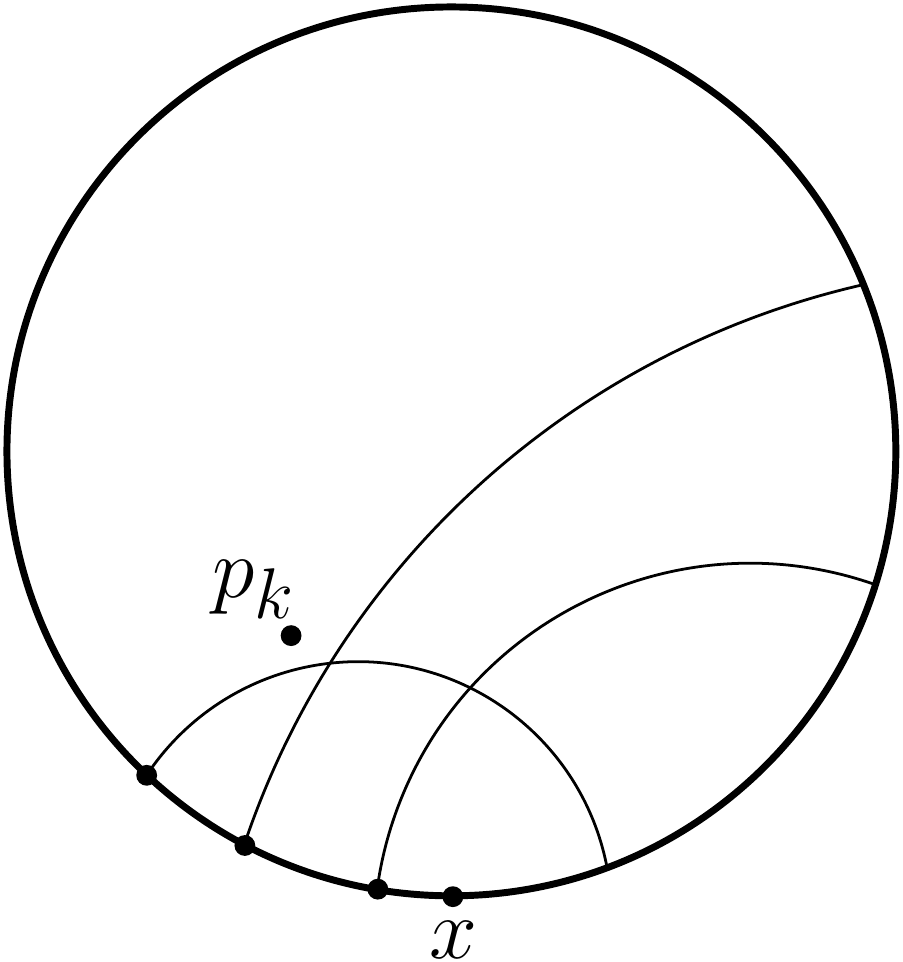}
\end{minipage}
\caption{Possible positions of the points $y_n(p_k)$, as the position of $p_k$ changes. The Euclidean distance  $d(y_n(p_k),x)$ is uniformly bounded in terms of $d(p_k,x)$, and it goes to zero as $d(p_k,x)\to 0$. \label{fig:positions}}

\end{figure}

One might consider the definition of measured geodesic lamination on a straight convex set of $\Hyp^2$, in the sense of \cite{bebo}, and prove more general statements of Lemma \ref{fuchsian weak approximation} and Lemma \ref{convergence earthquakes}. However, to avoid technicalities, we will circumvent this issue by approximating in two steps.

\begin{proof}[Proof of Lemma \ref{lemma approx acausal}]
Let $\Gamma$ be a weakly acausal curve in $\partial\AdS^3$. 
It is possible to find a sequence $\phi_n$ of orientation-preserving homeomorphisms of $\partial\Hyp^2$ such that the graphs $\Gamma_n=gr(\phi_n)$ converge to $\Gamma$ in the Hausdorff convergence. Moreover, by Lemma \ref{fuchsian weak approximation} for every $n$ there exists a sequence $\Gamma_{n,k}$ of graphs invariant for some pairs $((\rho_l)_{n,k},(\rho_r)_{n,k})$ of torsion-free cocompact Fuchsian representations. By Lemma \ref{convergence earthquakes}, the $\Gamma_{n,k}$ converge to $\Gamma_n$ in the Hausdorff convergence. By a standard diagonal argument, one finds a sequence $\Gamma_{n,{k(n)}}$ which converges to $\Gamma$.
\end{proof}

\subsection{Proof of the existence part}

The purpose of this section is to prove the existence part of $K$-surfaces of Theorem \ref{thm foliation part ads}.

A 1-step curve is the boundary of a totally geodesic lightlike plane, while a 2-step curve is the union of four null segments in $\partial\AdS^3$. Those are the two cases in which the convex hull of the curve $\Gamma$ coincides with the domain of dependence $\mathcal{D}(\Gamma)$. See Figure \ref{fig:step}.

\begin{figure}[htb]
\centering
\begin{minipage}[c]{.45\textwidth}
\centering
\includegraphics[height=6cm]{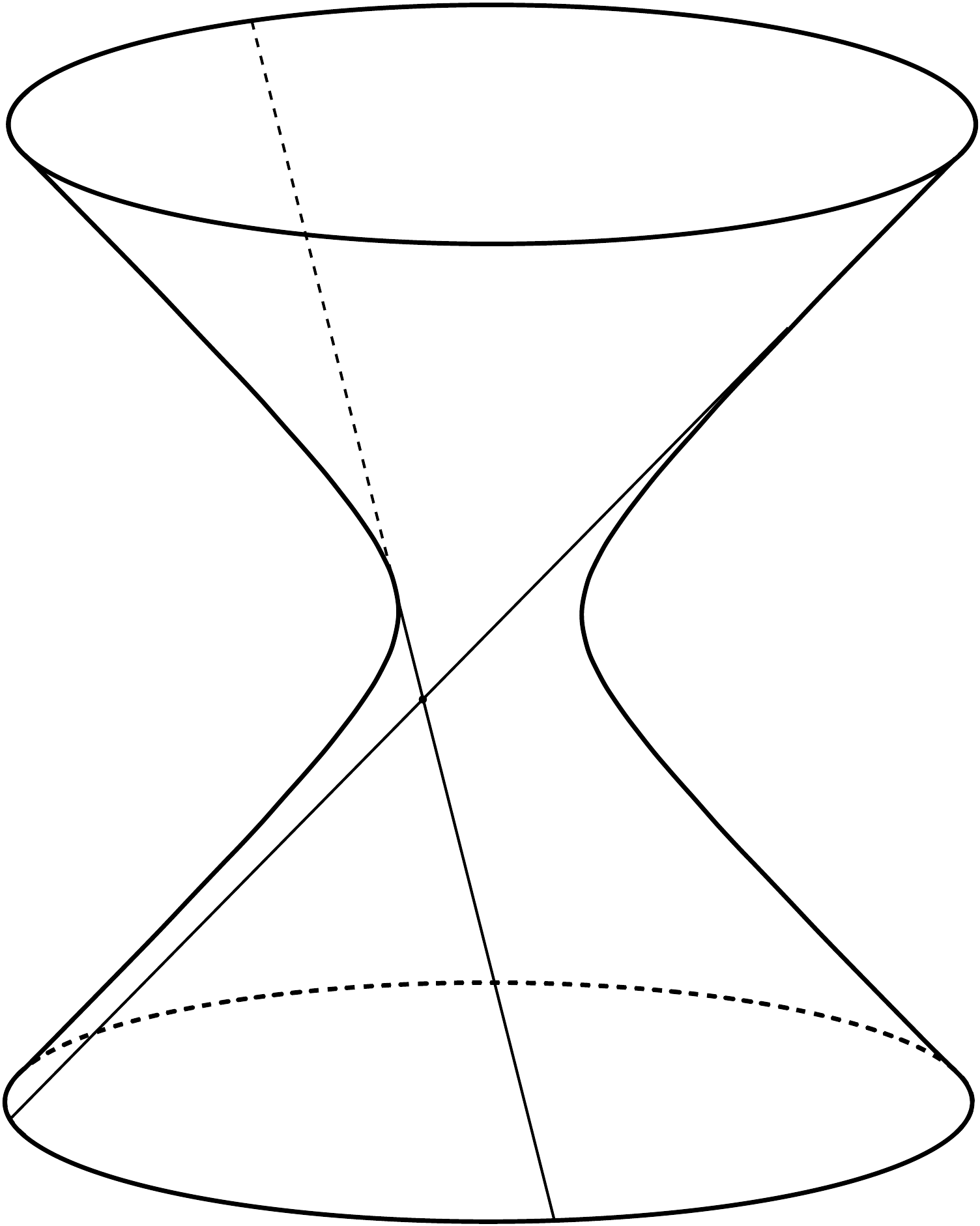}
\end{minipage}%
\hspace{1mm}
\begin{minipage}[c]{.45\textwidth}
\centering
\includegraphics[height=6cm]{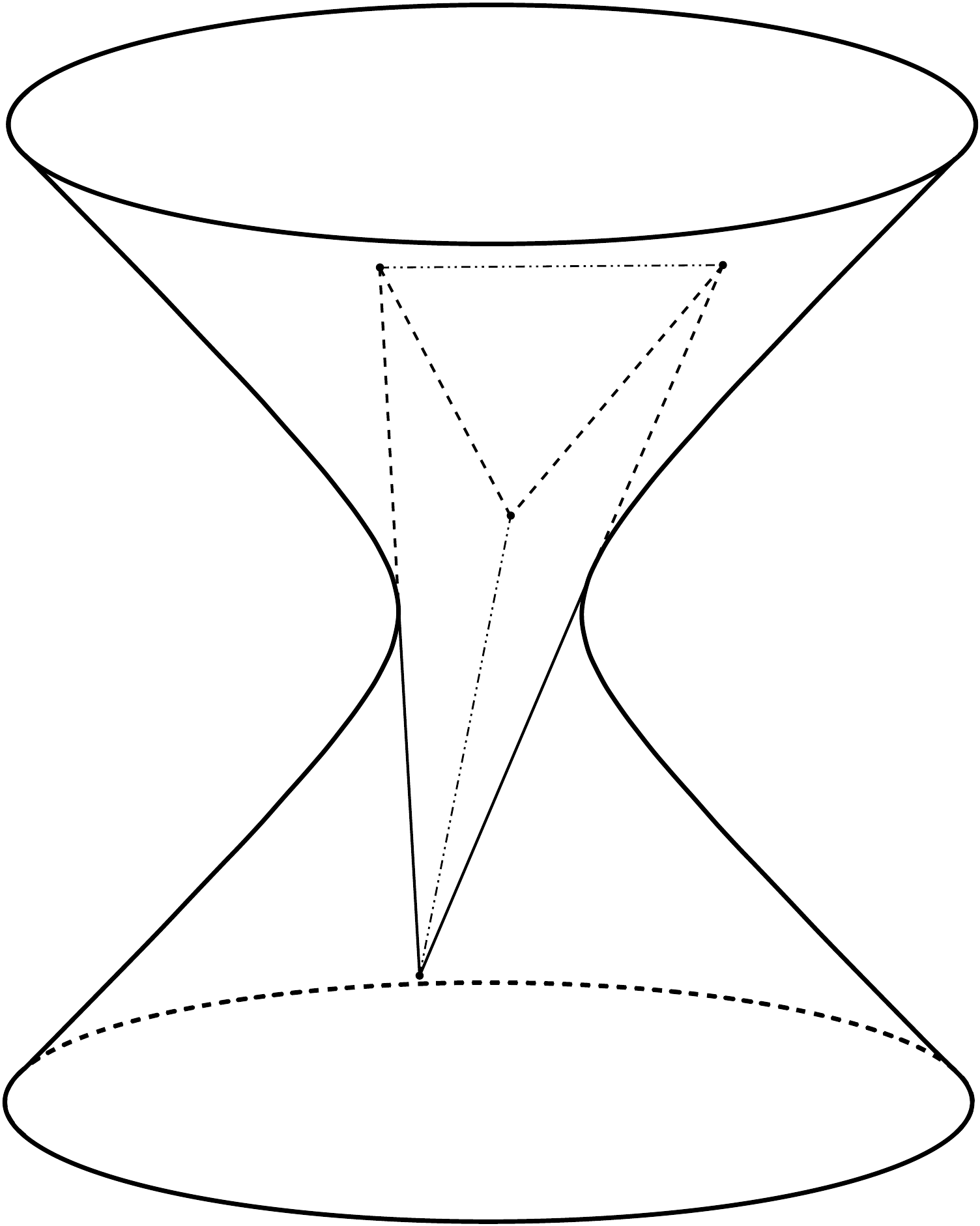}
\end{minipage}
\caption{A 1-step curve (left) bounds a lightlike plane. On the right, a 2-step curve, which is the boundary of the surface obtained as the union of two lightlike half-planes which intersect along a spacelike geodesic. \label{fig:step}}
\end{figure}



\begin{theorem} \label{prop existence}
Given any weakly acausal curve $\Gamma$ in $\partial\AdS^3$ which is not a 1-step or a 2-step curve, for every $K\in(-\infty,-1)$ there exist two convex surfaces $S^+_{K},S^-_{K}$ in $\AdS^3$ with $\partial S^\pm_{K}=\Gamma$, such that:
\begin{itemize}
\item The lightlike part of $S_K^\pm$ is $\overline{\mathcal D_\pm(\Gamma)}\cap \mathcal C(\Gamma)$, that is, the union of all lightlike triangles bounded by past-directed (resp. future-directed) sawteeth of $\Gamma$;
\item The spacelike part of $S_K^\pm$ is a smooth $K$-surface.
\end{itemize}

\end{theorem}

\begin{remark}
The first property of the surface $S_K^\pm$ in Theorem \ref{prop existence} means that its lightlike part is the smallest possible, under the condition that the boundary is the curve $\Gamma$. 
  In particular, if it is possible to construct a past-convex spacelike surface with boundary $\Gamma$, then the lightlike part of $S_K^+$ is empty and therefore $S_K^+$ is a smooth spacelike $K$-surface with boundary $\Gamma$. The same holds for future-convex surfaces. In general, the boundary of the spacelike part of $S_K^\pm$ coincides with the boundary of the spacelike part of $\mathcal C_\pm(\Gamma)$.
\end{remark}

\begin{proof}[Proof of Theorem \ref{prop existence}]
For definiteness, we will provide the argument for the past-convex surface $S_K^+$ contained in $\mathcal{D}_+(\Gamma)$, the other argument being completely analogous.
Let $\Gamma_n$ be a sequence of curves in $\partial\AdS^3$ invariant for a pair of torsion-free cocompact Fuchsian representations $((\rho_l)_n,(\rho_r)_n)$, converging to $\Gamma$, as in Lemma \ref{lemma approx acausal}. Fix $K<-1$. By Theorem \ref{bbz}, there exists a past-convex $K$-surface $(S_K)_n$ in $\mathcal{D}_+(\Gamma_n)$, with $\partial (S_K)_n=\Gamma_n$. 

 
\emph{Step 1: There exists a $C^0$-limit of the surfaces $(S_K)_n$.}
Introducing the model of $\AdS^3$ given by (see \cite{bon_schl}): 
\begin{equation} \label{eq:battista}
\left(\D^2 \times S^1,\frac{4|dz|^2-(1+|z|^2)dt^2}{(1-|z|^2)^2}\right)\,,
\end{equation}
we see that, since $(S_K)_n$ is spacelike, $\overline{(S_K)_n}$ is the graph of a $4$-Lipschitz function from $\overline{\D^2}$ to $S^1$, with respect to the Euclidean metric of the disc. 
 By an application of the Ascoli-Arzel\`a Theorem, there exists a subsequence (which we still denote by $(S_K)_{n}$) whose closure converges uniformly to the closure of a locally convex nowhere timelike surface $(S_K)_\infty$. 
As a consequence $\partial(S_K)_\infty=\Gamma$. 

It remains to show the two properties in the statement. 

\emph{Step 2: The lightlike part of the limit $(S_K)_\infty$ is the union of all maximal lightlike triangles bounded by past-directed sawteeth.}
We already know that, for any past-directed sawtooth contained in $\Gamma$, the surface $(S_K)_\infty$ contains entirely the totally geodesic lightlike triangle corresponding to the sawtooth. We will need to show the other inclusion. Namely, we will show that
any point in the complement of the past-directed lightlike triangles is not in the lightlike part of $(S_K)_\infty$. (Observe that the existence of a point  $x$ in the lightlike part of $(S_K)_\infty$  lying in the complement
of the lightlike triangles implies that $\Gamma$ is not the boundary of a lightlike plane or the boundary of two lightlike half-planes intersecting in a spacelike geodesic. See Figure \ref{fig:step}.)

For this purpose, suppose by contradiction that $x\in (S_K)_\infty$ is a point in the complement of such lightlike triangles, but still in the lightlike part of $(S_K)_\infty$. This implies that $(S_K)_\infty$ has a lightlike support plane $P$ which contains $x$. Let us show that the dual point $y$ of $P$ is necessarily on the curve $\Gamma$. Indeed, let $\ell$ be the past-directed light-ray joining $x$ to $y$ (which is contained in $P$). On the one hand, by convexity of $(S_K)_\infty$, $\ell$ does not intersect the open convex region bounded by $(S_K)_\infty$ (in an affine chart containing $\Gamma$). On the other hand, as $\ell$ is the limit of timelike rays though $x$, which are contained in this open convex region, it turns out that $\ell$ is in $(S_K)_\infty$. Hence the endpoint $y$ of $\ell$ is in the curve at infinity $\Gamma$.


Now, we can find a curve $\Gamma'$, as in Example \ref{ex crucial} (compare also Figure \ref{fig:comparison}), with the following properties:
\begin{itemize}
\item $\Gamma'$ is composed of a past-directed sawtooth contained inside the plane $P$ and of the boundary of a totally geodesic spacelike half-plane;
\item $\Gamma'$ does not intersect $\Gamma$ transversely;
\item In an affine chart which entirely contains $\mathcal D(\Gamma)$, the totally geodesic spacelike plane which contains a portion of $\Gamma'$ disconnects $x$ from $\Gamma$.
\end{itemize}
Indeed, if $P\cap \Gamma$ is composed of a single point $y$, one can insert a small enough past-directed sawtooth with vertex in $y$, such that the corresponding lightlike triangle is contained inside $P$ but does not contain $x$. Moreover $\Gamma'$ can be arranged to intersect $\Gamma$ only at $y$. The situation is similar if $P\cap\Gamma$ contains a light-like segment, but not a past-directed sawtooth.  Similarly, if $P\cap \Gamma$ contains a past-directed sawtooth, it suffices to choose a slightly larger sawtooth, in such a way that $\Gamma\cap\Gamma'$ coincides with the sawtooth $P\cap\Gamma$. See Figure \ref{fig:comparison}. Of course this is possible provided $\Gamma$ is not the boundary of a totally geodesic lightlike plane. 

\begin{figure}[htb]
\centering
\begin{minipage}[c]{.45\textwidth}
\centering
\includegraphics[height=6.5cm]{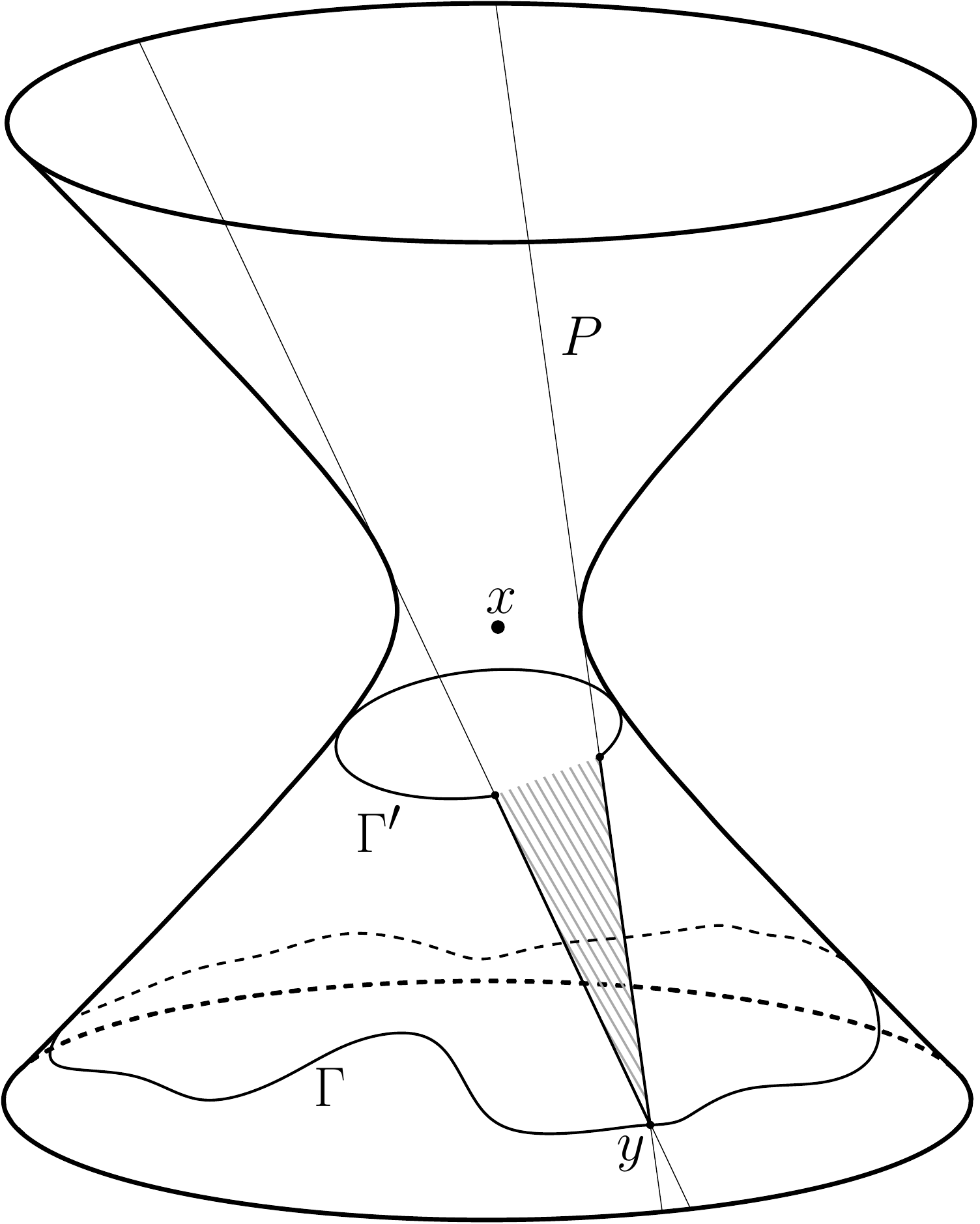}
\end{minipage}%
\hspace{1mm}
\begin{minipage}[c]{.45\textwidth}
\centering
\includegraphics[height=6.5cm]{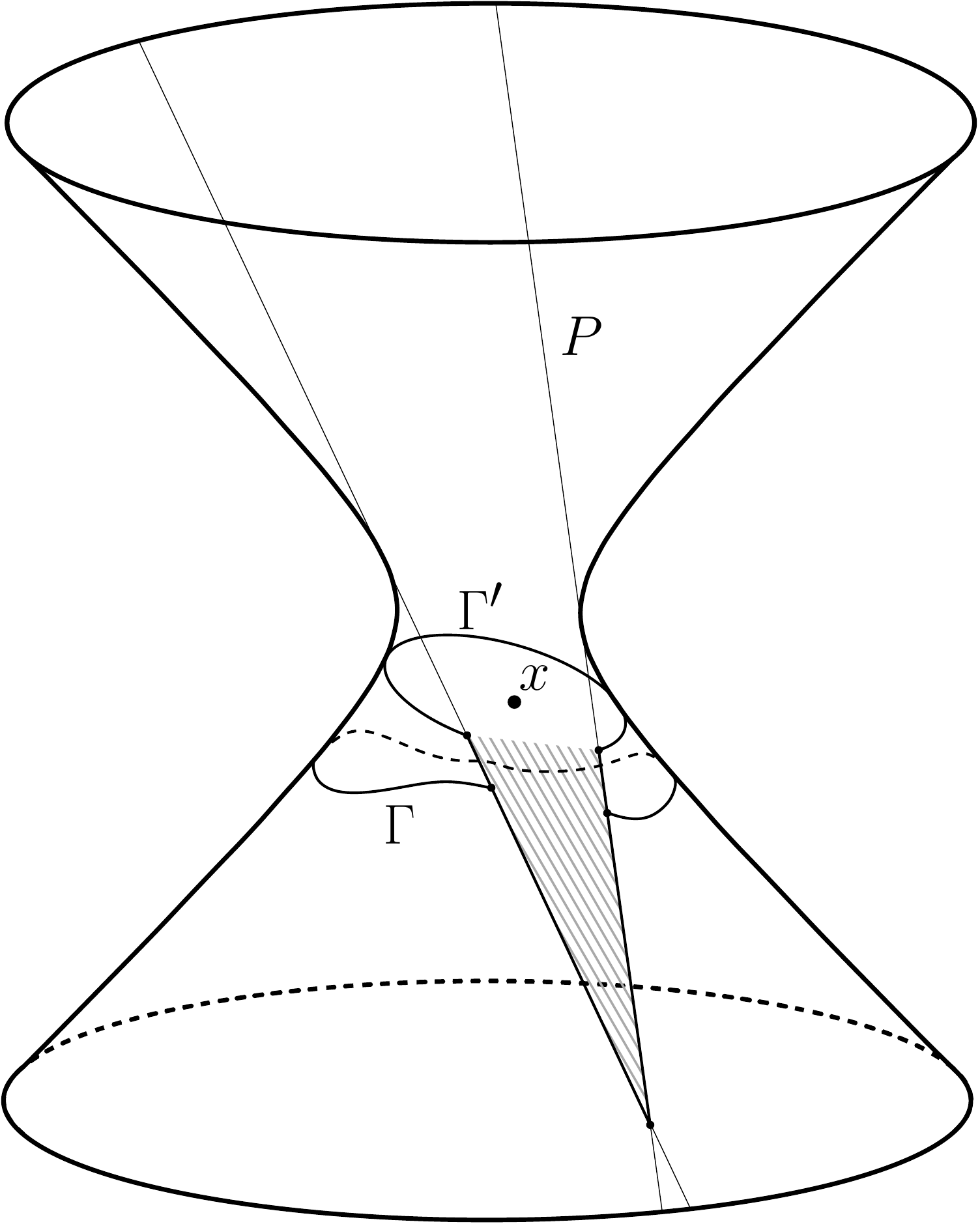}
\end{minipage}
\caption{The construction of the curve $\Gamma'$, when $\Gamma$ is the graph of an orientation-preserving homeomorphism (left) and when $\Gamma$ contains a past-directed sawtooth (right). \label{fig:comparison}}
\end{figure}

Let us consider the surface $S'_K$, which was constructed in Corollary \ref{lemma barrier}, with the following properties:
\begin{itemize}
\item The lightlike part of $S'_K$ is the maximal lightlike triangle with boundary in $\Gamma'$,
\item The spacelike part of $S'_K$ is a smooth $K$-surface.
\end{itemize}

To conclude Step 2, it suffices to show that $S'_K$ and $(S_K)_\infty$ do not intersect transversely. In fact, from Corollary \ref{lemma barrier} the surface $S'_K$ intersects $P$ exactly in the maximal lightlike triangle with boundary in $\Gamma'$. But by construction $x$ is not in such maximal lightlike triangle. Hence if we show that $S'_K$ and $(S_K)_\infty$ do not intersect transversely, it follows that $(S_K)_\infty$ cannot intersect $P$ at $x$, thus giving a contradiction.

To show the claim, 
 observe that since $\Gamma_n$ converges to $\Gamma$ as $n\to\infty$, one can choose a sequence of isometries $\chi_n\in\isom(\AdS^3)$ such that:
\begin{itemize}
\item $\chi_n(\Gamma_n)$ is disjoint from $\Gamma'$;
\item $\chi_n\to\mathrm{id}$ as $n\to\infty$.
\end{itemize}


We claim that the surfaces $S'_K$ and $\chi_n (S_K)_n$ do not intersect transversely. Since their boundaries $\Gamma'$ and $\chi_n(\Gamma_n)$ are disjoint, and $S'_K$ contains a lightlike triangle, the intersection must be contained in a compact region of the spacelike part of $S'_K$. Observe that $S'_K$ lies outside of the convex hull of $\chi_n(\Gamma_n)$. By Theorem \ref{bbz}, there is a function $$\kappa:\mathcal D_+(\chi_n(\Gamma_n))\to(-\infty,-1)$$ such that the level sets $\{\kappa=K\}$ are $K$-surfaces. Observe that $\kappa\to -\infty$ as $x\to\partial\mathcal D(\Gamma)$. Let $K_0$ be the maximum of $\kappa$ on $S'_K\cap \mathcal D_+(\chi_n(\Gamma_n)$. Then the $K_0$-surface $\chi_n (S_{K_0})_n$ is tangent to the spacelike part of $S'_K$ at an interior point and $\chi_n (S_{K_0})_n$ is contained in the convex side of  $S_K'$. 
By an application of the maximum principle at the intersection point, $K_0<K$. 
But again by Theorem \ref{bbz},  $\chi_n (S_{K})_n$ is in the convex side of $\chi_n (S_{K_0})_n$ and so $S_K'$ cannot intersect $\chi_n (S_{K})_n$.  

Since $\chi_n (S_K)_n$ converges uniformly on compact sets to $(S_K)_\infty$, this shows that $S'_K$ and $(S_K)_\infty$ do not intersect transversely. This concludes Step 2. In fact, we have proved that for any point $x\in (S_K)_\infty$ which is in the complement of the lightlike triangles corresponding to sawteeth of $\Gamma$, $x$ does not lie in the boundary of the domain of dependence $\mathcal D(\Gamma)$. Therefore, the lightlike part of $(S_K)_\infty$ coincides precisely with the union of all maximal past-directed lightlike triangles.

In particular, the spacelike part of $(S_K)_\infty$ in nonempty unless $\Gamma$ is a 1-step or a 2-step curve.

\emph{Step 3: The spacelike part of $(S_K)_\infty$ is a smooth $K$-surface.}
Let $x$ be any point of $(S_K)_\infty$, in the complement of the union of all maximal lightlike triangles, and let $x_n\in (S_K)_n$ such that $x_n\to x$. 
Regarding $(S_K)_n$ as the graph of a Lipschitz function in the conformal model of \eqref{eq:battista}, we argue that there is a cylindrical neighborhood $U$ of $x$ such that the restriction of the projection $\pi:\AdS^3\to \D^2$ is uniformly bi-Lipschitz on $U\cap (S_K)_n$. We deduce that there exists 
$\epsilon$ such that the metric ball $B_{(S_K)_n}(x,\epsilon)$ is compact in $(S_K)_n$, contained in $U$, and thus its projection is uniformly bounded in $\D^2$.
 Recall that the induced metric on $(S_K)_n$ has constant curvature $K$, and thus, up to rescaling the metric by a factor $|K|$, it is isometric to a subset of $\Hyp^2$. Let $o\in\Hyp^2$ and let
$$\sigma_n=:B_{\Hyp^2}(o,\epsilon)\to S_n$$
be an isometric embedding (after rescaling the metric by $|K|$) with $\sigma_n(o)=x_n$. Composing with the projection $\pi$, it then turns out that $\pi\circ\sigma_n$ are uniformly bounded and uniformly bi-Lipschitz. 
Hence the maps $\sigma_n$ converge $C^0$ up to a subsequence. The limit map $\sigma_\infty$ has image $B_{(S_K)_\infty}(x,\epsilon)$. 
 
 Observe that the 1-jet of $\sigma_n$ at $0$ is bounded. Indeed, if the 1-jet were not bounded, there would be a subsequence $n_k$ such that the tangent planes $T_{x_{n_k}}S_{n_k}$ converge to a lightlike plane, which is a support plane for $(S_K)_\infty$ at $x$. Hence $x$ is a point in the boundary of the domain of dependence $\mathcal{D}(\Gamma)$, contradicting the previous claim.

We can thus extract a subsequence in such a way that $j^1\sigma_n(0)$ converges. Since $\sigma_n$ are isometric embeddings of $B_{\Hyp^2}(o,\epsilon)$, (up to rescaling by the constant factor $|K|$), the pull-back $\sigma_n^*(g_{\AdS^3})$ of the Anti-de Sitter metric is constant, and we are in the right assumptions to apply Theorem \ref{schl_degeneration}. It follows that $\sigma_n$ converges $C^\infty$ in a neighborhood of $0$. Indeed, if this were not the case, by Theorem \ref{schl_degeneration} the surfaces $(S_K)_n$ would converge to the union of two half-planes of lightlike type in $\AdS^3$, intersecting along a spacelike geodesic, and this possibility is ruled out again by the fact that $x$ is in the spacelike part of $(S_K)_\infty$.
This concludes the proof that $\sigma_n$ converges $C^\infty$, and therefore the spacelike part of $(S_K)_\infty$ is smooth and has curvature $K$ at every point $x$.
\end{proof} 

\section{Foliations of the complement of the convex hull} \label{sec foliations}

The purpose of this section is to prove that the $K$-surfaces obtained in Theorem \ref{prop existence} provide foliations of $\mathcal D_+(\Gamma)$ and $\mathcal D_-(\Gamma)$, as $K\in(-\infty,-1)$.
We first compute the curvature of a spacelike surface in $\AdS^3$, expressed as a graph over the horizontal plane in an affine chart for the projective model of Subsection \ref{subsec models}.

\begin{prop} \label{prop formula affine curvature}
Let $V$ be an open subset of a smooth spacelike surface in $\AdS^3$. Suppose that, in an affine chart of the projective model for which $\AdS^3$ is the interior of the standard one-sheeted hyperboloid, $V$ is expressed as the graph of a function $u:\R^2\to\R$, namely:
$$V=\{(x,y,t):(x,y)\in\Omega,\,t=u(x,y)\}\,,$$
where $\Omega\subset\R^2$ is some domain. Then $u$ satisfies the equation:
\begin{equation} \label{monge ampere graph}
\det D^2 u=\frac{(1-|Du|^2+((x,y)\cdot Du-u)^2)^2}{(1+u^2-x^2-y^2)^2}\det B\,,
\end{equation}
where $B$ is the shape operator of $S$.
\end{prop}

\begin{proof}
Recall that in the projective description of $\AdS^3$ of Subsection \ref{subsec models}, the double cover $\SL(2,\R)$ of $\PSL(2,\R)$ is identified to the subset of $\R^4$ where the quadratic form $x_1^2+x_2^2-x_3^2-x_4^2$ takes the value $-1$. Hence we will denote by $\R^{2,2}$ the ambient $\R^4$ with the standard bilinear form of signature $(2,2)$, and 
we will suppose  that the affine chart is given by $x_4\neq 0$. We use the affine coordinates:
$$[x_1:x_2:x_3:x_4]\longrightarrow (x,y,t)=\left(\frac{x_1}{x_4},\frac{x_2}{x_4},\frac{x_3}{x_4}\right)\,,$$
so that the horizontal hyperplane $H_0$ corresponds to $x_3=0$. 

Denote by $H$ the hyperplane defined by $x_4=1$, and by  $H^{-}$ the region of $H$ where the quadratic form is negative, that is the image
of $\AdS^3$ in the affine chart.
We denote by $\xi: H^{-}\to \SL(2,\R)$ the radial map $\xi(x_1,x_2,x_3,1)=f(x_1,x_2,x_3)\cdot (x_1,x_2,x_3,1)$ where $f(x_1, x_2,x_3)=(1+x_3^2-x_1^2-x_2^2)^{-1/2}$.
The local parameterization of $S$ in the affine chart is of the form
\[
     \bar\sigma(x,y)=(x,y,u(x,y),1)\,,
\]
while the parameterization of the lifting of the surface in $\SL(2,\R)$ is $\sigma=\xi\circ\bar\sigma=f(\bar\sigma)\cdot\bar\sigma$.
We want to compute the determinant of the shape operator $B$ of $\sigma$. If $I$ e $\II$ denote the first and second fundamental form matrices of
$\sigma$ in the $\{\partial_x,\partial_y\}$ frame, it turns out that
\[
     \det B=(\det\II)/(\det I)~.
\] 
First let us compute $\det I$. 
If $\boxtimes$ denotes the vector product on $\SL(2,\R)$, then $\sigma_x\boxtimes\sigma_y$ is a tangent vector at $\sigma$  normal to immersion.
In particulat it is timelike and by standard facts
\[
   \det I=-\langle\sigma_x\boxtimes\sigma_y, \sigma_x\boxtimes\sigma_y\rangle~.
\]
Considering $\SL(2,\R)$ as a submanifold of $\R^{2,2}$ we can orient $\R^{2,2}$ so that the positive normal vector at $\eta\in\SL(2,\R)$ is $N(\eta)=\eta$.
Now it is immediate to check that the following formula holds for any $v,w\in T_\eta\SL(2,\R)$
\[
    v\boxtimes w=*(\eta\wedge v\wedge w)\,,
\]
where $*:\Lambda^3(\R^4)\to\R^4$ is the Hodge-operator on $\R^{2,2}$.
So a direct computation shows that
\[
   \sigma_x\boxtimes\sigma_y=*(f\bar\sigma\wedge d\xi(\bar\sigma_x)\wedge d\xi(\bar\sigma_y))=f^3\,*(\bar\sigma\wedge\bar\sigma_x\wedge\bar\sigma_y)
\]
 Now $\bar\sigma=\bar\sigma_0+e_4$, where $e_4$ is the positive normal of the hyperplane $H$, and 
 $$\bar\sigma_0(x,y)=(x,y,u(x,y),0)\,.$$
Denoting with $\boxtimes$ also the vector product
 on $H$ (that is intrinsically a copy of Minkowski space),
 \[
    *(e_4\wedge\bar\sigma_x\wedge\bar\sigma_y)=(\bar\sigma_x\boxtimes\bar\sigma_y)=\nu\,,
 \]
 where it turns out that $\nu=(u_x,u_y,1,0)$.
Writing $\bar\sigma_0=\bar\sigma-e_4$, one obtains:
\[
    *(\bar\sigma_0\wedge\bar\sigma_x\wedge\bar\sigma_y)=\langle \sigma_0, \nu\rangle e_4=\langle \sigma, \nu\rangle e_4\,.
\]
We thus conclude that
\begin{equation} \label{eq:hinwil}
\sigma_x\boxtimes\sigma_y=f^3(\nu+\langle\bar\sigma,\nu\rangle e_4)
\end{equation}
and thus $\det I=f^6g$ where $g(x,y)=\langle\bar\sigma,\nu\rangle^2-\langle\nu,\nu\rangle=1-|Du|^2+((x,y)\cdot Du-u)^2$.

To compute $\II$, notice that the normal vector of the immersion $\sigma$ is
\[
  N=\frac{\sigma_x\boxtimes\sigma_y}{\sqrt{\det I}}\in T_\sigma\SL(2,\R)~.
\]
Now $\II_{11}=\langle N, \sigma_{xx}\rangle$.
Observe that $\sigma_{xx}=f_{xx}\bar\sigma+2f_x\bar\sigma_x+f\bar\sigma_{xx}$. 
Using that $\sigma_x=f_x\bar\sigma+f\bar\sigma_x$ and that $N$ is orthogonal to $\sigma_x$ and to $\bar\sigma$, we get
\[
\II_{11}=f \langle N, \bar\sigma_{xx}\rangle=\frac{f}{\sqrt{\det I}}\langle \sigma_x\boxtimes\sigma_y, \bar\sigma_{xx}\rangle=\frac{f^4}{\sqrt{\det I}}\langle \bar\sigma_x\boxtimes\bar\sigma_y, \bar\sigma_{xx}\rangle\,,
\]
where in the last equality we have used Equation \eqref{eq:hinwil} and the fact that $\langle e_4,\bar\sigma_{xx}\rangle=0$, since $\bar\sigma_{xx}=(0,0,u_{xx},0)$. Hence, observing that $\langle\nu,\bar\sigma_{xx}\rangle=-u_{xx}$, one finally has:
$$\II_{11}
=-\frac{f^4}{\sqrt{\det I}}u_{xx}$$
With analogous computations we get
\[
\II=-\frac{f^4}{\sqrt{\det I}}D^2u\,,
\]
and therefore $\det\II=\frac{f^8}{\det I}\det D^2 u$. This concludes the proof that
\[
\det B=\frac{f^8}{(\det I)^2}\det D^2 u=f^{-4}g^{-2}\det D^2 u~.
\]
which corresponds to Equation \eqref{monge ampere graph}.
\end{proof}


\begin{remark} \label{remark degeneration}
From the proof of Proposition \ref{prop formula affine curvature}, the factor $(1-|Du|^2+((x,y)\cdot Du-u)^2)$ does not vanish as long as $S$ is a spacelike surface, whereas the factor $(1+u^2-x^2-y^2)$ vanishes precisely when $(x,y,u(x,y))$ is in $\partial\AdS^3$.
\end{remark}

We will now introduce a notion of \emph{locally pleated} surface. Recall that the definition of locally convex nowhere timelike surface was given in Definition \ref{defi locally convex}.

\begin{defi}
A locally convex nowhere timelike surface $S$ in $\AdS^3$ is \emph{locally pleated} if for every $\gamma\in S$ there exists a neighborhood $U$ of $\gamma$ in $\AdS^3$ such that $S\cap U$ is contained in the boundary of the convex hull of $S\cap \partial U$.
\end{defi}

The condition that $S$ is a locally pleated surface can be expressed by saying that no point of $S$ is a vertex (or extremal point) in the sense of convex geometry (\cite{rockafellar}). Using this characterization, we prove the following lemma.

\begin{lemma} \label{lemma boundary convex hull}
Let $S$ be a locally pleated surface in $\AdS^3$, with $\partial S=\Gamma$ a weakly acausal curve in $\partial\AdS^3$. 
Then $S$ is a pleated surface, that is, $S$ coincides with a boundary component of the convex hull of $\Gamma$.
\end{lemma}
\begin{proof}
Let $K$ be the convex hull of $S$. By convexity, $S$ lies in $\partial K$. By Krein-Milman Theorem (\cite{kreinmilman}), $K$ is the convex hull of extreme points of $S\cup \partial S$. As $S$ contains no extreme points from the above definition, we conclude that $K$ is the convex hull of $\partial S$. 
\end{proof}

Recall that, given a convex function $u:\Omega\subset\R^2\to\R$, the \emph{Monge-Ampère measure} $M\!A_u$ of $u$ is  defined in such a way that, if $\omega\subseteq \Omega$ is a Borel subset, then $M\!A_u(\omega)$ is the Lebesgue measure of the union of the subdifferentials of $u$ over points in $\omega$. In particular, if $u$ is $C^2$, then
$$M\!A_u(\omega)=\int_\omega (\det D^2 u)d\mathcal L\,.$$

We will use the following property of Monge-Ampère measures:

\begin{lemma}[{\cite[Lemma 2.2]{trudwang}}] \label{convergence of solutions}
Given a sequence of convex functions $u_n$ which converges uniformly on compact sets to $u$, the Monge-Amp\`ere measure $M\!A_{u_n}$ converges weakly to $M\!A_{u}$. 
\end{lemma}

Moreover, the following is a characterization of functions with vanishing Monge-Ampère measure:
\begin{theorem}[{\cite[Theorem 1.5.2]{gutierrez}}] \label{thm gutierrez}
Given a convex function $u:\Omega\subset\R^2\to\R$, $u\in C^0(\overline\Omega)$, where $\Omega$ is a convex bounded domain, if $M\!A_u\equiv 0$, then $u$ is the convex envelope of $u|_{\partial\Omega}$. 
\end{theorem}

We are now ready to prove that the $K$-surfaces with prescribed boundary a weakly spacelike curve $\Gamma\subset\partial\AdS^3$, whose existence was proved in Theorem \ref{prop existence}, provide a foliation of  $\mathcal{D}_+(\Gamma)$ and $\mathcal{D}_-(\Gamma)$.

\begin{theorem} \label{thm foliation part ads}
Given any weakly acausal curve $\Gamma$ in $\partial\AdS^3$ which is not a 1-step or a 2-step curve, $\mathcal{D}_+(\Gamma)$ and $\mathcal{D}_-(\Gamma)$ are foliated by $K$-surfaces $S^\pm_K$, as $K\in(-\infty,-1)$, in such a way that if $K_1<K_2$, then $S^\pm_{K_2}$ is in the convex side of $S^\pm_{K_1}$.
\end{theorem}

\begin{proof}
For definiteness, we give the proof for the region $\mathcal{D}_+(\Gamma)$, as the other case is completely analogous.

As in the proof of Theorem \ref{prop existence}, let $\Gamma_n$ be a sequence of curves in $\partial\AdS^3$ invariant for a pair $((\rho_l)_n,(\rho_r)_n)$ of torsion-free cocompact Fuchsian representations of $\pi_1(\Sigma_{g_n})\to\isom(\Hyp^2)\times \isom(\Hyp^2)$.
Recall that in Theorem \ref{prop existence} we showed that for every $K\in(-\infty,-1)$ there exists a subsequence on which the $K$-surfaces $(S_K)_{n}$ invariant for $((\rho_l)_n,(\rho_r)_n)$ converge to a convex  nowhere timelike surface $(S_K)_\infty$ with $\partial (S_K)_\infty=\Gamma$, whose lightlike part is the union of all maximal past-directed lightlike triangles bounded by $\Gamma$, and whose spacelike part is a smooth $K$-surface. 

By choosing a numeration of $\Q\cap(-\infty,-1)$, by a classical diagonal argument we can extract a subsequence $n_k$ such that for every $q\in\Q$, as $k\to\infty$, $(S_q)_{n_k}$ converges uniformly on compact sets to a convex surface, whose spacelike part is a smooth $q$-surface, which we denote (with a slight abuse of notation) by $S_q$. Let us denote again by $(S_q)_n$ the converging subsequence, omitting the subindex $k$. We will prove that the family of surfaces $\{S_q\}$ extend to a foliation of $\mathcal D_+(\Gamma)$ by $K$-surfaces. 

\emph{Step 1: The spacelike parts of the surfaces $S_q$ are pairwise disjoint.} Indeed, for every $q_1<q_2$, the surface $(S_{q_2})_n$ lies in the convex side of $(S_{q_1})_n$. Therefore the same holds for $S_{q_2}$ and $S_{q_1}$, and in particular they do not intersect transversely. Moreover, it is not possible that $S_{q_2}$ and $S_{q_1}$ are tangent at one point in the spacelike part, since the determinant of the shape operator of $S_{q_2}$ is strictly smaller than the determinant of the shape operator of $S_{q_1}$. Therefore if they were tangent, $S_{q_2}$ could not be in the convex side of $S_{q_1}$, since this would contradict the maximum principle.

\emph{Step 2: For every irrational $K<-1$ the sequences $\{S_{q'}:q'>K\}$ and  $\{S_{q''}:q''<K\}$ both converge to the same surface.} Let us denote
$$(S_K)'=\lim_{q'>K} S_{q'}\qquad\text{and}\qquad (S_K)''=\lim_{q''<K} S_{q''}\,,$$
where the limit clearly exists as the sequences are monotone. Moreover, since every $S_{q'}$ is in the convex side of every $S_{q''}$, the limit $(S_K)'$ is in the convex side of $(S_K)''$, and in particular they do not intersect transversely. Suppose by contradiction that there exists an open set $U$ which is contained in the region between $(S_K)'$ and $(S_K)''$. Now, since the surfaces $(S_K)_n$ foliate $\mathcal D_+(\Gamma_n)$ by Theorem \ref{bbz}, for large $n$ there is a surface $(S_{q_n})_n$ in the approximating sequence, for $q_n\in\Q\cap(-\infty,-1)$, which intersects $U$. By the uniform convergence on compact sets, for $n$ larger than some $n_0$, $q_n<q'$ for every $q'>K$. But for the same argument, for some possibly larger $n_0$, if $n\geq n_0$ then $q_n>q''$ for every $q''<K$. This gives a contradiction. 

\emph{Step 3: The spacelike part of the limit $S_K:=(S_K)'=(S_K)''$ is a smooth $K$-surface.} By the usual application of Theorem \ref{schl_degeneration}, for every $\gamma\in S_K$, pick a sequence $\gamma_{q'}\in S_{q'}$ converging to $\gamma$ and pick the homothetic embeddings
$$\sigma_{q'}:B_{\Hyp^2}(o,\epsilon)\to S_{q'}\,.$$ As in the proof of Theorem \ref{prop existence}, the $\sigma_{q'}$ converge $C^\infty$ and therefore $S_K$ is a smooth $K$-surface.


So far we have proved that $\{S_K:K\in(-\infty,-1)\}$ is a foliation by smooth $K$-surfaces of a subset of $\mathcal{D}_+(\Gamma)$. We finally need to prove that $\bigcup\{S_K:K\in(-\infty,-1)\}$ coincides with $\mathcal{D}_+(\Gamma)$. 

\emph{Step 4: The surfaces $S_K$ approach the spacelike part of $\partial_+\mathcal{C}(\Gamma)$ as $K\to-1$.} As $S_K$ is a monotone sequence in $K$, consider the $C^0$-limit
$$S_{-1}:=\lim_{K\to -1}S_K\,.$$

We claim that $S_{-1}$ is locally pleated. As the lightlike part of $S_{-1}$ 
is the union of lightlike triangles corresponding to sawteeth of $\Gamma$,
we only have to prove the claim for the spacelike part of $S_{-1}$.
Let $\gamma\in S_{-1}$ and let us choose an affine chart so that a support plane of $S_{-1}$ at $\gamma$ is the horizontal plane as in Proposition \ref{prop formula affine curvature}.

Hence, in a neighborhood of $\gamma$, the surfaces $S_K$ are the graph of a convex function $u_K:\Omega_K\to[0,+\infty)$. 
Using Remark \ref{remark degeneration}, as $K\to -1$ then $(1+u_K(x,y)^2-x^2-y^2)$ remains bounded for every $(x,y)\in\Omega_{-1}$ (since the point $(x,y,u_{-1}(x,y))$ is not in the boundary of $\AdS^3$). Also 
$(1-|Du_K(x,y)|^2+((x,y)\cdot Du_K(x,y)-u_K(x,y))^2)$
remains bounded for $(x,y)\in\Omega_{-1}$, as otherwise the limiting surface would have a lightlike support plane over $(x,y)$, and this is not possible for the usual argument, as $\gamma$ is in the spacelike part of $S_{-1}$.
Hence, applying Lemma \ref{convergence of solutions} to Equation \eqref{monge ampere graph}, we get that $u_{-1}$ satisfies $\det D^2 u_{-1}\equiv 0$ on $\Omega_{-1}$. Since $\gamma$ was arbitrary, by Theorem \ref{thm gutierrez}, $S_{-1}$ is locally pleated.  By Lemma \ref{lemma boundary convex hull}, $S_{-1}$ coincides with $\partial_+ \mathcal C(\Gamma)$ as claimed.

\emph{Step 5: The surfaces $S_K$ approach the boundary $\partial_+\mathcal{D}(\Gamma)$ as $K\to-\infty$.} As before, we can take the limit
$$S_{-\infty}=\lim_{K\to-\infty}S_K\,.$$
Suppose that $S_{-\infty}$ does not coincide with the boundary of $\mathcal{D}(\Gamma)$. Since $\partial_-\mathcal C(\Gamma)$ is the dual surface to $\partial_+\mathcal{D}(\Gamma)$, the dual surfaces $(S_K)^*$ would converge to $(S_{-\infty})^*$. But the shape operator of $(S_K)^*$ is identified to the inverse of the shape operator of $S_K$, and thus the $(S_K)^*$ are again surfaces of constant Gaussian curvature $K^*=-K/(K+1)$. In particular $K^*$ tends to $-1$ as $K\to-\infty$. Thus applying the same argument as in Step 4, one gets a contradiction.
\end{proof}




\section{Uniqueness and boundedness of principal curvatures} \label{sec uniqueness boundedness}

In this section we consider the problem of the uniqueness of the $K$-surfaces with fixed boundary curve $\Gamma$ (which we will be able to prove under the assumption that $\Gamma$ is the graph of a quasisymmetric homeomorphism) and the boundedness of principal curvatures. The two problems are tackled with similar techniques, and the key tool is the following compactness theorem for quasisymmetric homeomorphisms:

\begin{lemma}[\cite{basmajian}] \label{Compactness property of quasisymm homeo}
Let $\phi_n:\partial\Hyp^2\to \partial\Hyp^2$ be a family of orientation-preserving quasisymmetric homeomorphisms of $\partial \Hyp^2$. Suppose there exists a constant $M$ such that, for every symmetric quadruple $Q$ (i.e. such that the cross-ratio of $Q$ is $cr(Q)=-1$) and for every $n$,
\begin{equation} \label{cross-ratio boundedness}
|\log|cr\phi_n(Q)||\leq M\,.
\end{equation}
 Then there exists a subsequence $\phi_{n_k}$ for which one of the following holds:
\begin{itemize}
\item The homeomorphisms $\phi_{n_k}$ converge uniformly to a quasisymmetric homeomorphism $\phi:\partial\Hyp^2\to \partial\Hyp^2$;
\item The homeomorphisms $\phi_{n_k}$ converge uniformly on the complement of any open neighborhood of a point of $\partial\Hyp^2$ to a constant map $c:\partial\Hyp^2\to \partial\Hyp^2$.
\end{itemize}
\end{lemma}

We start with a remark, which shows us that the condition that $\Gamma=gr(\phi)$, with $\phi$ quasisymmetric, is not restrictive for the study of surfaces with bounded principal curvatures. That is, if $S$ is a surface with bounded principal curvatures and $\partial S=\Gamma$, then $\Gamma$ is necessarily the graph of a quasisymmetric homeomorphism. 

\begin{remark}
Suppose $S$ is a convex spacelike surface in $\AdS^3$ with $\partial S=\Gamma$, such that  the principal curvatures of $S$ are bounded. We claim that the associated map $\Phi$ is a bi-Lipschitz diffeomorphism of $\Hyp^2$. In particular $\Phi$ is quasiconformal and  $\Gamma$ is the graph of the quasisymmetric extension of $\Phi$.

By the same argument as in the proof of \cite[Proposition 6.21]{bebo}, the bound on the second fundamental form implies that $S$ is geodesically complete. Let us show that the projections $\pi_l,\pi_r$ are bi-Lipschitz. Using the completeness of $S$, this implies that they are coverings, hence diffeomorphisms. Therefore $\Phi$ is a bi-Lipschitz diffeomorphism of $\Hyp^2$.

The minimal and maximal stretch of $\pi_l$ (resp. $\pi_r$) are the eigenvalues of $E+JB$ (resp. $E-JB$). Since $B$ is bounded, the maximal stretch is bounded from above by some constant $M$. To obtain the bound from below, it suffices to notice that $\det(E\pm JB)=1+\det B>1$. Hence the minimal stretch is bounded from below by $1/M$. This proves that $\pi_l$ and $\pi_r$ are bi-Lipschitz.


\end{remark}

Another basic observation is the fact that the condition $\Gamma=gr(\phi)$, $\phi$ being quasisymmetric, is equivalent to the condition that the measured geodesic lamination on the upper (or on the lower) boundary of the convex hull of $\Gamma$ is bounded. Indeed, the latter condition is equivalent to saying that the Thurston norm of the left (or right) earthquake lamination is bounded, which is known to be equivalent to quasisymmetry of $\phi$, as proved independently in \cite{garhulakic} and \cite{Saric:2006ds} (see also \cite{thurstonearth,Saric:2008ek}). Boundedness of the measured geodesic lamination can be seen as a condition of boundedness of curvatures for the non-smooth surface $\partial_\pm\mathcal C(\Gamma)$. 

Theorem \ref{thm uniqueness and boundedness} gives a similar statement for $K$-surfaces, that is, a $K$-surface $S$ with $\partial S=\Gamma$ has bounded principal curvatures if and only if $\phi$ is quasisymmetric - and moreover it is unique. The proof of Theorem \ref{thm uniqueness and boundedness} is split in two propositions, which are proved with similar techniques. Proposition \ref{prop uniqueness quasisym} first proves the uniqueness part of Theorem \ref{thm uniqueness and boundedness}, while Proposition \ref{prop boundedness quasisym} proves boundedness of principal curvatures.

\begin{prop} \label{prop uniqueness quasisym}
Given any quasisymmetric homeomorphism $\phi:\partial\Hyp^2\to \partial\Hyp^2$, for every $K\in(-\infty,-1)$ there exists a unique (spacelike, smooth) future-convex $K$-surface $S^-_K$ and a unique past-convex $K$-surface $S_K^+$ with $\partial S^\pm_K=gr(\phi)$.
\end{prop}
\begin{proof}
Let us give the proof for past-convex surfaces, for definiteness. Let $\Gamma=gr(\phi)$. We have already proved the existence of a foliation by $K$-surfaces $S_K^+$ of $\mathcal D_+(\Gamma)$. Hence we have a function $\kappa:\mathcal D_+(\Gamma)\to(-\infty,-1)$, such that $\kappa(\gamma)=\kappa_0$ if $\gamma\in S_{\kappa_0}^+$. Let $S$ be another $K$-surface, for fixed $K$. We must show that $S$ coincides with $S_K^+=\kappa^{-1}(K)$.

Let $$\kappa_{\mathrm{max}}=\sup_{\gamma\in S} \kappa(\gamma)\qquad\text{and}\qquad \kappa_{\mathrm{min}}=\inf_{\gamma\in S} \kappa(\gamma)\,.$$
Let us show that $\kappa_{\mathrm{max}}\leq K$. Then by an analogous argument, we will have $\kappa_{\mathrm{min}}\geq K$, and thus $\kappa_{\mathrm{max}}=\kappa_{\mathrm{min}}=K$. 

Suppose the supremum $\kappa_{\mathrm{max}}$ is achieved at some point $\gamma_{\mathrm{max}}$, such that $\kappa_{\mathrm{max}}=\kappa(\gamma_{\mathrm{max}})$. First, it follows that $\kappa_{\mathrm{max}}>-\infty$, for otherwise $S$ would touch the boundary of the domain of dependence. Then in this case, $S$ and $S_{\kappa_{\mathrm{max}}}$ are tangent at $\gamma$, and $S_{\kappa_{\mathrm{max}}}$ is on the convex side of $S$. Hence by a standard application of the maximum principle, the determinant of the shape operator of $S_{\kappa_{\mathrm{max}}}$ at $\gamma$ is larger than the determinant of the shape operator of  $S$ (which equals $-1-K$). Therefore $K\geq \kappa_{\mathrm{max}}$.

On the other hand, suppose that $\kappa_{\mathrm{max}}=\lim_n \kappa(\gamma_n)$, where $\gamma_n\in S$ is a diverging sequence. Let $\chi_n\in \isom(\Hyp^2)$ such that $\chi_n(\gamma_n)=\mathcal I_x$ (for $x\in\Hyp^2$ a fixed point)  and $\chi_n(P_{\gamma_n}S)=\mathcal R_\pi$, where $P_{\gamma_n}S$ is the totally geodesic plane tangent to $S$ at $\gamma_n$. Observe that, if $\Gamma=gr(\phi)$ is the curve in $\partial\AdS^3$ of $S$, then $\chi_n(\Gamma)=gr(\phi_n)$, where $\phi_n$ is obtained by pre-composing and post-composing $\phi$ with isometries of $\Hyp^2$. Hence the condition of Equation \eqref{cross-ratio boundedness} in Lemma \ref{Compactness property of quasisymm homeo} is satisfied.

We claim that there can be no subsequence $\phi_{n_k}$ which converges to a constant map. In terms of Anti-de Sitter geometry, this means that the curve $\chi_{n_k}(\Gamma)$ would converge (in the Hausdorff convergence) to the boundary of a totally geodesic lightlike plane $P$. Let $\xi\in\partial\AdS^3$ be the point which determines the lightlike plane, namely the self-intersection point of $\partial P$ (see Figure \ref{fig:step}, left). Since $P_{\gamma_n}S$ is a support plane for $S$, $\partial P_{\gamma_n}S$ does not intersect $\Gamma$ for every $n$. By applying $\chi_n$, this means that $\partial \mathcal R_\pi$ does not intersect $\chi_n(\Gamma)$, and thus $\xi$ must necessarily be in $\partial \mathcal R_\pi$. See Figure \ref{fig:lightlikeplane}. But in this case, $(\mathcal I_x)^*$ does not contain $\xi$, so $(\mathcal I_x)^*$ meets $\chi_n(\Gamma)$ for $n$ large, which contradicts the fact that $\mathcal I_x$ is in the domain of dependence of $\chi_n(\Gamma)$ for every $n$.

\begin{figure}[htb]
\centering
\includegraphics[height=6.5cm]{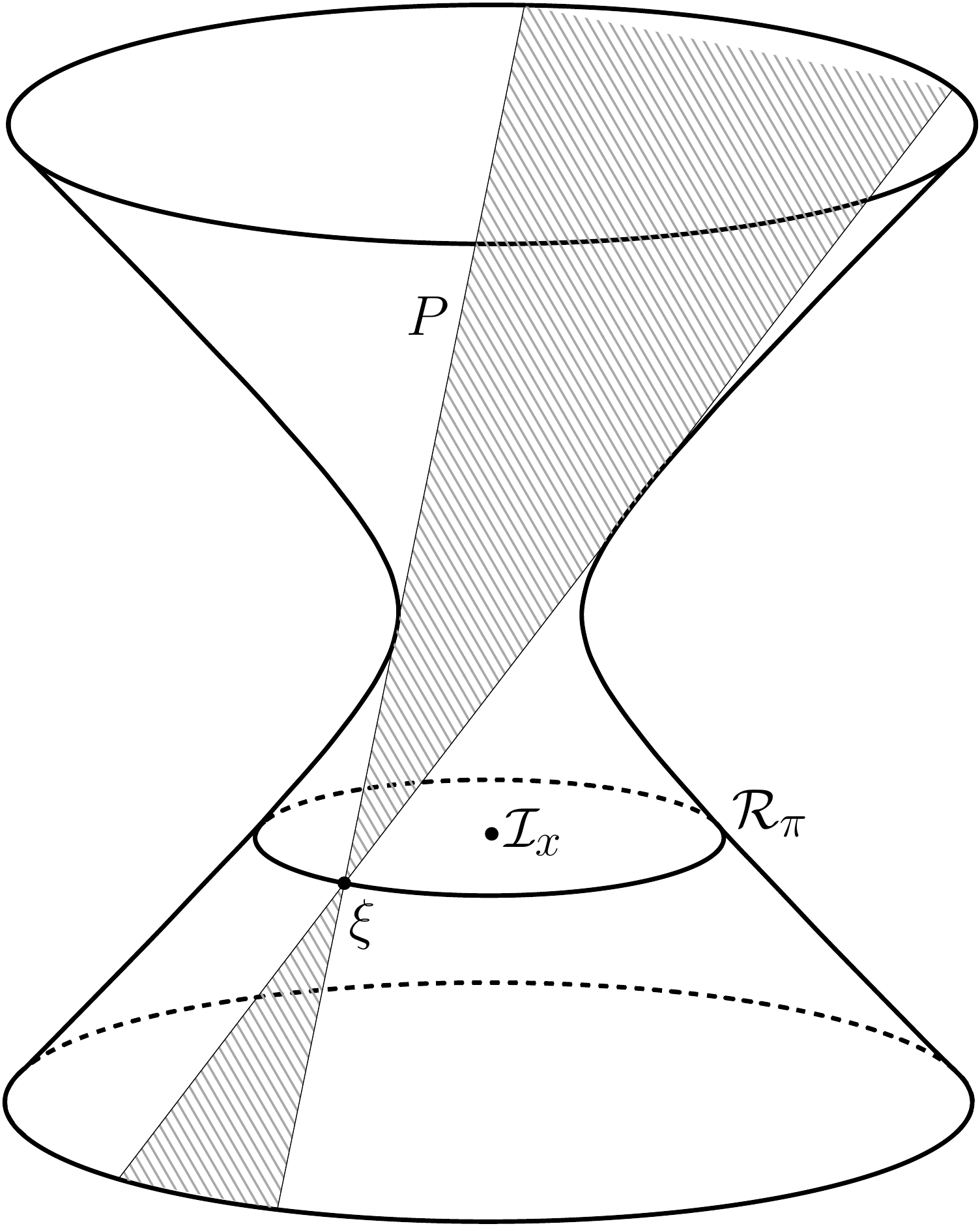}
\caption{The position of the lightlike plane $P$  of the claim in the proof of Proposition \ref{prop uniqueness quasisym}.  \label{fig:lightlikeplane}}
\end{figure}

Therefore, by Lemma \ref{Compactness property of quasisymm homeo}, there exists a subsequence $\phi_{n_k}$ which converges to a quasisymmetric homeomorphism $\phi_\infty$ of $\partial\Hyp^2$. Let $\Gamma_\infty=gr(\phi_\infty)$. By the same argument as the proofs of Theorem \ref{prop existence} and Theorem \ref{thm foliation part ads}, there exists a foliation of $\mathcal D_+(\Gamma_\infty)$ by $K$-surfaces which are obtained as limits (up to taking subsequences) of the surfaces $\chi_n(S^+_K)$. Moreover, taking a further subsequence, we can suppose $\chi_n(S)$ converges to a smooth $K$-surface $S_\infty$ contained in $\mathcal D_+(\Gamma_\infty)$ (as in the proof of Theorem \ref{prop existence}). Now we are again in the situation of the beginning of the proof, where $\kappa_{\mathrm{max}}=\sup_{\gamma\in S_\infty} \kappa_\infty(\gamma)$, and this supremum is achieved at $\mathcal I_x$. Thus applying again the above argument, the proof is concluded.
\end{proof}

We will now apply a similar argument to show that, under the assumption that $\Gamma$ is the graph of a quasisymmetric homeomorphism, the principal curvatures of the $K$-surfaces $S_K$ with $\partial S_K=\Gamma$ are uniformly bounded. 

\begin{prop} \label{prop boundedness quasisym}
Let $S$ be a $K$-surface, for $K\in(-\infty,-1)$, with $\partial S=gr(\phi)$, where $\phi:\partial\Hyp^2\to \partial\Hyp^2$ is a quasisymmetric homeomorphism. Then the principal curvatures of $S$ are bounded.
\end{prop}


\begin{proof}
By contradiction, assume that there exists a sequence of points $\gamma_n\in S$ where one principal curvature $\lambda_1(\gamma_n)$ tends to infinity (and thus the other tends to zero). 
As in the proof of Proposition \ref{prop uniqueness quasisym}, let $\chi_n\in \isom(\Hyp^2)$ such that $\chi_n(\gamma_n)=\mathcal I_x$ and $\chi_n(T_{\gamma_n}S)=\mathcal R_\pi$. By the same argument, the curves $\chi_n(\Gamma)$ converge to the graph of a quasisymmetric homeomorphism $\Gamma_\infty=gr(\phi_\infty)$, and as in Theorem \ref{prop existence}, the surfaces $\chi_n(S)$ converge to a smooth $K$-surface $S_\infty$ with $\partial S_\infty=\Gamma_\infty$. But then for the $C^\infty$ convergence, the principal curvatures $\lambda_1(\gamma_n)$, which are equal to the largest principal curvature of $\chi_n(S)$ at $\chi_n(\gamma_n)=\mathcal I_x$, converge to the largest principal curvature of $S_\infty$ at $\mathcal I_x$. Thus they cannot go to infinity and this gives a contradiction.
\end{proof}

We remark that the key points in the proofs of Proposition \ref{prop uniqueness quasisym} and Proposition \ref{prop boundedness quasisym} are the use of the compactness result for quasisymmetric homeomorphisms (Lemma \ref{Compactness property of quasisymm homeo}), and on the other hand, the convergence result of
Theorem \ref{schl_degeneration}, which essentially can be fruitfully applied (as for Theorem \ref{prop existence}) provided the curve $\Gamma$ is not a 1-step or a 2-step curve.

Proposition \ref{prop uniqueness quasisym} and Proposition \ref{prop boundedness quasisym} together give the statement of Theorem \ref{thm uniqueness and boundedness}.

Recall that the map $\Phi:\pi_r\circ(\pi_l)^{-1}$ associated to a past-convex $K$-surface $S$ is a $\theta$-landslide, for $K=-{1}/{\cos^2({\theta}/{2})}$, and that if $\partial S=gr(\phi)$ then $\Phi$ extends to $\phi$ on $\partial\Hyp^2$. Then we conclude by a corollary about the extensions by $\theta$-landslides of quasisymmetric homeomorphisms. 

\begin{cor}
Given any quasisymmetric homeomorphism $\phi:\partial\Hyp^2\to\partial\Hyp^2$ and any $\theta\in(0,\pi)$, there exist a unique $\theta$-landslides $\Phi_\theta:\Hyp^2\to\Hyp^2$ which extend $\phi$. Moreover, $\Phi_\theta$ is quasiconformal.
\end{cor}

In fact, Theorem \ref{prop existence} and Proposition \ref{prop uniqueness quasisym} prove existence and uniqueness of  a past-convex $K$-surface for every $K\in(-\infty,-1)$. Therefore the associated map is a $\theta$-landslide.
On the other hand, given any $\theta$-landslide $\Phi_\theta$, by the procedure described in Section \ref{sec representation formula}, in particular Corollary \ref{cor reconstruct ksurf}, $\Phi_\theta$ must necessarily come from one of the two $K$-surfaces with boundary $gr(\phi)$.


\bibliographystyle{alpha}
\bibliography{../bs-bibliography}

\end{document}